\documentclass[micropart]{gtpart}
\usepackage{amssymb,verbatim,graphicx,transparent}
\usepackage{graphics,epsfig,psfrag} 
\usepackage{pb-diagram,hyperref,graphics,epsfig,psfrag}
\usepackage[displaymath,mathlines,pagewise]{lineno}
\newcommand*\patchAmsMathEnvironmentForLineno[1]{%
  \expandafter\let\csname old#1\expandafter\endcsname\csname #1\endcsname 
  \expandafter\let\csname oldend#1\expandafter\endcsname\csname end#1\endcsname 
  \renewenvironment{#1}%
     {\linenomath\csname old#1\endcsname}%
     {\csname oldend#1\endcsname\endlinenomath}}% 
\newcommand*\patchBothAmsMathEnvironmentsForLineno[1]{%
  \patchAmsMathEnvironmentForLineno{#1}%
  \patchAmsMathEnvironmentForLineno{#1*}}%
\AtBeginDocument{%
\patchBothAmsMathEnvironmentsForLineno{equation}%
\patchBothAmsMathEnvironmentsForLineno{align}%
\patchBothAmsMathEnvironmentsForLineno{flalign}%
\patchBothAmsMathEnvironmentsForLineno{alignat}%
\patchBothAmsMathEnvironmentsForLineno{gather}%
\patchBothAmsMathEnvironmentsForLineno{multline}%
}
\usepackage{color}
\usepackage{braket}
\definecolor{darkred}{rgb}{0.5,0,0}
\definecolor{darkgreen}{rgb}{0,0.5,0}
\definecolor{darkblue}{rgb}{0,0,0.5}
\hypersetup{ colorlinks,
linkcolor=darkblue,
filecolor=darkgreen,
urlcolor=darkred,
citecolor=darkblue }

\renewcommand\theenumi{\alph{enumi}}
\renewcommand\theenumii{\roman{enumii}}
\newtheorem{theorem}{Theorem}[section]

\newtheorem{notation}{Notation}[section]

\newtheorem{corollary}[theorem]{Corollary}

\newtheorem{proposition}[theorem]{Proposition}
\newtheorem{lemma}[theorem]{Lemma}
\newtheorem{lem}[theorem]{}
\theoremstyle{definition}
\newtheorem{definition}[theorem]{Definition}
\theoremstyle{remark}
\newtheorem{remark}[theorem]{Remark}
\newtheorem{example}[theorem]{Example}
\newcommand{\blem}{\begin{lem} \rm}
\newcommand{\elem}{\end{lem}}

\newcommand\M{\mathcal{M}}

\renewcommand\S{\mathcal{S}}

\newcommand{\T}{\mathcal{T}}

\newcommand{\U}{\mathcal{U}}

\newcommand{\F}{\mathcal{F}}

\newcommand{\bR}{\mathbb{R}}
\renewcommand{\H}{\mathbb{H}}
\newcommand{\RR}{\mathcal{R}}

\newcommand{\bC}{\mathbb{C}}
\newcommand{\CC}{C\kern-1.3ex|}

\newcommand{\bZ}{\mathbb{Z}}
\newcommand{\bQ}{\mathbb{Q}}

\newcommand{\ddt}{\frac{d}{dt}}

\newcommand{\dds}{\frac{d}{ds}}

\renewcommand{\P}{\mathbb{P}}
\newcommand{\PP}{\mathcal{P}}
\newcommand{\on}{\operatorname}

\newcommand{\st}{\on{st}}
\newcommand{\supp}{\on{supp}}

\newcommand{\univ}{\on{univ}}

\newcommand{\ainfty}{{$A_\infty$\ }}

\newcommand{\Def}{\on{Def}}

\newcommand{\dist}{\on{dist}}

\newcommand{\dual}{\vee}

\newcommand{\JJ}{\mathcal{J}}
\renewcommand{\top}{{\on{top}}}

\newcommand{\Edge}{\on{Edge}}

\newcommand{\Lag}{\on{Lag}}

\newcommand{\loc}{{\on{loc}}}
\newcommand{\Ver}{\on{Vert}}

\newcommand\B{\mathcal{B}}
\newcommand\cU{\mathcal{U}}
\newcommand\cM{\mathcal{M}}

\newcommand{\aut}{ \on{aut} }

\newcommand{\Hol}{ \on{Hol} }

\newcommand{\Hom}{ \on{Hom}}
\newcommand{\Ind}{ \on{Ind}}
\renewcommand{\ker}{ \on{ker}}
\newcommand{\coker}{ \on{coker}}
\newcommand{\im}{ \on{im}}

\newcommand{\Vol}{  \on{Vol}}
\newcommand{\diag}{  \on{diag}}

\newcommand\dirac{/\kern-1.2ex\partial} % Dirac operator
\newcommand\qu{/\kern-.7ex/} % Categorical quotients
\newcommand\lqu{\backslash \kern-.7ex \backslash} % Categorical
\newcommand\dr{r_+ \kern-.7ex - \kern-.7ex r_-}
 
\newcommand{\labell}\label

\renewcommand{\d}{{\on{d}}}
\newcommand{\ol}{\overline}
\newcommand{\olp}{\ol{\partial}}

\newcommand\eps{\epsilon}

\newcommand{\f}{\frac}
\newcommand{\lan}{\langle}
\newcommand{\ran}{\rangle}
\newcommand{\hh}{{\f{1}{2}}}

\newcommand{\ti}{\tilde}

\newcommand\cE{\mathcal{E}}

\newcommand\cF{\mathcal{F}}

\newcommand\cI{\mathcal{I}}

\newcommand\mE{\mathcal{E}}

\newcommand\curv{\on{curv}}

\newcommand\Map{\on{Map}}

\newcommand\Vect{\on{Vect}}
\newcommand\ul{\underline}

\renewcommand\Im{\on{Im}}
\newcommand\Ker{\on{Ker}}
\newcommand\grad{\on{grad}}
\newcommand\reg{{\on{reg}}}
\newcommand\bdefn{\begin{definition}}
\newcommand\edefn{\end{definition}}
\newcommand\bea{\begin{eqnarray*}}
\newcommand\eea{\end{eqnarray*}}
\newcommand\bcv{\left[ \begin{array}{r} }
\newcommand\ecv{\end{array} \right] }

\newcommand\bma{\left[ \begin{array}{l} }
\newcommand\ema{\end{array} \right]}
\newcommand\ben{\begin{enumerate}}
\newcommand\een{\end{enumerate}}
\newcommand\beq{\begin{equation}}
\newcommand\eeq{\end{equation}}
\newcommand\bex{\begin{example}}
\newcommand\bsj{\left\{ \begin{array}{rrr} }
\newcommand\esj{\end{array} \right\}}

\newcommand\Cone{\on{Cone}}

\newcommand\DD{\mathbb{D}}

\newcommand\eex{\end{example}}
\newcommand\crit{{\on{crit}}}
  
\newcommand\val{{\on{val}}}  
\newcommand\sx{*\kern-.5ex_X}

\newcommand\white{{\includegraphics[width=.05in]{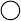}}}

\newcommand\black{{\includegraphics[width=.05in]{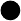}}}
\newcommand\whitet{\includegraphics[width=.07in]{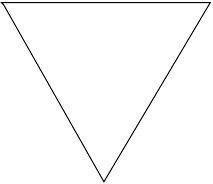}}
\newcommand\greyt{\includegraphics[width=.07in]{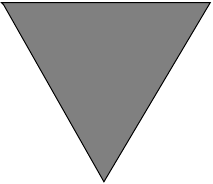}}

\newcommand{\zz}{v_+}
\newcommand{\bzz}{\ol{v}_+}
\newcommand{\mz}{v_-}
\newcommand{\bz}{\ol{v}_-}

\def\mathunderaccent#1{\let\theaccent#1\mathpalette\putaccentunder}
\def\putaccentunder#1#2{\oalign{$#1#2$\crcr\hidewidth \vbox
to.2ex{\hbox{$#1\theaccent{}$}\vss}\hidewidth}}

\begin{document}
\title[Immersed Floer theory and Maslov flows]{Invariance of immersed
  Floer cohomology \\ under Maslov flows}

\author{Joseph Palmer and Chris Woodward \\  with an erratum jointly with
  Hadi Azizi}

\begin{abstract} 
  We show that immersed Lagrangian Floer cohomology in compact
  rational symplectic manifolds is invariant under Maslov flow; this
  includes coupled mean curvature/K\"ahler-Ricci flow in the sense of
  Smoczyk \cite{smoczyk}.  In particular, we show invariance when a
  pair of self-intersection points is born or dies at a self-tangency,
  using results of Ekholm-Etnyre-Sullivan \cite{ees}.  Using this we
  prove a lower bound on the time for which the immersed Floer theory
  is invariant under the flow, if it exists.  This proves part of a
  conjecture of Joyce \cite{joyce:conjectures}.   An erratum written
  jointly with Hadi Azizi treates a missing case in the proof of a
  Lemma treating triple
  intersections, which are unavoidable in two dimensions.
\end{abstract}

\address{J. Palmer.   Mathematics - Altgeld Hall, University of Illinois at Urbana Champaign, 1409 W.~Green Street, Urbana, IL 61801, U.S.A.}
\email{ jpalmer5@illinois.edu }

\address{C. Woodward.  Mathematics-Hill Center,
Rutgers University, 110 Frelinghuysen Road, Piscataway, NJ 08854-8019,
U.S.A.}  \email{woodwardc@gmail.com}

\address{H. Azizi.  Mathematics - Bahen Centre, University of Toronto,
  40 St. George St., Toronto, ON, M5S 2E4, Canada}
\email{hadi.azizi@mail.utoronto.ca}

\maketitle

\tableofcontents

 \section{Introduction} 

 Lagrangian Floer theory is a cohomology theory whose differential
 counts pseudoholomorphic disks in a symplectic manifold with
 Lagrangian boundary conditions. For many purposes one wants to
 consider Floer theory of not only embedded but also {\em immersed}
 Lagrangians as introduced by Akaho-Joyce \cite{akaho}.  In Calabi-Yau
 manifolds, one expects special Lagrangians to play a special role as
 in the Thomas-Yau conjecture \cite{thomasyau}, namely, they should
 split-generate the Fukaya category.  In order to obtain special
 Lagrangians, one may hope to start with an arbitrary Lagrangian and
 minimize the volume by evolving under the mean curvature flow.
 Unfortunately, there is no reason to expect mean curvature flow of
 submanifolds of codimension greater than one to preserve
 embeddedness.  In addition, although short-time solutions to the mean
 curvature flow exist \cite{baker}, in general surgery is required to
 continue the flow beyond singular times, as in Colding-Minicozzi
 \cite{colding} and for the Lagrangian torus case Chen-Ma
 \cite{chenma}.

 From the point of view of mirror symmetry, mean curvature flow for
 Lagrangian branes is expected to be mirror to a (deformed version of)
 Yang-Mills flow for vector bundles or coherent sheaves.  Since
 Yang-Mills flow for vector bundles is achieved by a complex gauge
 transformation, one expects the isomorphism class of a Lagrangian
 brane to be invariant under the mean curvature flow.  In this paper,
 we prove several results in this direction.  We consider any flow
 whose deformation class at any time is equal to the Maslov class, and
 call such flows {\em Maslov flows}; a mean curvature flow, if it
 exists, is a special case. \label{masflow-comment} We prove
 invariance of Floer cohomology until the first singularity occurs or
 until the weakly bounding cochain hits a wall of zero valuation; in
 particular we show invariance of the Floer cohomology in the case
 that the (coupled, forward or reverse) mean curvature flow exists and
 passes through a self-tangency.  We have in mind an application to
 the existence of Floer non-trivial Lagrangians (and generators for
 the Fukaya category) in K\"ahler surfaces that we will discuss
 elsewhere.  In particular this proves ``step zero'' in the
 conjectures of Joyce \cite{joyceslides}, \cite{joyce:conjectures} for
 the existence of Lagrangian mean curvature flow with surgery for
 weakly unobstructed Lagrangians.
\begin{figure} 
\begin{center}
%% Creator: Inkscape inkscape 0.91, www.inkscape.org
%% Accompanies image file 'movie.pdf' (pdf, eps, ps)
\begingroup%
  \makeatletter%
  \providecommand\color[2][]{%
    \errmessage{(Inkscape) Color is used for the text in Inkscape, but the package 'color.sty' is not loaded}%
    \renewcommand\color[2][]{}%
  }%
  \providecommand\transparent[1]{%
    \errmessage{(Inkscape) Transparency is used (non-zero) for the text in Inkscape, but the package 'transparent.sty' is not loaded}%
    \renewcommand\transparent[1]{}%
  }%
  \providecommand\rotatebox[2]{#2}%
  \ifx\svgwidth\undefined%
    \setlength{\unitlength}{213.21216733bp}%
    \ifx\svgscale\undefined%
      \relax%
    \else%
      \setlength{\unitlength}{\unitlength * \real{\svgscale}}%
    \fi%
  \else%
    \setlength{\unitlength}{\svgwidth}%
  \fi%
  \global\let\svgwidth\undefined%
  \global\let\svgscale\undefined%
  \makeatother%
  \begin{picture}(1,0.52436037)%
    \put(0,0){\includegraphics[width=\unitlength,page=1]{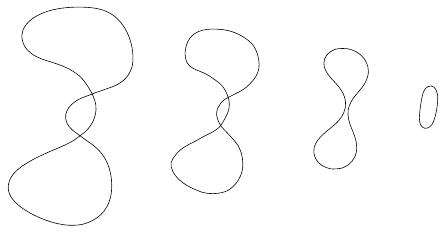}}%
%    \put(-0.08737601,0.43566475){\color[rgb]{0,0,0}\makebox(0,0)[lb]%{\smash{}}}%
%    \put(-0.02294217,0.25752417){\color[rgb]{0,0,0}\makebox(0,0)[lb]%{\smash{}}}%
 %   \put(0.12317101,0.4102703){\color[rgb]{0,0,0}\makebox(0,0)[lb]{\%smash{$A_1$}}}%
 %   \put(0.1436382,0.24899634){\color[rgb]{0,0,0}\makebox(0,0)[lb]{\%smash{$A_2$}}}%
 %   \put(0.12070735,0.11766488){\color[rgb]{0,0,0}\makebox(0,0)[lb]{%\smash{$A_3$}}}%
 %   \put(-0.22003387,0.27458018){\color[rgb]{0,0,0}\makebox(0,0)[lb]%{\smash{}}}%
  %  \put(0.00301501,0.28094087){\color[rgb]{0,0,0}\makebox(0,0)[lb]{%\smash{}}}%
 %   \put(0.01462137,0.2587858){\color[rgb]{0,0,0}\makebox(0,0)[lb]{\%smash{$A_0$}}}%
 %   \put(0,0){\includegraphics[width=\unitlength,page=2]{movie.pdf}}%
  \end{picture}%
\endgroup%
\end{center} 
\caption{An immersed curve flowing under the mean curvature flow.} 
\label{movie}
\end{figure} 
To explain the terminology, recall from Akaho-Joyce \cite{akaho} that
associated to a self-transverse immersion $\phi: L \to X$ of a compact
Lagrangian in a compact symplectic manifold $X$ is a {\em Fukaya
  \ainfty algebra} $CF(\phi)$ additively generated over a Novikov ring
by chains on the Lagrangian plus two copies of each self-intersection;
for each $d\geq 0$ the composition maps
\[\mu_d: CF(\phi)^{\otimes d} \to CF(\phi)\] 
count holomorphic disks with boundary in the Lagrangian. Via the 
homotopy units construction in \cite[(3.3.5.2)]{fooo}, \cite[Section 
2.2]{flips} one may furnish $CF(\phi)$ with a strict unit
$1_\phi \in CF(\phi)$. A {\em bounding cochain} (resp. {\em weakly
  bounding cochain}) is a solution $b \in CF(\phi)$ to the Maurer-Cartan equation
(resp. weak Maurer-Cartan equation)
\[ \mu_0(1) + \mu_1(b)  + \mu_2(b,b) + \ldots = 0 \ (\text{resp.} =
W(b) 1_\phi ).\]
The space of solutions to the weak Maurer-Cartan equation is denoted
$MC(\phi)$ and is equipped with a {\em potential function} 
\[ MC(\phi) \to \Lambda, \quad b \mapsto W(b) .\]
%
%The potential function is invariant under a natural notion of {\em
%  gauge equivalence}; we denote by $\ol{MC}(\phi) = MC(\phi)/\sim$ the
%set of gauge-equivalence classes. 
For any $b \in MC(\phi)$, the
operator
\[ \mu_1^b:CF(\phi) \to CF(\phi), \quad c \mapsto \sum_{k_- \ge 0,k_+
  \ge 0}
\mu_{k_- + k_+ +1}(\underbrace{b,\ldots, b}_{k_-}, c,\underbrace{ b,\ldots, b}_{k_+})
\]
squares to zero. 
%Its {\em Floer cohomology} $HF(\phi,b)$ depends up to
%isomorphism only on the class $[b] \in \ol{MC}(\phi)$ of $b$.
Non-vanishing of the Floer cohomology for some weakly bounding cochain
obstructs the Hamiltonian displaceability of the Lagrangian.  That is,
if there exists a Hamiltonian diffeomorphism $\psi: X \to X$ such that
$ \psi(\phi(L)) \cap \phi(L)$ is empty then the Floer cohomology
$HF(\phi,b)$ vanishes for any $b \in MC(\phi)$, or $MC(\phi)$ is empty.

We study the behavior of immersed Floer cohomology under flows that in
a cohomological sense are equivalent to forward or reverse mean
curvature flow coupled with K\"ahler-Ricci flow on the symplectic
manifold.  Consider a one-parameter family $\phi_t: L \to X$ of
Lagrangian immersions for a family of symplectic forms $\omega_t$ such
that the deformation class $ \Def(\phi_t) = [\dot{\phi}_t,
  -\dot{\omega}_t]$ as defined in Equation~\eqref{cartaneq} satisfies
\begin{equation} \label{defclass}
 \Def(\phi_t) = \pm c_1(\phi_t) \in H^2(\phi_t), 
\end{equation}
where $c_1(\phi_t)$ is the Maslov (relative Chern) class $c_1(\phi_t)$ in the relative de
Rham cohomology $H^2(\phi_t)$ defined in \eqref{maslovind} below.  We
call such a flow a {\em Maslov flow} (without surgery) as in
Lotay-Pacini \cite{lotay:coupled} (up to a sign).  In particular, this
implies that
\[ [ \dot{\omega_t} ] =  \pm c_1(X) \]
with $c_1(X) \in H^2(X)$ the first Chern class, as in a K\"ahler-Ricci
flow.  It is an observation of Smoczyk \cite{smoczyk} that the
combined mean curvature and K\"ahler-Ricci flow preserves the
Lagrangian condition.  A result of Lotay-Pacini \cite[Section
7]{lotay:coupled} (using a technique of Hamilton) says that such
coupled mean-curvature flows $\phi_t: L \to X, \omega_t$ exist for
short time.  All symplectic manifolds $X$ will be assumed to be
compact and Lagrangians $L$ will be assumed to be compact, oriented,
and equipped with relative spin structures and local systems.  In
order to apply Cieliebak-Mohnke perturbations \cite{cm:trans} we
furthermore assume that $X$ is simply-connected and the relative
symplectic class $[0,\omega] \in H^2(\phi,\bQ)$ is rational (which will
hold for rational times in the flow). \label{rattimes}

At times when the number of self-intersection points changes, we find
a correction to the weakly bounding cochain so that the Floer
cohomology and potential is preserved.  Suppose that under a flow
$\phi_t$ two new self-intersection points
$x_{1,t}, x_{2,t} \in \phi(L_t)$ are born at $t = 0$, leading to four
ordered self-intersection points
$v_{t,\pm}, \ol{v}_{t,\pm} \in L \times_{\phi_t} L$.\footnote{We
  use a bar over an ordered intersection point if the orientation
  induced by the splitting into tangent spaces of the branches of the
  Lagrangian disagrees with the symplectic orientation, and no bar
  otherwise.}  In our conventions, $\ol{v}_{t,\pm}$ is connected by a
Floer trajectory to $v_{t,\mp}$ as in
Figure \ref{small} by Theorem \ref{smallstrip}.
\label{mod2}

\begin{figure} 
\begin{center}
%% Creator: Inkscape inkscape 0.91, www.inkscape.org
%% Accompanies image file 'imex3.pdf' (pdf, eps, ps)
\begingroup%
  \makeatletter%
  \providecommand\color[2][]{%
    \errmessage{(Inkscape) Color is used for the text in Inkscape, but the package 'color.sty' is not loaded}%
    \renewcommand\color[2][]{}%
  }%
  \providecommand\transparent[1]{%
    \errmessage{(Inkscape) Transparency is used (non-zero) for the text in Inkscape, but the package 'transparent.sty' is not loaded}%
    \renewcommand\transparent[1]{}%
  }%
  \providecommand\rotatebox[2]{#2}%
  \ifx\svgwidth\undefined%
    \setlength{\unitlength}{64.31480227bp}%
    \ifx\svgscale\undefined%
      \relax%
    \else%
      \setlength{\unitlength}{\unitlength * \real{\svgscale}}%
    \fi%
  \else%
    \setlength{\unitlength}{\svgwidth}%
  \fi%
  \global\let\svgwidth\undefined%
  \global\let\svgscale\undefined%
  \makeatother%
  \begin{picture}(1,3.24090934)%
    \put(0,0){\includegraphics[width=\unitlength,page=1]{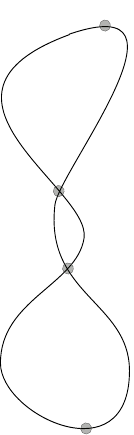}}%
    \put(0.68588224,3.16002918){\color[rgb]{0,0,0}\makebox(0,0)[lb]{\small \smash{$x$}}}%
    \put(0.37286981,1.92878027){\color[rgb]{0,0,0}\makebox(0,0)[lb]{\small \smash{$v_{t,+}$}}}%
    \put(0.39588543,1.05279238){\color[rgb]{0,0,0}\makebox(0,0)[lb]{\small \smash{$v_{t,-}$}}}%
    \put(0.53397911,0.10380579){\color[rgb]{0,0,0}\makebox(0,0)[lb]{\small \smash{$\ol{x}$}}}%
    \put(0.05544301,1.53534461){\color[rgb]{0,0,0}\makebox(0,0)[lb]{\small \smash{}}}%
    \put(0.02558908,1.48843143){\color[rgb]{0,0,0}\makebox(0,0)[lb]{\small \smash{$A_0$}}}%
    \put(0.30706891,0.66531598){\color[rgb]{0,0,0}\makebox(0,0)[lb]{\small \smash{$A_1$}}}%
    \put(0.41369007,1.49269626){\color[rgb]{0,0,0}\makebox(0,0)[lb]{\small \smash{$A_2$}}}%
    \put(0.3966307,2.58023192){\color[rgb]{0,0,0}\makebox(0,0)[lb]{\small \smash{$A_3$}}}%
    \put(0,0){\includegraphics[width=\unitlength,page=2]{imex3.pdf}}%
  \end{picture}%
\endgroup%

\end{center}
\caption{Small strips created by intersection points} 
\label{small}
\end{figure} 

\begin{theorem}
 {\rm (Invariance for self-tangencies)} 
\label{tanmain}
  Let $(\phi_t)_{t \in [-\eps,\eps]}$ be a (forward or reverse) Maslov
  flow \textup{(}as in Definition~\ref{def:maslovflow}\textup{)} of Lagrangian immersions developing a
  tangency $D\phi_t(T_{l_-} L) \cap D\phi_t(T_{l_+} L) \neq \{ 0 \}$
  at $t = 0$ with two additional self-intersections $x_{1,t}, x_{2,t}$
  for $t > 0$. For any family
\[ b_{t,-} \in MC(\phi_t), \quad  t \in [-\eps,0] \]  
of Maurer-Cartan solutions for $CF(\phi_t), t < 0$ with
$\val_q(b_0) > 0$, there exists a family 
\[ b_{t,+} \in MC(\phi_t), \quad t \in [0,\eps] \]
of Maurer-Cartan solutions for $CF(\phi_t)$ for $t \in (0,\eps)$ such
that
  \[ \lim_{t \to 0^+}   W(\phi_t,b_{t,+}) = \lim_{t \to 0^-}
  W(\phi_{t},b_{t,-})\]
  and there is an isomorphism of Floer cohomologies for
  $-\eps < t_- < 0 < t_+ < \eps$
  \[ HF(\phi_{t_-},b_{t_-,-}) \to HF(\phi_{t_+},b_{t_+,+}).\]
\end{theorem}

A very similar argument for the invariance of Legendrian contact
homology appears in Ekholm-Etnyre-Sullivan \cite{ees}, and our proof
uses the same basic idea although its realization in the language of
Lagrangian Floer theory is somewhat different.  In particular
Ekholm-Etnyre-Sullivan \cite[Lemma 4.6]{ees} describes the behavior of
holomorphic disks with boundary in a Lagrangian immersion under the
development of a self-tangency.  As in that paper, we work in a local
model in which the self-tangency is ``of standard form'', justified in
Lemma \ref{admissible}.  We consider families
\[ b_{t,\pm} \in CF(\phi_t), t \in [\pm \eps,0] \]  
which on either side of the tangency are connected by a correction
formula
\begin{equation} \label{informal} b_{t,+} = E_{t_-}^{t_+} b_{t,- } +
  \sum_{\pm} q^{-A(t)} \# \M(\phi_{t_+},b_{t,-};v_{t,\pm})
  \ol{v}_{t,\mp} \end{equation}
where $E_{t_-}^{t_+}$ is an operator acting by multiplication by some
power of $q$ in each graded piece (see Definition \ref{ongen}) \label{punc} $A(t)$
is the area of the small strip connecting $v_{t,\mp}$ with $v_{t,\pm}$
and, as usual, the count is a signed count weighted by the areas of
the disks.  In the case that the immersion stays self-transverse, we
identify the Maurer-Cartan solution spaces and the Floer cohomology
for small times.

\begin{theorem} 
  {\rm (Invariance for self-transverse families)} \label{transmain}
  Let $(X,\omega_0)$ be a compact, rational, simply-connected
  symplectic manifold as above and
  \[\omega_t \in \Omega^2(X), \quad \phi_t:L \to X, \quad  t \in [0,T]\] 
  a Maslov flow of symplectic forms and self-transverse Lagrangian
  immersions.  Given 
%$[b_0] \in \ol{MC}(\phi_0)$ with
  $b_0 \in {MC}(\phi_0)$ with $\val_q(b_0) > (\dim(L) - 1)T$ there
  exists $b_T \in {MC}(\phi_T)$ such that
 \begin{equation} \label{WHF} W(\phi_T,b_T) = q^{2T}W(\phi_0,b_0),  \quad HF(\phi_T,b_T) \cong
  HF(\phi_0,b_0) \end{equation}
where $\cong$ denotes group isomorphism.  Similarly, given
$b_T \in {MC}(\phi_T)$ with $ \val_q(b_T)> 2T$ there exists
$b_0 \in {MC}(\phi_0)$ such that \eqref{WHF} holds.
\end{theorem}

In the monotone case, invariance of Floer cohomology in this setting
is a result of Alston-Bao \cite{alston}.  The reason for the asymmetry
in the two cases in the Theorem (that is, the difference between the
bounds $(\dim(L) - 1)T$ versus $2T$) lies in the fact that the
necessary modification changes the weakly bounding cochain in degree
$d$ by a multiple of $\pm (d-1)T$; that is, there is a symmetry of the
factors around degree $1$ weakly bounding cochains which roughly
correspond to local systems and so require no modification under the
flow.

Combining Theorems \ref{tanmain} and \ref{transmain} above, consider a
situation in which a Maslov flow of Lagrangian immersions
$\phi_t: L \to X$ exists on some time interval $[0,T]$.  After
perturbation (see Lemma \ref{admissible} below) we may assume that
$\phi_t$ is self-transverse except for a finite collection of times
$t_1,\ldots, t_k \in [0,T]$ at which self-tangencies occur. Let
\[ A_i = \val_q( b_{t_i,+} - b_{t_i,-}) \] 
be the $q$-valuation of the correction term $b_{t_i,+} - b_{t_i,-}$ in
\eqref{informal} at the self-tangency $t_i$, for $i = 1,\ldots, k$ and
define $A_0 = \val_q(b_0)$.

\begin{theorem} \label{main}
 Let $(X,\omega_0)$ be a compact,
  rational, simply-connected symplectic manifold as above and
  \[\omega_t \in \Omega^2(X), \quad \phi_t:L \to X, \quad  t \in [0,T]\] 
  a Maslov flow of symplectic forms and Lagrangian immersions that are 
  self-transverse for $t \in \{ 0, T \}$. 
 Given 
  $[b_0] \in \ol{MC}(\phi_0)$ with 
  \[ \min_{i=0}^k \left( A_i - (\dim(L) - 1)( T - t_i) \right) > 0 \]
resp. $[b_T] \in \ol{MC}(\phi_T)$ with 
\[ \min\left( \val_q(b_T)- 2T, \min_{i=1}^k A_i - 2(T - t_i)\right) >
0 \]
there exists $[b_T] \in \ol{MC}(\phi_T)$ resp. 
$[b_0] \in \ol{MC}(\phi_0)$ such that 
\[ W(\phi_T,b_T) = q^{2T}W(\phi_0,b_0) \quad HF(\phi_T,b_T) \cong 
HF(\phi_0,b_0)\]
where $\cong$ denotes group isomorphism. 
\end{theorem}

  \begin{example} \label{initun}
    The following example shows that Floer theory can be initially
    unobstructed but become obstructed under mean curvature flow.  Let
    the symplectic manifold $X$ be the two-sphere $S^2$, thought of as
    the one-point compactification of a plane $\bR^2$, equipped with a
    metric that is flat on a large open subset $U \subset \bR^2$
    containing the image $\phi(L)$ of the immersion
    $\phi: L \cong S^1 \to X$.
    A ``movie'' showing a family of circles $\phi_t: L \to X$ under
    the mean curvature flow $\dot{\phi_t} = - H_{\phi_t}$ is shown in
    Figure \ref{joycemovie}.\footnote{Produced using the
      curve-shortening software by A. Carapetis \cite{shorten} and
      Inkscape.}  Indeed, initially the area of the middle bigon is
    smaller than the areas of the teardrops; we will see when we
    analyze this example in more detail later (Example
    \ref{teardrops}) that there exists a weakly bounding cochain
    $b_t \in MC(\phi_t)$.  On the other hand, once the area of the
    bigon becomes larger than the combined area of the teardrops then no
    choice of weakly bounding cochain $b_t$ can cancel the teardrop
    contributions to
    $\mu_0(1) \in CF(\phi_t)$.
    By rescaling this example one sees that similarly, Floer
    cohomology $HF(\phi_t,b_t)$ can not be invariant under arbitrary
    exact isotopy, that is, deformations $\phi_t$ generated by exact
    one-forms $\alpha_t \in \Omega^1(L)$.  In higher dimension, we
    show in \cite{pw:surger} that one can continue the flow by
    performing a surgery at one of the intersection points, at least
    in high dimensions.  This ends the example.

\begin{figure} 
\begin{center}
\includegraphics{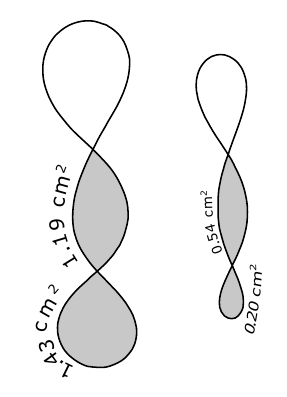}
\end{center}
\caption{A  flow in which the Floer cohomology becomes
  obstructed, with areas indicated}
\label{joycemovie}
\end{figure}

\end{example} 

Mean curvature flow for curves is also known as the {\em
  curve-shortening flow}. In this case, for embedded curves in the
plane $\phi_0: L \to X = \bR^2$ Grayson \cite{grayson} proved that the
curve $\phi_t(L)$ eventually becomes convex and then collapses to a
point in finite time.  For immersed curves $\phi_0: L \to X$
Angenent-Vel\'azquez \cite{angenent:cusp} studied the development of
cusps. Angenent \cite{angenent} studied further the curve-shortening
flow on arbitrary surfaces $X$.  Even for curves in the Euclidean
plane, Theorem \ref{main} includes a statement that is not obvious to
the authors:

\begin{corollary} \label{cusp} Suppose $\phi: S^1 \to \bR^2$ is an
  immersion such that the Maurer-Cartan moduli space $MC(\phi)$ is
  empty.  Then the curve-shortening flow of $\phi$ encounters a
  singularity, that is, does not flow to a round point in the sense of 
Grayson \cite{grayson}.
\end{corollary} 

Non-emptiness of the Maurer-Cartan space does not imply convergence to
a point, by Example \eqref{initun}. \label{addex}  Joyce's
conjecture \cite{joyce:conjectures} also makes sense for the
curve-shortening flow for immersed curves in the two-sphere $S^2$, and
is open even in that case.  Work of Angenent and others
\cite{angenent} describes the formation of singularities but note that
Joyce's suggested surgery should be performed {\em before} the
singularities arise, see Remark \ref{comms} and the sequel to this
paper \cite{pw:surger}.

The intended application of these results is to the Lagrangian minimal
model program as explained in Charest-Woodward \cite{flips}.  The
conjecture in \cite{flips} describes an orthogonal decomposition of
the split-closed derived Fukaya category $D^\pi \F(X)$ of a K\"ahler
manifold $X$ corresponding to singularities in the \label{thethe}
K\"ahler-Ricci flow-with-surgery; that is, a family of K\"ahler
manifolds $(X_t,\omega_t)$ satisfying $\dot{\omega_t} = -R_{\omega_t}$
except at finitely many times where flips, divisorial contractions, or
a fibration occur.  Floer non-trivial Lagrangians associated to the
singularities in the flow were described in \cite{flips}; the results
of this paper imply the invariance of the Floer cohomology as these
Lagrangians are ``flowed backwards'' under the K\"ahler-Ricci flow.
We address in \cite{pw:surger} the invariance of Floer cohomology
under simplest kind of surgery.  We close the paper with
Remark~\ref{comms}, which lists some related open questions.

\vskip .1in \noindent {\em This work was partially supported by NSF grant DMS 1711070. 
  Any opinions, findings, and conclusions or recommendations expressed 
  in this material are those of the author(s) and do not necessarily reflect the views of the National Science Foundation.  }

\section{Maslov flows of Lagrangian immersions}

Recall basic definitions regarding Lagrangian immersions following,
for example, Weinstein \cite{weinstein}.  Let $(X,\omega)$ be a
symplectic manifold.  A {\em Lagrangian immersion} is a smooth map
$\phi: L \to X$ from a smooth manifold $L$ to $X$ satisfying
\[\phi:L \to X, \quad \dim(L) = \dim(X)/2, \quad \phi^* \omega = 0 , \quad 
\ker(D_x \phi) =\{ 0 \}, \forall x \in L .\]
The set of ordered {\em self-intersection points} is
$$ \cI^{\on{si}}(\phi) = L \times_\phi L - \Delta_L $$ 
where $\Delta_L: L \to L \times L$ is the diagonal embedding.  A
Lagrangian immersion $\phi: L \to X$ from a compact manifold $L$ is a {\em Lagrangian embedding}
if $\phi$ is injective, or equivalently if $\cI^{\on{si}}(\phi)$ is empty.
Lagrangian immersions are closed under the {\em disjoint union}
construction: If $\phi_0:L_0 \to X$ and $\phi_1:L_1 \to X$ are
Lagrangian immersions then so is the disjoint union
\[\phi_0 \sqcup \phi_1:L_0 \sqcup L_1 \to X, \quad 
x \mapsto \begin{cases} \phi_0(x) & x \in L_0 \\ 
\phi_1(x) &  x \in L_1 \end{cases} .\]  
Locally any Lagrangian immersion is the disjoint union of Lagrangian
embeddings, so that the self-intersections of the union
$\phi_0 \sqcup \phi_1$ include the intersections between different
components $\phi_0(L_0), \phi_1(L_1)$.  Similarly, the product of
Lagrangian immersions $\phi_0: L_0 \to X_0 $ and $\phi_1:L_1 \to X$ is
a Lagrangian immersion to $X_0 \times X_1$.

A natural equivalence on the set of Lagrangian immersions is given by
{\em Lagrangian isotopy}.  Two Lagrangian immersions
$\phi_0,\phi_1:L \to X$ are {\em isotopic} if there exists a family
\[ \phi_t:L \to X, t \in [0,1] \]
of Lagrangian immersions connecting them, smooth in the parameter $t$.
The derivative of an isotopy of Lagrangian immersions
$\phi_t: L \to X$ at any time $t \in [0,1]$ may be identified with a
closed one-form on the domain as follows: The quotient
$\phi_t^* TX/TL$ is the normal bundle of the immersion.  Consider the
normal vector field
\[v_t: L \to \phi_t^* TX/TL, \quad v_t(l) = \ddt
\phi_t(l) \ \text{mod} \  D\phi_t(T_l L) .\]
The pairing $\iota(v_t) \phi_t^* \omega$ is well-defined since
$\phi_t$ is Lagrangian.  The derivative
of the pull-back of the symplectic form satisfies
\begin{equation} \label{cartan} \ddt (\phi_t^* \omega) = \d
  \dot{\phi_t} + \iota(v_t) \phi_t^* \d \omega \end{equation}
where 
\[ \dot{\phi_t} :=  
\omega(   D\phi_t ( \cdot), v_t) 
 \in \Omega^1(L) .\]
 This follows locally from the Cartan homotopy formula and the fact
 that in the embedded case
 $ \omega( D\phi_t ( \cdot), v_t) = \phi_t^* \iota(\ti{v}_t) \omega $
 where $\ti{v}_t$ is an extension of $v_t$ to a neighborhood of
 $\phi_t(L)$.  Since $\phi_t^* \omega$ and $\d \omega$ are both zero,
 $\dot{\phi_t}$ is closed for all $t \in [0,1]$. Its cohomology class
 is the {\em deformation class} of the isotopy
\[ \Def(\phi_t) =  [ \dot{\phi_t}] \in H^1(L) .\]
 If $\psi_t:X \to X, t \in [0,1]$ is a smooth family of
 symplectomorphisms then $\psi_t \circ \phi_0: L \to X$ provides an
 isotopy from $\phi_0$ to $\psi_1 \circ \phi_0$.  If
 $\psi_t: X \to X, \ t \in [0,T]$ is a family of Hamiltonian
 symplectomorphisms then we say that the immersions
 $\phi_t = \psi_t \circ \phi_0$ are {\em Hamiltonian isotopic}
 and we write $\phi_0 \sim \phi_1$.  In
 this case, the deformation class is trivial.  The Gromov-Lees
 h-principle \cite{lees}, \cite{gromov} classifies isotopy classes of
 Lagrangian immersions by the homotopy classes of their tangent maps
 $D\phi: TL \to TX$:
\[ \phi_0 \sim \phi_1 \iff D\phi_0 \simeq D \phi_1 .\]  
So in particular, immersions of a circle $S^1$ into the plane $\bR^2$
are classified up to isotopy by the winding numbers of their Gauss
map, while there are two isotopy classes of immersions of a circle
$S^1$ into the two-sphere $S^2$.

To study further Lagrangian isotopies of immersions we recall the
definition of the relative de Rham complex as in Bott-Tu \cite[Section
6]{bt}.  For $k \ge 0$ denote the space of relative forms
\[ \Omega^k(\phi) := \Omega^{k-1}(L) \oplus \Omega^k(X) .\]
Equip $\Omega^k(\phi)$ with the relative de Rham differential the
operator of order $1$ given by
\[ \d :\Omega(\phi) \to \Omega(\phi), \quad (\beta_L, \beta_X) \mapsto
(-\d \beta_L + \phi^* \beta_X, \d \beta_X) .\]
The relative de Rham cohomology is 
\[H(\phi) = \bigoplus_{k=0}^{m(X,L)} H^k(\phi), \quad
H^k(\phi) =  \frac{\ker(\d) \cap \Omega^k(\phi)}{\im(\d) \cap
\Omega^k(\phi)} \]
$m(X,L) = \max(\dim(X),\dim(L))$.  The relative cohomology fits into a
long exact sequence
\[ \ldots \to H^k(\phi) \to H^k(X) \to H^k(L) \to H^{k+1}(\phi) \to \ldots  .\]
Relative cocycles integrate naturally over any relative cycle.  Such a
cycle is a pair of maps from a manifold and its boundary:
\[u: S \to X, \quad \partial u: \partial S \to L, \quad \phi
\circ \partial u(z) = u(z), \forall z \in \partial S .\]
Integration of a cocycle $(\beta_L,\beta_X) \in \Omega^k(\phi)$ where
$k = \dim(S)$ is given by
\begin{equation} \label{relpair}
 \int_u (\beta_L,\beta_X) = \int_S u^*(\beta_X) - \int_{\partial S} (\partial u)^*(\beta_L)
 .\end{equation} 
By Stokes' theorem, the integral \eqref{relpair} depends only on the
cohomology class $[\beta_L,\beta_X] \in H^k(\phi)$.

In particular the relative symplectic form defines a map on relative
second homology of the immersion. Let $\phi: L \to X $ be a Lagrangian
immersion.  The pair $ (0,\omega) \in \Omega^2(\phi)$ is a relative de
Rham cocycle since 
$$ \phi^* \omega = 0, \quad \d \omega = 0 .$$
Denote the unit disk in the complex plane
\[ S = \{ z \in \bC \ | \ |z| \leq 1 \} .\]
A {\em disk with boundary $\phi$} is a pair consisting of a map
$u: S \to X$ and a lift of the restriction to the boundary
$u|_{\partial S}$ to a map to $L$:
\[u: S \to X, \quad  \partial u: \partial S \to L,  \quad  \phi
\circ \partial u = u |_{\partial S} .\]
%
%Denote by $\pi_2(\phi,x)$ or $\pi_2(\phi)$ for short the space of
%homotopy classes of relative disks with a given base point $x \in L$,
%that is, so that $\partial u$ maps $1$ to $x$.  The relative homotopy
%group fits into a long exact sequence
%%
%\[ \ldots \pi_2(L) \to \pi_2(X) \to \pi_2(\phi)  \to \pi_1(L)  \to \pi_1(X) \to 
%\ldots .\]
%%
%In most of our examples, $L$ will be a torus and $X$ is simply-connected, so that $\pi_2(L) = \pi_1(X) = \{ %1 \}$. Thus
%$\pi_2(\phi) \cong H_2(\phi)$ (where $H_2(\phi)$ denotes the homology
%f the mapping cone) and fits into a short exact sequence
%
%\[ \{ 0\}  \to H_2(X) \to H_2(\phi) \to H_1(L) \to \{ 0 \} .\]

The homomorphisms associated to the symplectic and Maslov classes are
denoted as follows.  Given a disk $u: (S,\partial S) \to (X,L)$ the
symplectic area $A(u)$ is the pairing of $[\omega]$ with the class of
$[S,\partial S]$ and induces a map
\[ A: H_2(\phi) \to \R, \quad [u] \mapsto \int_D u^* \omega .\]
The {\em Maslov (or relative Chern) index} measures the winding of the
Lagrangian boundary condition.  
To explain the definition
\eqref{maslovind} let $\dim(X) = 2n$ for some positive integer $n$.
Let $U(n)$ resp. $SO(n)$ denote the space of unitary resp. special
orthogonal $n \times n$ matrices and let
\[ \Lag(\bC^n) \cong U(n)/SO(n) \]
denote the Grassmannian of oriented Lagrangian subspaces of $\bC^n$.
Since the disk $S$ is contractible, there exists a trivialization of
symplectic vector bundles $u^* TX \cong S \times \bC^n$, unique up to
isomorphism.  Let
\[(\partial u)^* TL: \partial S \to \Lag(\bC^n)\]
denote the family of Lagrangian, or equivalently, totally real
subspaces on the boundary of the disk $\partial S$. The determinant
$\det: U(n) \to U(1)$ factors through $U(n)/SO(n)$ and induces a loop
whose winding number is by definition half the Maslov index:
\begin{equation} \label{maslovind} I: H_2(\phi) \to \bZ, \quad [u] \mapsto
  2 [ \det( (\partial u)^* TL)] \in \pi_1(S^1) \cong \bZ .\end{equation}
In \cite[Appendix C]{ms:jh} Robbin shows that the Maslov index for
bundles is the unique invariant that is additive under direct sum,
additive under sewing boundaries and suitably normalized for line
bundles over the disk.  It follows that the Maslov index map
$[u] \mapsto I([u])$ is the relative Chern class of $TX$ (as pointed
out by T. Perutz) and defines a cohomology class denoted
$c_1(\phi) \in H^2(\phi)$.  Alternatively, the explicit differential
form representing the Maslov class \cite{pacini} \label{expl} implies that the
Maslov class is a relative cohomology class.

In the case of isotopies in which the symplectic form varies, the 
deformation of the Lagrangian and symplectic form combine to a 
relative cocycle.  Let 
\[ \omega_t \in \Omega^2(X), t \in [0,1], \quad \text{resp.} \ \phi_t:
L \to X \] 
be a family of symplectic forms resp.  a family of Lagrangian
immersions for $\omega_t$.  The derivative of the pullback symplectic
form satisfies
\[ 0 = \ddt (\phi_t^* \omega_t) = \d \dot{\phi_t} + \phi_t^*
\dot{\omega_t} \]
by an argument similar to that of \eqref{cartan}.  The deformation
class of the isotopy is an element in the relative de Rham cohomology
\begin{equation} \label{cartaneq} \Def(\phi_t) = \left[ \dot{\phi_t},
    -\dot{\omega_t} \right] \in H^2(\phi_t) .\end{equation}
If the class \eqref{cartaneq} vanishes, then the isotopy is an {\em
  exact isotopy}.  The notions of immersed Lagrangian isotopy
resp. exact Lagrangian isotopy $\phi_0 \simeq \phi_T$ are easily shown
to be equivalence relations; concatenation of isotopies can be taken
to be smooth by deforming so that the isotopies are constant near the
starting and ending times $t = 0,T$.

Deformations of Lagrangian isotopies by the Maslov class arise from a
choice of connection one-form on the anti-canonical bundle.  Recall
that the tangent bundle $TX$ to the symplectic manifold $X$ has a
complex structure well-defined up to isotopy.  The top exterior power
of the dual $T^\dual X$ of $TX$
\[K^{-1} = \Lambda^{\top}_\C(T^\dual X) \]
is the {\em anticanonical bundle}.  By definition the first Chern
class of $K^{-1}$ is the first Chern class of $X$.  Assume that $L$ is
oriented by a section of the orientation bundle
\[ o_L :L \to \Lambda^{\top}_\R(T^\dual L ).\]
Since $TL$ is a totally real sub-bundle of $\phi^* TX $ one has an
isomorphism 
\[ \phi^* TX \cong TL \otimes_\bR |C .\]
This isomorphism of bundles induces an isomorphism of top exterior
powers
\begin{equation} \label{topext}
 \ \Lambda^{\top}_\R(T^\dual L) \otimes_\bR \C
\to \Lambda^{\top}_\C( \phi^* T^\dual X) .\end{equation}
In this way the orientation $o_L$ on $L$ induces a trivialization
\[ \phi^* K^{-1} \to L \times \bC \]  
of the pull-back of the anticanonical bundle $K^{-1}$.  Let $\alpha_X$
be a connection one-form on $K^{-1}$.  By definition, 
\[ \alpha_X \in \Omega(K^{-1}_1)^{S^1},  \quad \alpha_X( \partial/\partial
\theta) = 1 \] 
is an $S^1$-invariant one-form on the unit circle bundle
$K^{-1}_1 \subset K^{-1}$ with the property that \label{that} the contraction with the
rotational vector field $\partial/\partial \theta \in \Vect(K^{-1})$
is $1$.  The curvature is defined by
\begin{equation} \label{curva}
\curv(\alpha_X) \in \Omega^2(X), \quad  \pi^* \curv(\alpha_X) = \d
\alpha_X\end{equation} 
where $\pi:K^{-1}_1 \to X$ is the projection.

On the other hand, the given trivialization of $\phi^* K^{-1}$ defines
a flat connection \label{flatcon} we denote $\alpha_L$.  The difference between the two
connection one-forms $\phi^* \alpha_X$ and $\alpha_L$ is a one-form on
the base that we write $\phi^* \alpha_X - \alpha_L \in \Omega^1(L)$.
The derivative of the difference one-form is the difference between
the curvatures, which since $\alpha_L$ has trivial curvature is the
pull-back of the curvature on $X$:
\[ \d( \phi^* \alpha_X - \alpha_L) = \phi^* (\curv(\alpha_X)) .\]
We include the following result of Cieliebak-Goldstein \cite{cg:note},
see also Lotay-Pacini \cite[Section 7]{lotay:coupled} for convenience.

\begin{lemma}
 For any Lagrangian immersion $\phi: L \to X$, the pair 
\begin{equation} \label{thepair}
( \phi^* \alpha_X - \alpha_L, \curv(K^{-1},\alpha_X)) \in 
  \Omega^2(\phi)\end{equation}
is a relative cocycle representing the Maslov class.
\end{lemma} 

\begin{proof} The cocycle property is a consequence of the definition
  of the curvature \eqref{curva} and flatness of $\alpha_L$.  We show
  that the pair \eqref{thepair} represents the Maslov class.  Let
  $u: (S, \partial S) \to (X,L)$ be a relative cycle.  A
  trivialization $u^* K^{-1} \cong S \times \C$ induces a connection
  one-form denoted $\alpha_S$ on $u^* K^{-1}$.  The Maslov index
  $I(u)$ is the integral of the difference of this connection with the
  connection one-form $(\partial u)^*\alpha_L$ induced by the
  trivialization $(\partial u)^* K^{-1} \cong \partial S \times \C$:
  Let $\iota: \partial S \to S$ denote the inclusion.  The difference
  between any two connection one-forms
  $(\partial u)^*\alpha_L,\iota^* \alpha_S$ on $(\partial u)^* K^{-1}$
  is the pull-back of a one-form on the base $\partial S \cong S^1$,
  which we write (abusing notation) as $\alpha_L - \iota^* \alpha_S $.
  The Maslov index is 
\begin{eqnarray*} 
  \int_{\partial S} (\partial u)^*\alpha_L -  \iota^* \alpha_S &=&
                                                                   \int_{\partial
                                                                   S}
                                                                   (\partial
                                                                   u)^*\alpha_L
                                                                   -
                                                                   (\partial
                                                                   u)^*
                                                                   \phi^*
                                                                   \alpha_X
                                                                   +
                                                                   \int_{\partial
                                                                   S}
                                                                   (\partial
                                                                   u)^*
                                                                   \phi^*
                                                                   \alpha_X
                                                                   -
                                                                   \iota^*
                                                                   \alpha_S
  \\  &=&  -  \int_{\partial S} ((\partial u)^*   \phi^* \alpha_X -
          (\partial u)^*\alpha_L ) +  \int_S \curv(K^{-1},\alpha_X) \end{eqnarray*} 
        using Stokes' theorem and the fact that the curvature of
        $\alpha_S$ vanishes.  This gives a special case of the pairing
        in \eqref{relpair}. \end{proof}

                                               Maslov flows are
                                               isotopies of
                                               Lagrangians given by
                                               the above relative
                                               cocycle.  Let
                                               $\omega_t$ be \label{isbe} a family
                                               of symplectic forms and
                                               $\alpha_t$
                                               be \label{letbe} a family of
                                               connections on $K^{-1}$
                                               satisfying
\[ \dot{\omega_t} = \curv(K^{-1}, \alpha_t) .\] 
Consider a family of immersions $\phi_t: L \to X$ satisfying
\begin{equation} \label{maslovflow} \dot{\phi_t} := \phi_t^* \omega_t \left(v_t,\cdot \right) =
  \phi^* \alpha_t  - \alpha_L \end{equation}
where 
\[ v_t: L \to \phi_t^* TX/TL \]
is the vector field determined by $\ddt \phi_t$.  The family $\phi_t$
is Lagrangian if $\phi_0$ is.  In order to give ourselves a bit more
freedom, we allow ourselves a finite number of Hamiltonian isotopies
in addition to the flow:

\begin{definition} \label{def:maslovflow} A {\em Maslov flow} for a family of symplectic
  forms $\omega_t, t \in (0,T)$ is a family of Lagrangian immersions
\[ \phi_{t,i}: L \to X, t \in [t_i,t_{i+1}] \] 
satisfying \eqref{maslovflow}, except for a finite number of times
\[ t_1,\ldots,t_k \in (0,T) \] 
for which $\phi_{t_i,i}$ and $\phi_{t_i,i+1}$ are Hamiltonian
isotopic: 
\[ \phi_{t_i,i+1} = \phi_{t_i,i} \circ \psi_{H} \] 
where $H: [0,1] \times X$ is a time-dependent smooth function and
$\psi_H: X \to X$ its Hamiltonian flow.
  \end{definition}

\begin{example} 
\begin{enumerate} 
\item {\rm (Mean-curvature flow)} Suppose that the symplectic manifold
  $(X,\omega)$ is equipped with a K\"ahler structure $J$, that is, an
  integrable almost complex structure $J: TX \to TX$ so that
  $\omega( \cdot, J \cdot)$ is a Riemannian metric.  The {\em mean
    curvature one-form} is the unique one-form
\[ H_\phi \in \Omega^1(L) \] 
satisfying the following: Given a closed one-form
$\beta \in \Omega^1(L)$, let $\phi_t: L \to X, t \in (-\eps,\eps)$ be
a variation of $\phi:L \to X$ corresponding to the one-form $\beta$
and denote by $\Vol(L,\phi_t) > 0 $ its volume with respect to the
metric induced by pull-back of the K\"ahler metric on $X$.  Then
  \[ \dds \Vol(L,\phi_t) = \int_L \lan \beta, H_\phi \ran \d \Vol_L \]
  where $\lan \cdot, \cdot \ran$ \label{doubledot} is the pairing between one-forms induced by
  the metric.  In particular an immersion $\phi: L \to X$ represents a
  critical point of the volume if and only if the mean curvature one-form
  vanishes.  The mean curvature one form $H_\phi$ is not closed but
  rather its derivative is the Ricci curvature
  $R_\omega \in \Omega^2(X)$:
\[ \d H_\phi = \phi^* R_\omega \]
by an observation of Smoczyk \cite[Section 1.7]{smoczyk} using the
{\em traced Codazzi equation}.  

Several other papers comment on the resulting coupled flow.  By
Cieliebak-Goldstein \cite{cg:note}, Lotay-Pacini \cite[Proposition
4.3]{lotay:coupled}, with respect to the trivialization of the
anticanonical bundle over the Lagrangian, the mean curvature one-form
is related to the difference of connections by 
\[ H_\phi = ( -1/\pi ) (\phi^* \alpha_X - \alpha_L) \in
\Omega^1(L). \]
Hence any coupled mean-curvature/K\"ahler-Ricci flow is a reverse
Maslov flow up to a rescaling of the time parameter by $1/\pi$.  The
same paper \cite[Section 7]{lotay:coupled} proves short-time existence
of the coupled flow.
\item \label{unions} {\rm (Unions)} Let $\phi_{t,0}: L_0 \to X$ and
  $\phi_{t,1}: L_1 \to X$ be Maslov flows. Then the disjoint union
\[  \phi_{t,0} \sqcup \phi_{t,1}: L_0 \sqcup L_1 \to X \]  
is also a Maslov flow.  Indeed, the mean curvature on the union
restricts to $H_{\phi_{t,k}}$ on $L_k, k \in \{ 0, 1 \}$.
\item \label{products} {\rm (Products)} Let $\phi_{t,0}: L_0 \to X_0$
  and $\phi_{t,1}: L_1 \to X_1$ be Maslov flows.  Then the product
\[ \phi_{t,0} \times \phi_{t,1}: L_0 \times L_1 \to X_0 \times X_1 \] 
is also a Maslov flow.  Indeed, the mean curvature on the product
$H_{\phi_{t,0} \times \phi_{t,1}} \in \Omega^1(L_0 \times L_1)$ is the
sum of pull-backs of $H_{\phi_{t,k}}$ on $L_k, k \in \{ 0, 1 \}$.
\end{enumerate} 
\end{example} 

The following are elementary examples of Maslov flow:

\begin{example} \label{elem}
\begin{enumerate}
\item {\rm (Linear case)} Let $X = \bR^{2n}$ and $L = \bR^n$ with
  immersion $l \to (l,0)$.  Then the constant isotopy
\[  \phi_s: L \to X, \ q \mapsto (q,0) \]  
is a Maslov flow.  Indeed the standard trivialization
$K^{-1} = X\times \C$ agrees with the trivialization
$\phi_s^* K^{-1} \cong L \times \C$, and so the connection one-form
$\alpha_s \in \Omega^1(K^{-1}_1)$ (and the mean curvature) vanishes.
\item {\rm (Circle case)}  \label{circlecase}
Consider the embedding 
  of the circle of radius $1$ 
\[ \phi: L =  \{ q^2 + p^2 = 1 \} \to \bR^2 =X \] 
Let $\theta: L \to S^1$ denote the angular coordinate.  The
trivialization of $\phi^* K^{-1}$ induced by the Lagrangian is that
induced by the section $e^{i \theta} \partial_x $ and so is related to
the standard trivialization by multiplication by $e^{ i \theta}$.
Thus the connection one-form of the connection with respect to the
trivialization of $\phi^* K^{-1}$ is
  \[\phi^* \alpha_X - \alpha_L = \d \theta \in \Omega^1(L) .\]
  Writing the standard symplectic form $\omega = \d q \wedge \d p $ on
  $X$ as $\omega = r \d r \wedge \d \theta$, the corresponding vector
  field is the outward normal $v_t = \partial_r/ r \in \Vect(L)$
  (corresponding to the fact that the mean curvature $H_\phi$ of a
  circle $\phi : S^1 \to \bR^2$ is inversely proportional to its
  radius).  Thus a Maslov flow $\phi_t: L \to X$ is given by outward
  flow at speed $1 / r$ with $r = 1$ at $t = 0$:
  \[ \phi_t: L \to X, \quad (q,p) \mapsto (2t +1)^{1/2} (q,p) .\]
  Thus in particular the area enclosed by the immersion is
  $\pi (1 + 2t)$ and increases linearly with $t$.
\item \label{lm} {\rm (Local model for self-tangencies)} Let $L$
  denote the union of $S^1 \times i\bR^{n-1}$ and
  $\{1 \} \times \bR^{n-1}$.  By the disjoint union and product axioms
  \eqref{unions} and \eqref{products} above, a Maslov flow is given by
\[ \phi_t(z_1,\ldots,z_n) = \begin{cases} 
  (  (1 + 2t)^{1/2} z_1,z_2,\ldots, z_n) & z_1 \in S^1 \\
  (z_1,z_2,\ldots, z_n) & z_1 \in  1 + i \bR \end{cases} .\]
 \end{enumerate}
\end{example} 

In order to obtain the invariance of Floer cohomology we will assume
that our flow is in some sense generic.  

\begin{definition} 
  By a {\em generic} Maslov flow we mean a flow $\phi_t$ defined by a
  generic choice of connection one forms $\alpha_t$ with respect to
  the $C^\infty$ topology.  For $H \in C^l([0,T] \times X)$ a change
  in the connection one-form $\alpha_t$ on $K^{-1}$ by the pull-back
  of $\d H_t$ changes the Maslov flow equation to
\begin{equation} \label{equivto}
 \dot{\phi_t} = \alpha_t - \phi_t^* \alpha_L + \phi_t^* \d H_t .\end{equation}
Thus a generic change in the connection one-form is essentially the
same as a generic Hamiltonian perturbation.  \end{definition}

The following is standard from properties of generic homotopies of
functions in Cerf \cite{ce:st}, see also Sullivan
\cite[3.12]{sullivan}.

\begin{theorem} \label{sardsmale} Let $\phi_t: L \to X, t \in [0,1] $
  be a Maslov flow such that $\phi_0$ has only transverse
  self-intersections.  There exists an open dense set of Hamiltonian
  perturbations $C^\infty([0,T] \times X)^{\reg}$ with the property
  that for $H \in C^\infty([0,T] \times X)^{\reg}$, the $H$-perturbed
  flows $ \phi_t: L \to X, t \in [0,1]$ have the property that the
  self-intersections
  $(\phi_t \times \phi_t)^{-1}(\Delta_L) -\Delta_L $ of $\phi_t$ are
  transverse except for all but finitely many values of $t$ for which
  there is a single self-intersection
  $(x_-,x_+) \in L, \phi_t(x_-) = \phi_t(x_+)$ that has a
  non-degenerate quadratic tangency.
\end{theorem}

\begin{proof} The proof is a Sard-Smale argument applied to a
  universal moduli space of intersection points.  Let
  $C^l([0,T] \times X)$ denote the space of time-dependent functions of
  class $C^l$.  Let
 \[ \Lag(L,X)_l := \Set{ \phi \in \Map([0,T] \times L,
    X)_l  | \begin{array}{l}  \ker(D \phi) = \{0 \}, \phi_0 = \phi(0)
              \\ \phi^* \omega =
    0 \ \forall t \in [0,T] \end{array} }  \]
denote the space of time-dependent Lagrangian immersions from $L$ to
$X$ of class $C^l$.  Consider the {\em universal space of
  self-intersections}
\begin{equation} \label{univsi}
\M^{\univ,si}
 := \left\{ \begin{array}{ll}  (x_1,x_2,t_0,H,\phi) \in  L \times L \times [0,1] \times C^l([0,T] \times X)
             \times \Lag(L, X)_l  \\ 
             \dot{\phi_t} = \alpha_t + \d \phi_t^* H_t, \   \quad \phi_{t_0}(x_1)
             = \phi_{t_0}(x_2) \end{array} 
         \right\}. \end{equation} 
       The universal moduli space \eqref{univsi} is a smooth Banach
       manifold by an application of the implicit function theorem, as
       follows.  First, variations in $H_t$ near $t =0 , x = x_1$ span
       the space of normal vector fields to $\phi_0$ near $x_1$.
       Indeed, let $H \in C^\infty(X)$ and $\varkappa \in C^\infty([0,T])$
       supported and equal to $1$ near $0$ a cutoff function.  The
       Maslov flow $\phi_t': L \to X$ defined by the connection
       \[ \alpha + \pi^* (1/\eps) \d H( \eps t) \in
       \Omega^1(K^{-1}_1) \]
       converges as $\eps \to 0$ to the composition of the restriction
       to $L$ of the Hamiltonian flow $\psi_t: X \to X$ of $H$ with
       the original Maslov flow $\phi_t:L \to X$.  It follows that the
       linearization of the condition
       $\phi_{t_0}(x_1) = \phi_{t_0}(x_2)$ is surjective.

       From the universal moduli space we obtain a comeager set of
       regular values for the projection.  By the Sard-Smale theorem
       for $l \ge 1$ the projection 
\[ \pi : \M^{\univ,si} \to  C^l([0,T] \times X) \] 
has a comeager set of regular values $C^l([0,T] \times X)^{\on{reg}}$.
For such regular $H \in C^l([0,T] \times X)^{\on{reg}}$, the
projection of $\pi^{-1}(H)$ onto the fiber $[0,1]$ has open and
comeager set of regular values $[0,1]^{\reg}$.  Since $L$ is compact,
the set of irregular values $[0,1] \backslash [0,1]^{\reg}$ is
discrete, hence finite.  Thus for regular $H$ the set of times $t$ for
which the intersection $L \times_{\phi_t} L$ is not transverse is
finite.  Since the space of smooth elements of this comeager set is
also open and dense, the claim follows.

         Next we show that for generic Hamiltonian perturbations the
         non-transverse intersections are non-degenerate.  Consider
         the {\em universal space of tangencies}
         \[
         \M^{\univ,t} :=
 \left\{ \begin{array}{ll} (v_1,v_2,H,\phi) \in TL \times
             TL \times [0,1] \times C^l([0,T] \times X) \times \Lag(L,
             X)_l |
                      \\
                      \dot{\phi_t} = \alpha_t + \d \phi_t^* H_t, \ D
                      \phi_{t_0}( v_1 ) = D\phi_{t_0}(v_2) \neq 0
 \end{array} \right\} \]
is cut out transversally, since the linearization of $ D \phi_t$ is
the second derivative of $\phi_t$ and a Hamiltonian variation at
$x_1,x_2$ at time $0$ produces an arbitrary variation this derivative
at time $t_0$.  Let $C^l([0,T] \times X)^{\on{\reg},'}$ denote the set
of regular values for the projection $\pi$ onto $ C^l([0,T] \times X)$
For $H \in C^l([0,T] \times X)^{\on{\reg},'}$ the intersection
$D_{t} \phi( T_{x_1} L ) \cap D_t \phi(T_{x_2} L)$ is trivial except
for finitely many times $t = t_1,\ldots, t_k$ at which the
intersection is one-dimensional.  So there is a point of tangency
$(x_{i,-},x_{i,+}) \in L \times_{\phi_{t_i}} L$ and for these
intersections the tangency is quadratic.

         A final argument using a {\em universal space of
           simultaneous self-tangencies}
\[ \M^{\univ,sst} :=  \left\{ \begin{array}{l} (v_1,v_2,v_3,v_4,t,H,\phi) \in  
TL^4  \times [0,1] \times C^l([0,T] \times X)
             \times \Lag( L, X)_l | \\
             \dot{\phi_t} = \alpha_t + \d \phi_t^* H_t,  
             D\phi_t(v_1) = D\phi_t(v_2), \ D\phi_t(v_3) =
             D\phi_t(v_4) \\ \{ \bR v_1,\bR v_2, \bR v_3, \bR v_4 \}  \ 
             \text{distinct, non-zero}
                           \end{array} \right\} \]
            guarantees that at most one tangency occurs at each
            non-transverse time. 
          \end{proof}

\begin{definition} 
\begin{enumerate} 
\item {\rm (Self-tangency moves)} A Maslov flow $\phi_t: L \to X$
  undergoes {\em a self-tangency move} at $p \in X$ and time $t = 0$
  if the following holds: There exist Darboux coordinates
  $(x_1,y_1,\ldots,x_n,y_n)$ on a contractible open neighborhood $U$
  of $p$ such that the intersection $U \cap \phi_s(L)$ is the union of
  Lagrangians $L_1$ and $L_2(t)$ where
\begin{equation} \label{localmodel} L_1 = L_1' \times L_1'', \quad
  L_2(t) = L_2'(t) \times L_2'' \end{equation}
for some Lagrangians 
\[ L_1', L_2'(t) \subset \C, \quad L_1'', L_2'' \subset \bC^{n-1} \]
where 
\[ L_1' \subset \{ y_1 = 0 \}, \quad L_2'(t) \subset \{ x_1^2 + (y_1 -
1)^2 = (1 + 2t)^2 \} \] 
are subarcs of the straight line and the circle in the $z_1$ plane,
respectively, and the factors $L_1''$ and $L_2''$ are linear
Lagrangians in $\bC^{n-1} = \{ z_1 = 0 \}$ intersecting transversally
at $0$.  The mean curvature to a circle $L_2(t)'$ is the outward
normal while the mean curvatures for the linear spaces are trivial.
Hence the isotopy $\phi_t$ is a Maslov flow as in Example \ref{elem}.
\item {\rm (Admissible flows)} Let $\phi_t: L \to X$ be a Maslov flow
  of Lagrangian immersions with self-tangency instants
  $t_1,\ldots, t_k \in (0,T)$ at points $x_1,\ldots, x_k \in X$. Say
  that $\phi_t:L \to X, s \in [0,T]$ is an {\em admissible} family of
  immersions if there exist small disjoint intervals
  $(t_j - \delta, t_j + \delta) \subset [0,T]$ such that all
  self-tangency instants $t_j$ are standard in the sense that $\phi_t$
  near $t_j$ is given by the local model \eqref{localmodel}, and in
  addition any self-tangency occurs at a rational time
  $t \in (0,T) \cap \bQ$.
\end{enumerate}
\end{definition} 

The rationality of the time will be used in Section
\ref{bdeath} \label{bdeathlab}
to
define the Floer theory in a neighborhood of the singular time.

\begin{lemma} \label{admissible} {\rm (Similar to Lemma 3.6 of
    \cite{ees})} Let $\phi_t:L \to X, t \in [t_0,t_1]$ be a Maslov
  flow of Lagrangian immersions connecting $\phi_{t_0}$ to
  $\phi_{t_1}$.  Then for any $\eps$ there is an admissible Maslov
  flow $\phi_t'$ of Lagrangian immersions $\eps$-close to $\phi_t$
  connecting $\phi_{t_0}$ to $\phi_{t_1}$.
\end{lemma}  

\begin{proof} 
First we perturb to make the flow generic.  Locally the
  family of immersions $\phi_t: L \to X$ is given by a family of
  closed one-forms $\alpha_t \in \Omega^1(L)$ as in \eqref{cartaneq}.
  For each self-tangency instant $t_j$, there exist points
  $x_j^-, x_j^+ \in L$ with $\phi(x_j^+) = \phi(x_j^-) = x_j$ such
  that
\[   
D \phi_{t_j}(x_j^-) (T_{x_j^-} L) \cap D \phi_{t_j}(x_j^+) (T_{x_j^+}
L) \neq \{ 0 \} .\]
Using the Sard-Smale theorem one checks that for a dense open subset
of Hamiltonian perturbations $H_t$ as in \eqref{equivto}, the equation
cuts out the solutions $(x_j^-, x_j^+)$ transversally.  Equivalently,
\[ \dim I_j = 1, \quad I_j := D \phi_{t_j}(x_j^-) (T_{x_j^-} L) \cap D
\phi_{t_j}(x_j^+) (T_{x_j^+} L) \]
and the tangency of the components of $\phi_{t_j}^{-1}(I_j)$ is
quadratic.  

Having achieved a quadratic tangency, we apply a Hamiltonian isotopy
to make the Lagrangians of standard form just before the isotopy.  We
suppose that the self-tangency represents the death of two
self-intersection points; the birth case is easier.  By
changing \label{bychanging} the cubic and higher order terms in the
generating functions for the Lagrangians locally we obtain an isotopy
$L_{12}(t)$ between $L_1(t_j - \delta) \cup L_2(t_j - \delta)$ and
$L_1^{\st} \cup L_2^{\st}(t)$ where $L_2^{\st}(t)$ is a circle of
radius $r(t)$ depending on $t$.  As in Moser's proof of the Darboux
lemma, there exists a diffeomorphism $\psi$ between neighborhoods
$U_0,U_1$ of $0$ in $\bC^n$ taking $L_1(t) \cup L_2(t)$ to
$L_1^{\st} \cup L_2^{\st}(t)$.  Let $\omega_0 \in \Omega^2(U_0)$ be
the standard symplectic form, and
$\omega_1 = \psi^* \omega_0 \in \Omega^2(U_1)$ the pull-back.  Let
\[ \omega_t = (1-t) \omega_0 + t \omega_1 \in \Omega^2(U_1). \] 
In order to correct $\psi$ to a symplectomorphism note
that \label{notethe} $\ddt \omega_t$ is exact if $\ddt A_t = 0$ where
$A_t = \int_{S} u^* \omega_t$ is the area of the disk $u: S \to \bC^n$
whose boundary from $1$ to $-1$ resp.  $-1$ to $1$ lies in $L_1^{\st}$
resp. $L_2^{\st}$ connecting the two self-intersection points.  By
choosing the parameter $r(t)$ and the isotopy suitably, we may ensure
that this area $A_t$ is constant in the isotopy $L_{12}(t)$; then
Moser's argument produces the desired symplectomorphism equal to the
identity on $L_1^{\st} \cup L_2^{\st}(t)$.

The flows of the original Lagrangian immersion and the flow in the
standard model may be combined using a cutoff function.  In the local
model, the flow of $L_1(t)$ is given by the graph of the differential
$\d H_1(t)$ of a time-dependent Hamiltonian
$H_1(t): L_1^{\st} \to \R$, while $L_2(t)$ is the graph of $\d H_2(t)$
for some function $H_2 : L_2^{\st}(t) \to \R$.  Finally, by choosing a
generic complement to $T_0 L_1^{\st}$, we may assume that
$L_2^{\st}(t)$ is the graph of $\d H_{12}(t)$ for some
$H_{12} (t) : L_1^{\st} \to \R$, By choosing $U$ sufficiently small,
we may assume that $\d H_1(t), \d H_2(t)$ are small in comparison with
$\d H_{12}(t)$, since $\d H_{12}(t)$ has at worst quadratic vanishing
at $(q,p) = 0$ and $t = 0$, and $\d H_2(t), \d H_1(t)$ vanish to third
order.  Let $\varkappa$ be a cutoff function vanishing in a neighborhood of
$0$.  The flow $\phi'_t: L \to X$ obtained by replacing
$L_1(t), L_2(t)$ by the graphs of $\d( \varkappa H_1(t)), \d( \varkappa H_2(t))$
have the same self-intersections, and are equal to the standard flow
in a neighborhood of $0$.  By \eqref{equivto}, $\phi'_t$ is also a
Maslov flow, and is of standard form near the tangency.  

At the end time of the combined flow, the two nearby self-intersection
points have disappeared, and so the resulting Lagrangians are disjoint
and equal outside of a small ball. In particular, the flow produces an
exact deformation from $L_1'(t_j + \delta) \cup L_2'(t_j + \delta)$ to
$L_1(t_j + \delta) \cup L_2(t_j + \delta)$ that is the identity
outside of a small ball, and a standard argument implies that this
isotopy is achievable by a Hamiltonian flow.  Thus, changing the flow
by a Hamiltonian isotopy at $t_j + \delta$ produces a piece-wise
smooth flow that matches up with the original flow for
$t \ge t_j + \delta$.

Finally, the time of the self-tangency can be perturbed an arbitrary
small amount using a Hamiltonian isotopy, so that the times $t_j$ of
the self-tangencies
$ D\phi_{t_j}( T_{x_1} L) \cap D\phi_{t_j}(T_{x_2} L) \neq \{ 0 \}$
become rational.
\end{proof}

\section{Holomorphic disks with self-transverse boundary condition}
\label{tbc}

The Morse model of Floer theory counts treed pseudoholomorphic disks
with Lagrangian boundary condition.  First we recall basic terminology
regarding stable disks.  A {\em disk} will mean a
$2$-manifold-with-boundary $S_{\white}$ equipped with a complex
structure $j_{S_{\white}}: TS_{\white} \to TS_{\white}$ so that the
surface $S_{\white}$ is biholomorphic to the closed unit disk
$ \{ z \in \bC \ | \ |z| \leq 1 \}$.  A {\em sphere} will mean a
complex one-manifold $S_{\black}$ biholomorphic to the complex
projective line
$\P^1 = \{ [\zeta_0:\zeta_1] \ | \ \ \zeta_0,\zeta_1 \in \bC \}$.  A
{\em nodal disk} $S$ is a union
\[ S = \left( \bigcup_{i=1}^{n_{\white}} S_{\white,i} \right) \cup 
\left(\bigcup_{i=1}^{n_{\black}} S_{\black,i} \right) / \sim \]
of a finite number of disks $S_{\white,i}, i = 1,\ldots, n_{\white}$
and spheres $S_{\black,i}, i =1,\ldots, n_{\black}$ identified at
pairs of distinct points called {\em nodes} $w_1,\ldots, w_m$.  Each
node 
\[ w_k = (w_k^-,w_k^+) \in S_{i_-(k)} \times S_{i_+(k)} \]  
is a pair of distinct points $w_k^\pm \in S_{i_\pm(k)}$ where
$S_{i_\pm(k)}$ are the (disk or sphere) components of $S$ adjacent to
the node $w_k$; the resulting topological space $S$ is required to be
simply-connected, see Figure \ref{disk} (produced using Inkscape and
\cite{wiki}.)  The complex structures on the disks and spheres induce
a complex structure on the tangent bundle $TS$ (which is a vector
bundle except at the nodal points) denoted $j:TS \to TS$.
\begin{figure} 
\begin{center} \includegraphics[height=1.5in]{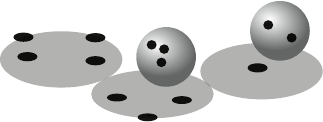} 
\end{center} 
\label{disk} 
\caption{A stable disk with boundary and interior markings} 
\end{figure} 
A {\em boundary resp. interior marking} of a nodal disk $S$ is an
ordered collection of non-nodal points
$\ul{z} = (z_0,\ldots, z_d) \in \partial S^{d+1}$ resp.
$\ul{z}' = (z_1',\ldots,z_c') \in \on{int}(S)^c$ on the boundary
resp. interior, whose ordering is compatible with the orientation on
the boundary $\partial S$ in the case of boundary markings.  The {\em
  combinatorial type} 
is the graph
\[ \Gamma(S) = (\on{Vert}(\Gamma(S)), \on{Edge}(\Gamma(S): (h \times
t): \on{Edge}(\Gamma(S)) \to \on{Vert}(\Gamma(S)) \cup \{ \infty
\} \]  
obtained by setting $\on{Vert}(\Gamma(S))$ to be the set of disk and
sphere components in $S$ and $\on{Edge}(\Gamma(S))$ the set of nodes
$w_1,\ldots,w_m$ (each connected to the vertices corresponding to the
disks or spheres they connect) and markings
$z_{i,\black}, z_{j,\white}$ (each connecting the corresponding vertex
to infinity); the graph $\Gamma(S)$ is required to be a tree, that is,
connected with no cycles among the combinatorially finite edges.  The
set of edges
 is naturally equipped with a
partition into subsets
\[ \on{Edge}(\Gamma(S)) = \on{Edge}_{\black}(\Gamma(S)) \cup
\on{Edge}_{\white}(\Gamma(S)) \] 
corresponding to interior resp. boundary markings and nodes,
respectively.  The set of boundary edges
$h^{-1}(v) \cup t^{-1}(v) \cap \Edge_{\white}(\Gamma(S))$ meeting some
vertex $v \in \Ver(\Gamma(S))$ is naturally equipped with a cyclic
ordering giving $\Gamma(S)$ the partial structure of a ribbon graph.
The set of boundary and interior {\em semi-infinite leaves}
\begin{eqnarray*}
\Edge_{\white,\rightarrow}(\Gamma(S)) &:=& \Edge_{\white}(\Gamma(S))
\cap \Edge_{\rightarrow}(\Gamma(S)) ,\\
\Edge_{\black,\rightarrow}(\Gamma(S)) &:=& \Edge_{\black}(\Gamma(S)) \cap
\Edge_{\rightarrow}(\Gamma(S)) , \\ 
\Edge_{\rightarrow}(\Gamma(S)) &:=&
h^{-1}(\infty) \cup
t^{-1}(\infty) \end{eqnarray*}
is each equipped with an ordering; these orderings will be omitted
from the notation to save space.  We call the leaves corresponding to
interior resp. boundary markings {\em bulk} resp. {\em boundary
  leaves}, for lack of better terminology.  A marked disk $(S,\ul{z},\ul{z}')$
is {\em stable} if it admits no automorphisms preserving the markings.
The moduli space of stable disks with fixed number of boundary
markings and no interior markings admits a natural structure of a cell
complex which identifies the moduli space with Stasheff's
associahedron.

Treed disks are defined by replacing nodes with broken segments as in
the pearly trajectories of Cornea-Lalonde \cite{cl:clusters},
Biran-Cornea \cite{bc:ql}, and also Seidel \cite{seidel:genustwo}.  By
a {\em segment} we will mean an oriented connected Riemannian manifold
$T$, necessarily isometric with a closed sub-interval of the real line
$[t_-,t_+] \cap \R, \ t_\pm \in \{ - \infty \} \cup \bR \cup \{ \infty
\}$.
A {\em broken segment} $T$ is obtained from a finite collection of
segments $T_1,T_2$ with an infinite positive resp. negative end by
$T = T_1 \cup \{ \infty \} \cup T_2$ where
$\{ \infty \} = \ol{T_1} \cap \ol{T_2} $ is the {\em point of
  breaking}, with the topology that makes $T$ also a closed interval.
The metrics on the open subsets $T_1,T_2$ are given as part of the
data; that is, any broken segment is equipped with a metric on the
complement of the finite set of breaking points.  A {\em treed disk}
is a space $C$ obtained from a nodal disk $S$ by replacing each
boundary node or boundary marking corresponding to an edge
$e \in \Gamma(S)$ with a broken segment $T_e$; we also allow the case
$C \cong \R$ with a single edge infinite in both directions in which
case the underlying ``disk'' is empty.  \footnote{More generally one
  could also replace interior nodes and markings with segments, but
  since in this case we will rule out sphere bubbling there is no need
  for the most general construction.}  Each broken line is obtained
from a sequence of possibly infinite closed real intervals by gluing
together endpoints, and is equipped with a {\em length}
 and {\em number of breakings}
 \[ \ell(e) \in [0,\infty], \quad b(e) \in \bZ_{\ge 0} \]
 where the length $\ell(e)$ is equal to infinity by assumption if the
 segment $T_e$ replaces a marking $z_e \in S$.  For the
 combinatorially finite edges, $\ell(e) = \infty$ if and only if the
 number of breakings $b(e)>0$.  Denote by $T \subset C$ the union of
 (finite or infinite length) segments so produced.  Then
 $C = S \cup T$ where the one-dimensional part $T$ is joined to the
 two-dimensional part $S$ at a finite set of points on the boundary of
 $S$, which we call the {\em nodes} of the treed disk (as they
 correspond to the nodes in the underlying nodal disk.)  The
 semi-infinite edges in the one-dimensional part $T$ are oriented by
 requiring that the first boundary marking is outgoing while the
 remaining semi-infinite edges are incoming; the outgoing
 semi-infinite edge is referred to as the {\em root} while the other
 semi-infinite edges are {\em leaves}.  On each segment
 $T_e \subset T$ corresponding to an edge $\on{Edge}(\Gamma(S))$ we
 assume the existence of a coordinate $s$ that identifies $T_e$ with
 the interval $[0,\ell(e)]$; this coordinate will be used in the
 construction of gradient trees below.  The combinatorial type
 $\Gamma(C)= (\on{Vert}(C), \on{Edge}(C))$ of a broken tree $C$ is
 defined similarly to that for broken disks but now the set of edges
 $\on{Edge}(C)$ is equipped with a partition
 $\on{Edge}(C) = \on{Edge}_0(C) \cup \on{Edge}_{(0,\infty)}(C) \cup
 \on{Edge}_\infty(C) $
 corresponding to whether the length is zero, finite and non-zero, or
 infinite, \label{orinfinite} and equipped with a labelling by the
 number of breakings.  A treed disk is stable if the underlying disk
 is stable, each combinatorial finite edge has at most one breaking,
 and each semi-infinite edge is unbroken.  An example of a treed disk
 with one broken edge (indicated by a small hash through the edge) is
 shown in Figure \ref{treeddisk}.

\begin{figure} 
\begin{center} \includegraphics[height=2.5in]{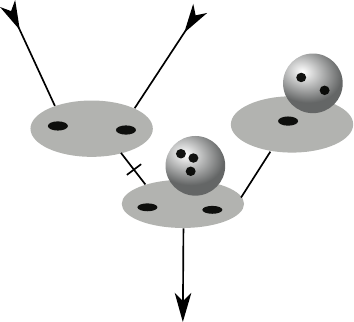} 
\end{center} 
\label{treeddisk} 
\caption{A treed disk with $d = 2$ incoming edges} 
\end{figure} 

The stability condition gives a compact, Hausdorff moduli space with a
universal curve.  For a given combinatorial type $\Gamma$ denote by
${\M}_\Gamma$ the moduli space of treed disks of combinatorial type
$\Gamma$ and $\ol{\M}_d = \cup_\Gamma {\M}_\Gamma $ the union over
combinatorial types $\Gamma$ with $d$ incoming semi-infinite
edges. The moduli space $\ol{\M}_d$ is compact with a universal curve
$\ol{\U}_d$ given as the space of isomorphism classes of pairs $[C,z]$
where $C$ is a holomorphic treed disk and $z \in C$ is a point in the
one or two dimensional locus.  Depending on which is the case one has
a splitting
\begin{equation} \label{splitting} \ol{\U}_d = \ol{\S}_d \cup
  \ol{\T}_d \end{equation}
of the universal treed disk into one-dimensional and two-dimensional
parts $\ol{\T}_d$ resp.  $\ol{\S}_d$. That is,
$\ol{\U}_d = \ol{\T}_d \cup \ol{\S}_d $ where the fibers of
$\ol{\T}_d \to \ol{\M}_d$ resp.  $\ol{\S}_d \to \ol{\M}_d$ are one
resp. two-dimensional.  We denote by $\ol{\S}_\Gamma, \ol{\T}_\Gamma$
the parts of the universal treed disk living over $\ol{\M}_\Gamma$.

Holomorphic treed disks for immersed Lagrangians are defined as in the
embedded case, but requiring a double cover of the tree parts to
obtain the boundary lift.  Given a treed disk $C$ we define a
one-manifold $\partial C$ by gluing together the boundary of each disk
$S_v, v \in \Ver(\Gamma)$ minus the points
$T_e \cap S, e \in \Edge(\Gamma)$ where edges attach with {\em two
  copies} of each edge $T_e$ as in Figure \ref{boundary}.

\begin{figure} 
\begin{center} \includegraphics[height=2.5in]{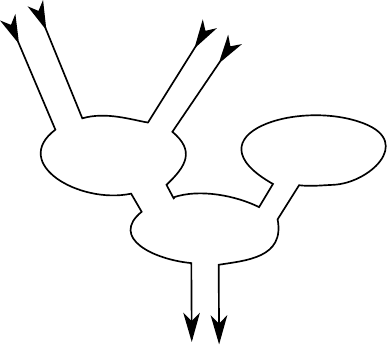} 
\end{center} 
\caption{Boundary of a treed disk with $d = 2$ incoming edges} 
\label{boundary} 
\end{figure} 

To define holomorphic treed disks, we make a choice of almost complex
structure and Morse function.  
\begin{enumerate} 
\item Let $J: TX \to TX$ be an almost complex structure taming the
  symplectic form $\omega \in \Omega^2(X)$; such a choice is unique up
  to isotopy.  Later, for the purposes of constructing Donaldson
  hypersurfaces we will assume that $J$ is {\em compatible} with
  $\omega$ in the sense that $\omega( \cdot, J \cdot)$ is a Riemannian
  metric.
\item Let $m: L \to \R$ be a Morse function.  Given a metric on $L$ we
  obtain a gradient vector field and time $t$ flow 
  \[ \grad(m) \in \Vect(L), \quad \psi_t \in\on{Diff}(L) .\]
  For any critical point $x \in \crit(m)$ we have {\em stable and
    unstable manifolds}
\[ W_x^\pm = \Set{ y \in L \ | \ \lim_{t \to \pm \infty} \psi_t(y) = x } .\]
We assume that the Morse function and metric are {\em
  Morse-Smale} in the sense that the unstable and stable manifolds
intersect transversally:
\[ W_x^+ \pitchfork W_y^- , \quad \forall x,y \in \crit(m) .\]
Thus in particular the {\em moduli space of Morse trajectories} 
\[ \M(x,y) \cong (W_x^+ \cap W_y^- ) / \bR \]
is smooth for any $x,y \in \crit(m)$; these are special cases of
moduli spaces of holomorphic treed disks, namely the case that the
number of disks in the configuration is zero.  Choose coordinates $s$
on each segment of $T$ so that each segment has the given length.  We
extend $m$ to the space of self-intersection points $L \times_\phi L$
by extension by zero, so that any gradient trajectory
$u: T^2_e \to L \times_\phi L$ in $L \times_\phi L$ is constant.
\end{enumerate} 

\begin{definition} A {\em holomorphic treed disk} with boundary in
  $\phi: L \to X$ consists of a treed disk $C = S \cup T$ and a pair
  of continuous maps
  \[u: C \to X, \quad \partial u: \partial C \to L \]
  satisfying the conditions that the map is a pseudoholomorphic map on
  the surface part, gradient trajectory on the tree parts, and the map
  $\partial u$ is a lift to $L$ of the restriction of $u$ to the
  boundary of $S$:
  \begin{eqnarray} \label{conds}
    J \d_H  (u  |_{S}) &=& \d_H  (u |_{S} ) j  \\
    - \grad(m) (\partial u |_T)    &=& \dds ( \partial u |_T)\\
    u | \partial C &=& \phi \circ \partial u \end{eqnarray}
  where $\d_H u = \d u - H(u)$.
\end{definition} 

Note that the one-dimensional part $T$ of the domain $C$ is equipped
with a decomposition $T = T_2 \cup T_1$, so that $u | T_2$ takes
values in the finite set of self-intersections
\[L \times_\phi L - \Delta_L,  \quad \Delta_L = \{ (l,l) \ | l \in L
\}\]
and $u | T_1$ is a map to $L$. We call $T_1$ resp. $T_2$ the {\em
  unbranched} resp. {\em branched} one-dimensional locus.  

The combinatorial data of a treed holomorphic disk is packaged into a
labelled graph called the {\em combinatorial type}: for a
pseudoholomorphic treed disks $u: C \to X$ is the type is the
combinatorial type $\Gamma$ of the underlying treed disk $C$ together
with the labelling of vertices $v \in \on{Vert}(\Gamma)$ corresponding
to sphere and disk components $S_v, v \in \Ver(\Gamma)$ by their
(relative) homology classes 
\[ d(v) \in H_2(X) \cup H_2(\phi) \] 
and the labelling 
\[ t(e) \in \{ 1, 2\} \]
 of edges by their branch type (whether
they map to $\Delta_L$ or to $L \times_\phi L$), and the labelling of the
semi-infinite edges by the limits $x_0,\ldots, x_d \in L$.

A compactified moduli space for any type is obtained after imposing a
stability condition.  A holomorphic tree disk $u: C = S \cup T \to X$
is {\em stable} if it has no automorphisms, or equivalently
\begin{enumerate} 
\item each disk component $S_{v,\white} \subset S$ on which the map
  $u$ is constant (that is, a ghost disk bubble) has at least one
  interior node $S_{v,\white} \cap T_\black \neq \emptyset$ or has at
  least three boundary nodes
  $\# ( S_{v,\white} \cap T_\white ) \ge 3$, where $T_\white$
  resp. $T_\black$ denotes the union of segments $e$ corresponding to
  $e \in \Edge_{\white}$ resp. $e \in \Edge_{\black}$;

\item each sphere component $S_{v,\black} \subset S$ on which the map
  $u$ is constant (that is, a ghost sphere bubble) has at least three
  nodes $\# ( S_{v,\black} \cap T_\black ) \ge 3$;
\item each broken segment $T_{e,i} \subset T_e$ on which the map $u$
  is constant has at most one infinite end, that is, one of the ends
  of $T_{e,i}$ is an attaching point to a sphere or disk
  $S_v \subset S$.
\end{enumerate}
Note that the case $C \cong \R$ equipped with a non-constant 
Morse trajectory $u: C \to L$ is allowed under this stability
condition. 
The {\em energy} of a treed disk is the sum of the energies of the
surface components,
\[E_H(u) = \int_S (1/2) | \d_H u|_S^2 \d \Vol_S \]
and for holomorphic treed disks with $H = 0$ is equal to the
symplectic area
\[ A(u) = \int_S (u|_S)^* \omega .\]
For any combinatorial type $\Gamma$ denote by ${\M}_\Gamma(\phi)$ the
moduli space of finite energy stable treed holomorphic disks of type
$\Gamma$.  Denote by 
\[ \ol{\M}_d(\phi) = \bigcup_\Gamma \M_\Gamma(\phi) \]
the union over combinatorial types with $d$ incoming edges.  

\begin{proposition} \label{prop:cutout} In a neighborhood of any
  holomorphic treed disk $u: C \to X$ of type $\Gamma$ with stable
  domain the moduli space $\M_\Gamma(\phi)$ is cut out by a Fredholm
  map of Banach spaces.
\end{proposition}

\begin{proof} The proof is standard combination of various Banach
  spaces of maps and Sobolev multiplication theorems.  Let
  $C = S \cup T $ be the domain of $u$ and denote by $S^\circ$ the
  surface obtained by removing the nodes $q_k$ that map to
  self-intersection points $\phi(L \times_\phi L - L)$.  Fix a metric
  on $S^\circ$ with strip-like ends near those self-intersection
  points.  For $kp > 2$ let $\Map_{k,p}(S^\circ,X)$
  resp. $\Map_{k,p}(T,L)$ denote the space of continuous maps with
  finite $W^{k,p}$-norm with respect to suitably chosen covariant
  derivatives on the domains and target.  The base of the required
  bundle is the space
\begin{multline} \label{bundlebase} \B_\Gamma = \left\{ 
    \begin{array}{l} (C,u_S, \partial u_S, u_{T_1},u_{T_2})
      \in \left( \begin{array}{l} \M_\Gamma
      \times \Map_{k,p} (S^\circ, X) 
   \times \Map_{k-1/p,p} (\partial S^\circ, L) 
\\ \times \Map_{k,p}(T_1,L) \times
      \Map_{k,p}(T_2, L \times_\phi L) \end{array} \right) |   \\
      u | \partial S^\circ = \phi \circ\partial u, 
 \phi \circ u_{T_1} |_{S^\circ \cap
      T_1} = u_S |_{\partial S^\circ \cap T_1} 
\end{array}  \right\} .\end{multline}
Note that by the Sobolev trace theorem the boundary map
$\Map_{k,p}(S^\circ,X)$ takes values in
$\Map_{k- 1/p,p}(\partial S^\circ, X)$, so the boundary condition
$ u | \partial S^\circ = \phi \circ\partial u $ is well-defined.
Local charts for $\B_\Gamma$ can be constructed using geodesic
exponentiation $\exp: TX \to X$ for some metric for which $L$ is
totally geodesic; such a metric exists as long as the
self-intersections of $L$ are self-transverse.  Given such a metric we
have compatible maps
\begin{equation} \label{exp} \Omega(S^\circ, u^* TX)_{k,p} \to
  \Map_{k,p}(S^\circ, X) \quad \Omega(\partial S^\circ, u^* TL)_{k-1/p,p}
  \to \Map_{k-1/p,p}(\partial S^\circ, L) \end{equation}
providing the local charts for $\B_\Gamma$.  The fiber of the bundle
$\cE_\Gamma$ over some map $u$ is the vector space
\begin{equation} \label{bundlefiber} 
\cE_{\Gamma,u} :=  \Omega^{0,1}(S, u_S^* TX)_{k-1,p} 
\oplus 
\Omega^1(T_1, u_{T_1}^* TL)_{k-1,p} \oplus 
\Omega^1(T_2, u_{T_2}^* TL^2)_{k-1,p}
 .\end{equation} 
Local charts are provided by almost complex parallel transport 
\begin{equation} \label{trans} \Pi_u^\xi:   \Omega^{0,1}(S, \exp_u(\xi)_S^* TX)_{k-1,p} \to
\Omega^{0,1}(S, u_S^* TX)_{k-1,p} \end{equation} 
along $\exp_u(s\xi)$ for $s \in [0,1]$.  However, in general
$\cE_\Gamma \to \B_\Gamma$ is only a $C^0$-Banach manifold, because
the transition maps between the local trivializations involve
reparametrizing the domains $S$ in $C = S \cup T$ and these
reparametrization actions are not smooth on Sobolev spaces
$W^{k,p}(S)$.

However, in any local trivialization of the universal curve one can
obtain Banach bundles with arbitrarily high regularity.  Let
$\cU^i_\Gamma \to \cM_\Gamma^i \times C$ be a collection of local
trivializations of the universal curve.  Let $\B_\Gamma^i$ denote the
inverse image of $\cM_\Gamma^i$ in $\B_\Gamma$ and $\cE_\Gamma^i$ its
preimage in of $\cE_\Gamma $.  For integers $k,p$ determining the
Sobolev class as above the Fredholm map cutting out the moduli space
over $\cM_\Gamma^i$ is
\begin{equation} \label{eqn:cutout} 
 \cF_\Gamma^i: \B_\Gamma^i \to \cE_\Gamma^i,   u \mapsto \left(\olp_{u|S}, 
  \dds u_{T_1}  + \grad(m(u_{T_1})), \dds u_{T_2} \right) .\end{equation}
In other words, $u$ is $J$-holomorphic on the surface parts
$S \subset C$, a gradient trajectory of $-m$ on the unbranched
segments $T_e \subset T^1$ and constant on the branched segments
$T_e \subset T^2$.  Indeed, in local charts and after restricting to a
sheet of the immersion $\iota: L \to X$, the condition
$ u | \partial S^\circ = \phi \circ\partial u $ simply imposes
Lagrangian boundary conditions.  After imposing a finite energy
condition each strip-like end limiting to such a node has a
well-defined limit $u(q_{j,\pm}) \in L \times_\phi L$.  The zero set
of $\cF_\Gamma^i$ need not satisfy the matching conditions
$u(q_{j,-}) = u(q_{j,+})$ on the nodes of $S$ mapping to
self-intersection points.  The matching condition may be imposed a
posteriori, that is, $\M_\Gamma(\phi)$ is the subset of
$\cF_\Gamma^{-1}(0)$ such that for each node
$u(q_{j,-}) = u(q_{j,+})$.  Since the set of self-intersection points
$L \times_\phi L - \Delta_L$ is finite, this identifies
$\M_\Gamma(\phi)$ as a connected component of the zero set
$\cF_\Gamma^{-1}(0)$.  Elliptic regularity for pseudoholomorphic
curves in \cite[Theorem B.4.1]{ms:jh} (note that the boundary
condition is embedded locally) imply that the solution space consists
of {\em smooth} maps.
\end{proof}

The linearization of the map \eqref{eqn:cutout} cutting out the moduli
space is a combination of the standard linearization of the
Cauchy-Riemann operator with additional terms arising from the
gradient operator and variation of conformal structure.  With $k,p$
integers determining the Sobolev class as above let
\begin{equation} \label{du}
\begin{split}
 D_u: \Omega^0(S, u_S^* TX, (\partial
  u_S)^* TL )_{k,p} &\to     \Omega^{0,1}(S, u_S^* TX)_{k-1,p} \\ 
               \xi  &\mapsto \nabla_H^{0,1} \xi - \hh (\nabla_\xi J) J \partial_H u_S
  \end{split}
 \end{equation}
 denote the linearization of the Cauchy-Riemann operator, c.f.
 McDuff-Salamon \cite[p. 258]{ms:jh}; here
 $\partial_H u_S = (1/2) ( J \d_H u_S + \d_H u_S j )$ and
 $\nabla_H^{0,1} \xi$ is the $0,1$-projection of
 $\nabla_H \xi = \nabla \xi - \nabla_\xi H$, where
 $H \in \Omega^1(S,\Vect_h(X))$ is the Hamiltonian vector field.
 Denote by $\ti{D}_u$ the operator given by combining the
 linearization of the Cauchy-Riemann operator with the linearized
 gradient operator and the variation of conformal structure and
 metrics on the domain:
 \begin{multline} \label{linop} \ti{D}_u: T_{[C]} \M_\Gamma \oplus
   \Omega^0(S^\circ, u_S^* TX)_{k,p} \oplus \Omega^0(T_1, u_1^* TL)_{k,p}
   \oplus \Omega^0(T_2, u_2^* TL^2)_{k,p} \\ \to
   \Omega^{0,1}(S, u^* TX)_{k-1,p} \oplus \Omega^1(T_1, u^* TL)_{k-1,p}
\oplus \Omega^1(T_2, u^* TL^2)_{k-1,p} \oplus
   \Map(S \cap T, u^* TX), \\ \quad ( \zeta, \xi_S, \xi_1,\xi_2)
   \mapsto \left( \begin{array}{l} D_u \xi_S - (1/2) J \d u \zeta, \nabla
     \xi_1 + D_\xi \grad(m(u)), \nabla \xi_2, \\
     \xi_S |_{S \cap T} - (D \phi \xi_1 |_{T_1 \cap S} \sqcup D \phi
     \xi_2 |_{T_2 \cap S}  ) \end{array}  \right). \end{multline}
Here the last factor $\Map(S \cap (T_1,T_2), u^* TX)$ enforces the
matching condition for the sections $\xi_S,\xi_1,\xi_2$ at the finite
set of nodes $S \cap (T_1 \cup T_2) \subset C$.  A holomorphic treed
disk $u: C \to X$ with stable domain $C$ is {\em regular} if the
linearized operator $\ti{D}_u$ is surjective, and {\em rigid} if it is
regular and $\ti{D}_u$ is an isomorphism, or more generally, if $C$ is
unstable, if $\ti{D}_u$ is surjective and the kernel of $\ti{D}_u$ is
generated by the infinitesimal automorphism $\on{aut}(C)$ of $C$.

A standard application of the implicit function theorem in Banach
spaces implies that the moduli space of rigid regular disks is
invariant under small perturbations of the boundary condition or
almost complex structure:

\begin{corollary} \label{implicit} Let $u_0: C \to X$ be a rigid
  regular holomorphic treed disk with boundary in a Lagrangian
  immersion $\phi_0: L \to X$ with respect to a symplectic form
  $\omega_0$.  Suppose that $\phi_t: L \to X$ is a family of
  Lagrangian immersions with respect to a family of symplectic forms
  $\omega_t$ for $t \in [-T,T]$, agreeing with the given immersion and
  form at $t = 0$.  Let $D \subset X$ a Donaldson hypersurface
  symplectic for all $\omega_t$ and $P_{\Gamma,t}$ a family of
  perturbation data for adapted holomorphic treed disks for
  combinatorial type $\Gamma_0$.  Then
\begin{enumerate} 
\item {\rm (Existence)} there exists an $\eps > 0$ such that there
  exists a family $u_t: C \to X$ for $|t | < \eps$ of holomorphic
  treed disks with boundary in the Lagrangian immersion
  $\phi_t : L \to X$;
\item {\rm (Uniqueness)} the solution $u_t$ is unique in the following
  sense: There exists a $\delta > 0$ such that $u_t$ is the unique
  such map up to isomorphism in an open neighborhood of $u_0$ in the
  Gromov topology.
\item \label{item:uniform} {\rm (Uniformity)} The sizes of $\delta $ and
  $\eps$ depend on uniform estimates for the distance from $\phi_0$ to
  $\phi_t$ and from $P_{\Gamma,0}$ to $P_{\Gamma,1}$ in $C^1$ norm and
  a lower bound for the right inverse of the linearized operator
  $\ti{D}_u$.
\end{enumerate} 
\end{corollary} 

The moduli space of holomorphic treed disks admits a natural version
of the Gromov topology which allows bubbling off spheres, disks, and
Morse trajectories.  Given a sequence
$u_\nu: C_\nu = S_\nu \cup T_\nu \to X, \partial u_\nu: \partial C_\nu
\to L$
of treed holomorphic disks with boundary in $\phi$ with bounded energy
Gromov compactness of pseudoholomorphic curves with Lagrangian
boundary conditions as in, for example, Frauenfelder-Zemisch
\cite{totreal} implies that there exist stable limits
$\lim u_{\nu} | S_\nu \to X$ on the surface parts, while limits on the
tree parts $\lim \partial u_\nu | \partial C_{\nu} \to L$ are a
consequence of convergence for gradient trajectories up to breaking.
Thus for any fixed energy bound $E$, the subset
\[\ol{\M}_d^{<E }(\phi) = \Set{ u \in \ol{\M}_d(\phi) \ | \ E(u) < E  
} \] 
satisfying the given energy bound is compact.  
 
The moduli space further decomposes according to the limits at
infinity and the expected dimension. For each end $e= 0,\ldots, d$ let
$\eps_e : [0,\infty) \to C$ denote a local coordinate on the $e$-th
end.  Define the sets of possible limits
\begin{eqnarray*} 
  \cI^{\on{c}}(\phi) &=& \crit(m) = \{ x  \in L \ | \ \d m (x) = 0 \} \\ 
  \cI^{\on{si}}(\phi) &=& \{ (x_-,x_+) \in L^2  \ | \ \phi(x_-) = \phi(x_+),
                     \ x_- \neq x_+ \} .\end{eqnarray*}
The set of self-intersections $\cI^{\on{si}}(\phi)$ has a natural 
involution 
\[\cI^{\on{si}}(\phi) \to \cI^{\on{si}}(\phi), \ x= (x_-,x_+) \mapsto \ol{x} =
(x_+,x_-) .\]
Given 
\[\ul{x} = (x_0,\ldots, x_d) \in \cI^{\on{c}}(\phi) \cup \cI^{\on{si}}(\phi)\]
denote by
\[\ol{\M}_d(\phi,\ul{x}) = \Set{  [u: C \to X] \in \ol{\M}_d(\phi,\ul{x})\ | \
 \lim_{s \to \infty}
u(\eps_e(s)) = x_e ,\  \forall e =0,\ldots, d }  \]
the locus with limits $\ul{x}$. For any integer $p$ denote by
\[\ol{\M}_d(\phi)_p = \{ [u : C \to X] \ | \ \Ind(\ti{D}_u) -
\dim(\aut(C))= p \} \]
the locus with expected dimension $p$, where $\ti{D}_u$ is the
operator of \eqref{linop}.  In the next section, we construct
perturbed moduli spaces using Cieliebak-Mohnke perturbations
\cite{cm:trans} and show that they are still compact.

\section{Transversality and compactness}

In this section we construct perturbations of the moduli spaces of
treed holomorphic disks with immersed Lagrangian boundary conditions
so that the perturbed moduli spaces are transversally cut out and
compact.  To achieve transversality, we combine the approach in
Akaho-Joyce \cite{akaho} with Biran-Cornea \cite{bc:ql}, Charest
\cite{charest:clust}, and Charest-Woodward \cite{cw:traj} which uses
Donaldson hypersurfaces to stabilize the domains of the holomorphic
treed disks; once the domains are stabilized, one may choose generic
domain-dependent almost complex structures and Morse functions to
achieve regularity, except for the following trivial case: To achieve
transversality for constant maps to self-intersections of the
Lagrangian immersion domain-dependent Hamiltonian perturbations are
required.

The existence of suitable Donaldson hypersurfaces uses auxiliary line
bundles whose curvature is the symplectic form, and so requires
rationality assumptions.  Say that $\phi$ is {\em weakly rational} if
the pairing of $[\omega] \in H^2(\phi)$ with any relative cycle is
rational.  That is, the class $[\omega]$ lies in the image of the
natural map from singular cohomology with rational coefficients:
\[[0,\omega] \in \on{Im}(H^2(\phi,\bQ) \to H^2(\phi,\R)).\]
In particular, $[\omega] \in H^2(X,\bQ)$ so that for some integer $k$,
$ k [\omega] = c_1(\ti{X})$ for some line bundle $\ti{X} \to X$.  By
Stokes' theorem the pairings of $[\omega]$ with disk classes
$u_* [S] \in H_2(\phi)$ are related to the holonomies of the
connection around the boundary by
\[ \exp( 2\pi i (u_* [S], k [\omega])) = \Hol( (\partial u)^* \ti{X}) \]
and so trivial.  The immersion $\phi$ is {\em strongly rational} if
there exists a line bundle $\ti{X} \to X$ with connection whose
curvature equals $k (2\pi/i) \omega$ for some integer $k > 0$ such that the
restriction of $\ti{X}$ to $\phi(L)$ is trivial as a line bundle with
connection.  That is, there exists an isomorphism of line bundles with
connection $\psi: \phi^* L \to L \times \C$ (where $L \times \C$ is
equipped with the trivial connection) such that
$\phi(x_1) = \phi(x_2)$ implies that $\psi_{x_1}$ agrees with
$\psi_{x_2}$ using the canonical identification
$L_{\phi(x_1)} = L_{\phi(x_2)}$.

Strong rationality implies the existence of Donaldson hypersurfaces in
the complement of the image of the Lagrangian immersion, similar to
the construction in Charest-Woodward \cite{cw:traj}.  Let
$\tilde{X} \to X$ be a line-bundle with connection $\alpha$ over $X$
whose curvature two-form $\curv(\alpha)$ satisfies
$\curv(\alpha)= (2\pi/i) \omega$; since our symplectic manifold
$(X,\omega)$ has rational symplectic class $[\omega]$ we may always
assume such $\ti{X}$ after taking a suitable integer multiple of the
symplectic form $\omega$.  Let $(\sigma_k)_{k \ge 0}$ be a sequence of
sections of $\tilde{X}^k \to X$.  The sequence $(\sigma_k)_{k \ge 0}$
is {\em asymptotically holomorphic} if there exists a constant $C$ and
integer $k_0$ such that for $k \ge k_0$,
\begin{equation} \label{asymhol}
 | \sigma_k | + | \nabla \sigma_k| + | \nabla^2 \sigma_k | \leq C, \quad |\olp \sigma_k| +
 | \nabla \olp \sigma_k| \leq C k^{-1/2} .\end{equation} 
The sequence $(\sigma_k)_{k \ge 0}$ is {\em uniformly transverse} to $0$ if
there exists a constant $\eta$ independent of $k$ such that for any
$x \in X$ with $|\sigma_k(x)| < \eta$, the derivative of $\sigma_k$ is
surjective and satisfies $| \nabla \sigma_k(x)| \ge \eta$.  In both
definitions the norms of the derivatives are evaluated using the
metric $g_k = k \omega( \cdot , J \cdot)$.  Donaldson's construction
\cite{don:symp} produces from an asymptotically holomorphic sequence
of sections a nearby sequence that is uniformly transverse, so that
the zero section $D = \sigma_k^{-1}(0)$ is a symplectic submanifold of
codimension two representing $k[\omega]$, called a {\em Donaldson
  hypersurface}.

\begin{lemma} \label{dexists} Suppose that $\phi: L \to X$ is a
  strongly rational Lagrangian immersion with only transverse
  self-intersections or non-degenerate tangencies.  There exists a
  Donaldson hypersurface $D \subset X$ with the property that
  $\phi(L)$ is exact in $X - D$.  That is, there exists a one-form
\[\alpha \in \Omega^1(X - D), \quad d \alpha = \omega\]
 and a function
\[\beta: L \to \R, \quad \d \beta = \phi^* \alpha\]
such that for any $x_1,x_2 \in L$, equality $\phi(x_1) = \phi(x_2)$ implies
$\beta(x_1) = \beta(x_2)$.  In particular, any disk with boundary in
$L$ with non-zero area intersects $D$. 
\end{lemma} 

\begin{proof} The first step is to choose an asymptotically
  holomorphic sequence of sections concentrated on the image of the
  immersion.  By the Auroux-Gayet-Mohsen \cite{auroux:complement}
  prescription, there exists a Gaussian asymptotically holomorphic
  sequence of sections $\sigma_k: X \to \ti{X}^k$ concentrated on
  $\phi(L)$.  At least locally we may model $X$ in a neighborhood $U$
  of the image of $L$ as the cotangent bundle $T^\dual L$ with
  coordinates $q_1,\ldots,q_n,p_1,\ldots,p_n$.  Parallel transport
  along geodesics $\exp_x(sv), v \in T_x X$ normal to $L$ defines a
  local trivialization of $\ti{X}$ near the image of $\phi(L)$,
  depending on the choice of branch.  An asymptotically-holomorphic
  sequence of sections $\sigma_{k,\on{loc}} : U \to \ti{X} | U $ is
  then given by
  \[\sigma_{k,\on{loc}}(q_1,\ldots,q_n,p_1,\ldots,p_n) = \exp( - k
  (p_1^2 + \ldots + p_n^2) ) .\]
  The sequence $\sigma_{k,\on{loc}}$ can be extended to a globally
  defined sequence of sections $\sigma_k$ on $X$ by multiplying by a
  cutoff function $\varkappa: X \to \R$ supported in $U$ and equal to $1$
  in neighborhood of each branch of $\phi(L)$.  Near a transverse
  self-intersection or admissible self-tangent intersection
  $x \in L \times_\phi L - \Delta_L$, the section $\sigma_k$ is the
  sum of the contributions $\sigma_{k,1} + \sigma_{k,2}$ from the
  branches.  The sum $\sigma_k$ is non-vanishing on each branch of
  $\phi(L)$, since the maximum of the norm of $\sigma_{k,m}$ occurs on
  the branch $\phi(L)_m$. Since $\sigma_{k,m}, m \in \{ 1,2 \}$
  restricts to a section on $L_{m+1}$ (with $m$ mod $2$) of norm
  strictly less than one, the sum $(\sigma_m + \sigma_{m+1}) |L_m$ is
  homotopic to $\sigma_m | L_m$, and so the Maslov computation still
  implies that the intersection number with $\sigma^{-1}(0)$ is the
  symplectic area.

  In the second step one applies Donaldson's perturbation procedure
  \cite{don:symp} to obtain asymptotically holomorphic sections
  $\sigma_k' : X \to \ti{X}^k$ uniformly transverse to the zero
  section $X \subset \ti{X}^k$. This gives a symplectic hypersurface
  $D_k := (\sigma_k')^{-1}(0)$ in the complement of the Lagrangian.
  Because the section $\sigma_k'$ is homotopic to a covariant constant
  section of $\ti{X}^k$ on the Lagrangian $L$, on the complement of
  $D_k$ we have
\[ k \omega =  \d \alpha_k,    
\alpha_k | L = \d f_k \]
where $\alpha_k$ is the connection one-form corresponding to the
section $\sigma_k$.  Thus the Lagrangian $L$ is exact in the
complement $X - D_k$ of the Donaldson hypersurface $D_k \subset X$.
%so of the form 
%%
 % \[ \sigma(q,p) \sim \exp( - kp^2) + \exp( -k q^2 + kipq) \]
%%
 % which one sees easily is non-vanishing for either $q$ or $p$
 % vanishing and $k$ sufficiently large.  Similarly near a 
 % self-tangency we have 
%%
%\[ \sigma(q,p) \sim \exp(-k p^2) + \exp(  - k( p - q^2)^2 +
%ikq^3/3) \]
%%%
%which is non-vanishing for either $p = 0$ or $p = q^2$. 
\end{proof} 

Given a weakly rational immersion one may choose a small perturbation of
the symplectic form and connection to obtain a strongly rational
immersion with respect to a slightly different symplectic form:

\begin{lemma} \label{weaktostrong} Let $\phi:L \to X$ be a weakly
  rational immersion with only isolated self-intersections
  $\phi(x_{k,+} ) = \phi(x_{k,-})$ for $k = 1,\ldots, m$.  There
  exists a collection of paths
\begin{equation} \label{coll} 
\{\gamma_{k,l,\pm,\pm}: [0,1] \to L\mid (k\pm,l\pm) \in
\Theta\} \end{equation}
  where $\gamma_{k,l,\pm,\pm}$ is a path from $x_{k,\pm}$ to
  $x_{l,\pm}$ in $L$ that are disjoint except for the endpoints and
  $\Theta \subset \cI^{\on{si}}(\phi)$ is an indexing set such that all
  distinct pairs of self-intersection points are connected by a
  concatenation of such paths.  For a generic perturbation of the
  connection one-form $\alpha_t'$ that is closed on $L$, the parallel
  transports along $\gamma_{k,l,\pm,\pm}$ are rational and so
  $\phi: L \to X$ is strongly rational with respect to the resulting
  symplectic form $\omega'$.
\end{lemma} 

\begin{proof} 
  The strategy of proof is to introduce a variation of connection and
  symplectic form on each branch individually, which preserves the
  Lagrangian condition.  The proof is slightly different in the case
  of surface target.  If $\mathrm{dim}(X)>2$ then $\Theta$ can be
  chosen to connect all self-intersection points, and if
  $\mathrm{dim}(X)=2$ then $\Theta$ can be chosen to connect adjacent
  self-intersection points.  By assumption the self-intersections
  $\phi(x_{k,\pm})$ are isolated, so that at any self-intersection
  point $\phi(x_{k,\pm})$ there are two branches $L_0,L_1$ of
  $\phi(L)$ in any sufficiently small neighborhood of a
  self-intersection point $\phi(x_{k,\pm})$.

  To produce the variation in connection, let
  $f: \R_{\ge 0} \to [0,1]$ be a smooth bump function supported in
  $[0,2]$ and equal to $1$ on $[0,1]$.  Let $x \in L$ be a point on
  which $\phi$ is not injective, that is, $\phi(x) = \phi(y)$ for some
  $y \neq x \in L$.  Choose local Darboux coordinates
  $p_1,\ldots, p_n,q_1,\ldots, q_n$ on $X$ centered at $\phi(x)$ so
  that the given branch of $\phi(L)$ is equal near $x$ to the locus
  defined by $p_1 = \ldots = p_n = 0$.  Consider perturbations
  $\alpha'$ of the connection $\alpha$ by multiples of the one-forms
\[ \alpha_{\delta,\eps} = f(\eps |p|) \d f(\delta |q|) ;\]
the support is shown in Figure \ref{supp}.  In the Figure, the branch
$\phi(L)$ on which the connection is modified is the top branch.
\begin{figure} 
\centering
%% Creator: Inkscape inkscape 0.91, www.inkscape.org
%% Accompanies image file 'supp.pdf' (pdf, eps, ps)
%% To scale the image, writex
%%  instead o
\begingroup%
  \makeatletter%
  \providecommand\color[2][]{%
    \errmessage{(Inkscape) Color is used for the text in Inkscape, but the package 'color.sty' is not loaded}%
    \renewcommand\color[2][]{}%
  }%
  \providecommand\transparent[1]{%
    \errmessage{(Inkscape) Transparency is used (non-zero) for the text in Inkscape, but the package 'transparent.sty' is not loaded}%
    \renewcommand\transparent[1]{}%
  }%
  \providecommand\rotatebox[2]{#2}%
  \ifx\svgwidth\undefined%
    \setlength{\unitlength}{140.6827431bp}%
    \ifx\svgscale\undefined%
      \relax%
    \else%
      \setlength{\unitlength}{\unitlength * \real{\svgscale}}%
    \fi%
  \else%
    \setlength{\unitlength}{\svgwidth}%
  \fi%
  \global\let\svgwidth\undefined%
  \global\let\svgscale\undefined%
  \makeatother%
  \begin{picture}(1,0.39595319)%
    \put(0,0){\includegraphics[width=\unitlength,page=1]{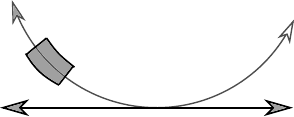}}%
    \put(0.21344341,0.24785136){\color[rgb]{0,0,0}\makebox(0,0)[lb]{\small \smash{$\supp(\alpha_{\delta,\eps})$}}}%
    \put(0,0){\includegraphics[width=\unitlength,page=2]{supp.pdf}}%
  \end{picture}%
\endgroup%
\caption{Support of the perturbation of the connection one-form} 
\label{supp} 
\end{figure} 
Since the non-injective points are assumed isolated, for sufficiently
small $\delta,\eps$ the support of $\alpha_{\delta,\eps}$ does not
intersect any other branch of $\phi(L)$.  The effect of adding
$\alpha_{\delta,\eps}$ to the connection one-form on $\ti{X}$ is to
modify the parallel transport in the fibers of $\ti{X}$ along any path
$\gamma$ to or from $x$ by a small amount $\delta h \in S^1$, without
changing the parallel transports to $x$ along other branches of
$\phi(L)$, at the cost of a small change $\delta \omega$ in the
symplectic form $\omega$ which keeps the immersion $\phi$ Lagrangian,
that is, $\phi^* (\omega + \delta \omega) =0 $.  Thus we ensure that
there is a path between each pair of self-intersection points which
has rational holonomy.  Since this perturbation preserves the holonomy
of each loop we still have that each $\gamma: S^1 \to L$ has rational
holonomy for $\phi^* \ti{X}$, all paths $\gamma_{k,l,\pm,\pm}$ from
\eqref{coll} between self-intersection points $x_k,x_l$ have the
property that the loops $\phi(\gamma_{k,\pm} \gamma_{l,\pm}^{-1})$
have rational holonomy.  After taking a suitable tensor power
$\ti{X}^{\otimes k}, k \gg 0$ of the bundle $\ti{X}$, all paths
between intersection points have the same parallel transport.  It
follows that the bundle $\ti{X}$ admits a trivialization over
$\phi(L)$.
\end{proof}

\begin{remark} Given Lemma \ref{weaktostrong}, we may assume that the
  Lagrangian immersions $\phi_s: L \to X$ for rational times $s$ are
  strongly rational.  Indeed, suppose that $L$ is exact in $X - D$
  with respect to some perturbed symplectic form $\omega'$.  The
  number of intersection points $u^{-1}(D)$ of $u$ with $D$ is $k$
  times the pairing of $u_*[C]$ with $[\omega']$; the number
  $\# u^{-1}(D)$ is necessarily positive for any almost complex
  structure $J \in \JJ_\tau(X,\omega) \cap J_\tau(X,\omega')$ taming
  both symplectic forms $\omega,\omega'$ which exists as long as
  $\omega'$ is sufficiently close to $\omega$.
\end{remark} 

Domain-dependent perturbations as in Cieliebak-Mohnke \cite{cm:trans}
are chosen using the Donaldson hypersurfaces above.  For sufficiently
large degree and a generic almost complex structure
$J_D \in \JJ_\tau(X)$ making $D$ an almost complex submanifold, the
Donaldson hypersurface $D$ contains no pseudoholomorphic spheres
$v: \P^1 \to D$ and any pseudoholomorphic sphere $v:\P^1 \to X$
intersects the hypersurface $D \subset X$ in at least three but
finitely many points $v^{-1}(D)$ \cite[Proposition 8.11]{cm:trans}.
The additional intersection points of a holomorphic treed disk
$u: C \to X$ with a Donaldson hypersurface $D$ may be used to
stabilize the domains $C$ of holomorphic disks.  Then domain-dependent
almost complex structures achieve transversality, as we now explain.
Call a holomorphic treed disk $u: C = S \cup T \to X$ {\em adapted} to
the Donaldson hypersurface $D \subset X$ if each bulk leaf $T_e$ (with
terminology from Section \ref{tbc}) lies in $u^{-1}(D)$ and each
connected component of $u^{-1}(D)$ contains an bulk leaf
$T_e, e \in \Edge_{\black,\rightarrow}(\Gamma)$; in the absence of
sphere bubbles $S_v, v \in \Ver_{\black}(\Gamma)$ mapping to the
hypersurface $D$ this means that the bulk leaves are the
intersection points with $D$, that is,
\[u^{-1}(D) = \bigcup_{ e \in \Edge_{\black,\rightarrow}(\Gamma) } T_e
.\]
Let 
\[\JJ_\tau(X) = \Set{ J : TX \to TX \ | \  \ J^2 = - \on{Id}, \quad \omega(
\cdot, J \cdot ) > 0 } \] 
denote the space of tame almost complex structures on $X$.  As in
Cieliebak-Mohnke \cite[Section 8]{cm:trans}, there exist an open
subset $\JJ_\tau(X,J_D,\Gamma)$ in the space of such domain-dependent
almost complex structures $J$ near $J_D$ with the property that each
$J$-holomorphic sphere
$u|S_v : S_v \to X, v \in \Ver_{\black}(\Gamma)$ intersects $D$ in
finitely many but at least three points.  For each combinatorial type
$\Gamma$, a {\em domain-dependent almost complex structure} is a map
\[ J_\Gamma: \ol{\S}_\Gamma \to \JJ_\tau(X,J_D,\Gamma) \]
(notation from \eqref{splitting}) agreeing with the given almost
complex structure $J_D$ on the hypersurface $D$ and in a neighborhood
of the nodes $q_k \in S$ and boundary $S$ for any fiber
$S \subset \ol{\S}_\Gamma$.  We also choose a domain-dependent
Hamiltonian perturbation: Let $\Vect_h(X) \subset \Vect(X)$ denote the
space of Hamiltonian vector fields and $\Vect_h(X,D)$ the space of
vector fields $v : X \to TX$ that vanish on an open neighborhood of
the Donaldson hypersurface $D \subset X$.  A {\em domain-dependent
  Hamiltonian perturbation} is a one-form on $\ol{\S}_\Gamma$ with
values in $\Vect_h(X,D)$
\[ H_\Gamma \in \Omega^1( \ol{\S}_\Gamma, \Vect_h(X,D)) \]
vanishing on an open neighborhood of the boundary of $\ol{\S}_\Gamma$.
A {\em domain-dependent Morse function} is a map
\[ m_\Gamma: \ol{\T}_\Gamma \times L \to \bR \]
agreeing with the given Morse function $m: L \to \R$ in a neighborhood
of the endpoints of each segment $T_i \subset T$ for any fiber
$T \subset \ol{\T}_\Gamma$.  A {\em perturbation datum} is a triple
$P_\Gamma = (J_\Gamma,H_\Gamma,m_\Gamma)$.  The space of perturbation
data of class $C^l$ is denoted $\PP^l_\Gamma = \{ P_\Gamma \}$.  The
combinatorial type of an adapted map is that of the map with the
additional data of a labelling $d(e), e \in \Edge(\Gamma)$ of any
interior node by intersection multiplicity $d(e)$ with the
hypersurface $D$; we let $d(e) =0$ if the map is constant with values
in the hypersurface $D$ near the node.  We denote the moduli space of
adapted treed holomorphic disks of type $\Gamma$ by
$\M_\Gamma(\phi,D)$.  Denote by $\ol{\M}(\phi,D)$ the union over
combinatorial types.  The requirement that the intersections with the
Donaldson hypersurface are the bulk leaves means that the energy
of a given combinatorial type of map to $X$ is determined by the
bulk leaves as well as the set of ghost bubbles. Thus in
particular, for any combinatorial type $\Gamma$ of holomorphic treed
disk there exist finitely many combinatorial types $\Gamma_{X}$ with
domain of type $\Gamma$ (allowing any combination of stable sphere
components to be ghost bubbles.)  Thus for each type $\Gamma$ of
domain the union
\[ \M_\Gamma(\phi,D) = \bigcup_{\Gamma \to \Gamma_{X}}
\M_{\Gamma_{X}}(\phi,D) \]
is a finite union of types $\Gamma_{X}$ of maps.

In order to obtain good compactness properties, we assume the
following coherence properties of the perturbations.  For each vertex
$v \in \on{Vert}(\Gamma)$, let $\Gamma(v)$ denote the subtree of
$\Gamma$ consisting of the vertex $v$ and all edges of $\Gamma$
meeting $v$.  There is natural inclusion
$\pi^* \cU_{\Gamma(v)} \to \cU_\Gamma$ and we assume that the
perturbations $\ul{P} = (P_\Gamma)$ satisfy the following axioms:
\begin{itemize}
\vskip .1in
\item{\rm (Locality)} The next axiom prevents ghost bubbles from
  forming when intersection points with divisors come together.  For
  each vertex $v \in \on{Vert}(\Gamma)$, let
\[ \Gamma(v) = \cup_{e \ni v} e \]  
denote the subtree of $\Gamma$ consisting of the vertex $v$ and all
edges $e$ of $\Gamma$ meeting $v$.  Let $\Gamma_{\white}$ denote the
subgraph of $\Gamma$ whose vertices are $v \in \Ver_{\white}(\Gamma)$.
Let
\[ \pi = \pi_1 \times \pi_2: \U_\Gamma \to \M_{\Gamma_\white} \times \U_{\Gamma(v)} \] 
be the product of the maps where $\pi_1$ is given by projection
followed by forgetful morphism and $\pi_2$ is the map $C \mapsto S_v$,
equipped with its special points. The perturbations $P_\Gamma$ are
{\em local} for $v$ if and only if $P_\Gamma$ restricts on $S_v$ to the pull-back
under $\pi$ of a family of perturbation data $P_{\Gamma,v}$ on
$\M_{\Gamma_\white} \times \U_{\Gamma(v)}$ to $\U_\Gamma$.  (The
dependence on $\M_{\Gamma_\white}$ is necessary to allow compatibility
with the (Cutting Edges) axiom; we thank Guangbo Xu for pointing this
out.)
 \vskip .1in
\item[] (Cutting-edges axiom) If $\Gamma$ is obtained from types
  $\Gamma_1,\Gamma_2$ by gluing along semi-infinite edges then
  $P_\Gamma$ is the product of the perturbations
  $P_{\Gamma_1},P_{\Gamma_2}$ under the isomorphism
  $\cU_\Gamma \cong \pi_1^* \cU_{\Gamma1} \cup \pi_2^*
  \cU_{\Gamma_2}$.
  \vskip .1in
\item[] (Collapsing-edges axiom) If $\Gamma'$ is obtained from
  $\Gamma$ by setting a length equal to zero or infinity or collapsing
  an edge then the restriction of $P_{\Gamma}$ to
  $\cU_{\Gamma} | \M_{\Gamma'} \cong \cU_{\Gamma'}$ is equal to
  $P_{\Gamma'}$.
\end{itemize} 
The origins of the axioms is rather different: the (Cutting-edges) and
(Collapsing-edges) axioms in particular imply that the moduli space
$\M_\Gamma(\phi,D)$ over the image of the inclusion
$\M_{\Gamma_1} \times \M_{\Gamma_2} \to \ol{\M}_\Gamma$ is a product
of moduli spaces over $\M_{\Gamma_1}(\phi,D)$ and
$\M_{\Gamma_2}(\phi,D)$; this implies that the terms in the \ainfty
axiom are associated to the boundary points on the moduli space of
holomorphic tree disks.  On the other hand, in principle one could
also have sphere bubbling: Cieliebak-Mohnke perturbations
$\ul{P} = (P_\Gamma)$ \cite{cm:trans} do not make all strata
$\ol{\M}_\Gamma(\phi,D)$ expected dimension in the case of ghost
bubbles.  Without the (Locality axiom), this fact could cause
additional terms in the boundary of the one-dimensional component
$\ol{\M}(\phi,D)_1$ of the moduli space of treed holomorphic maps
$\ol{\M}(\phi,D)$.  The (Locality axiom) implies that if at least one
ghost sphere bubble $u: S_v \to X, u_*[S_v] =0$ appears then
forgetting all but one marking $z_k \in S_v$ on each ghost bubble
$S_v \subset S, d(v) = 0$, one obtains a configuration of lower
expected dimension $u': S \to X,\ul{z}' \subset \ul{z}$ at least two
lower.  However, there is no recursive constraint on such
perturbations $P_{\Gamma,v}$.

Obtaining strict units requires the addition of {\em weightings} to
the combinatorial types as in Ganatra \cite{ganatra} and
Charest-Woodward \cite{flips}.  A weighting for a type $\Gamma$ is a
map $\varkappa$ from the space $\Edge_{\infty}(\Gamma)$ of semi-infinite
edges of $\Gamma$ to $[0,\infty]$.  The set of generators of the space
of Floer cochains $CF(\phi)$ is enlarged by adding two new elements
$x^{\whitet}, x^{\greyt}$ of degree $0$ resp. $-1$, with the
constraint that if the weight $\varkappa(e)$ is non-zero then the only
allowable labels of the edge $e$ are $x^{\whitet}$ (if the weight is
infinite) $x^{\greyt}$ (for any weighting).  \label{anyweighting} One then requires that the
perturbation system satisfies a forgetful axiom:
\begin{itemize}
\item[] {\rm (Forgetful axiom)} For any semi-infinite edge $e$ with
  infinite weighting $\varkappa(e) = \infty$, the perturbation datum
  $P_\Gamma$ is pulled back from the perturbation datum $P_{\Gamma'}$
  under the forgetful map $\cU_{\Gamma} \to \cU_{\Gamma'}$ obtained by
  forgetting that semi-infinite edge and stabilizing.
\end{itemize}
In particular, this axiom implies that the resulting moduli spaces
admit forgetful morphisms
$\M_{\Gamma}(\phi,D) \to \M_{\Gamma'}(\phi,D)$ whenever there is a
semi-infinite edge with infinite weighting.  See \cite{flips} for more
details on the allowable weightings.

In order to apply the Sard-Smale theorem we construct a suitable
Banach manifold of perturbation data for each type.  We assume
perturbation data $P_\Gamma$ matches that obtained by gluing
perturbation data $P_{\Gamma'}$ on strata $\U_{\Gamma'}$ contained in
$\ol{\U}_\Gamma$ on a fixed neighborhood of the boundary.  The space
$\PP_\Gamma^l$ of perturbation data of class $C^l$ taking values in
$\JJ_\Gamma(J_D)$ is a Banach manifold and the space $\PP_\Gamma$ of
perturbation data of class $C^\infty$ and fixed to be the given almost
complex structure $J_D$ and Morse function $m$ on the complement of a
compact subset is a Baire space, by the usual combination of $C^k$
norms into a non-linear metric:
\[ \dist( P_{\Gamma,1}, P_{\Gamma_2}) = 
\sum_{k \ge 0} 2^{-k} \min( 1, \Vert P_{\Gamma,1} - P_{\Gamma,2}
\Vert_{C^k} ) ;\]
see for example Royden \cite[Chapter 7.8]{royden}. 

Given a relative spin structure for the immersion, orientations on the
moduli spaces may be constructed following Fukaya-Oh-Ohta-Ono
\cite[Orientation chapter]{fooo}, \cite{orient}.   We may ignore the
constraints at the interior nodes
$z_1,\ldots, z_m \in \on{int}(S)$, since the tangent spaces to these
markings and the linearized constraints $\d u(z_i) \in T_{u(z_i)}D$
are even dimensional and oriented by the given complex structures.  At
any regular element $(u:C \to X) \in \M(\phi,D)$ the tangent space to
the moduli space of treed disks is the kernel of the linearized
operator
\[ T_u \M(\phi,D) \cong \ker(\ti{D}_u) .\]
The operator $\ti{D}_u$ is canonically homotopic via family of
operators $\ti{D}_u^t, t \in [0,1]$ to the operator
$0 \oplus D_u \oplus \dds$ on the direct sum in \eqref{linop}.  For
any vector spaces $V,W$ we have isomorphism
$\det(V \oplus W) \cong \det(V) \otimes \det(W)$.  The deformation
$\ti{D}_u^t, t \in [0,1]$ of operators induces a family of determinant
lines $\det(\ti{D}_u^t)$ over the interval $[0,1]$, necessarily
trivial, and so (by taking a connection on this family) an
identification of determinant lines
\begin{equation} \label{split1} \det(T_{u} \M(\phi,D)) \to \det(T_C
  \M_\Gamma ) \otimes \det(D_{{u}}) \end{equation}
well-defined up to isomorphism.  (Here $D_u$ denotes the linearized
operator subject to the constraints that require the attaching points
of edges mapping to critical points to map to the corresponding
unstable manifolds of the Morse function.)  In the case of nodes of
$S$ mapping to self-intersection points $x \in \cI^{\on{si}}(\phi)$ the
determinant line $\det(D_{u})$ is oriented by ``bubbling off
one-pointed disks'', see \cite[Theorem 44.1]{fooo} or \cite[Equation
(36)]{orient}.
\begin{figure}
\centering
\includegraphics[width=4in]{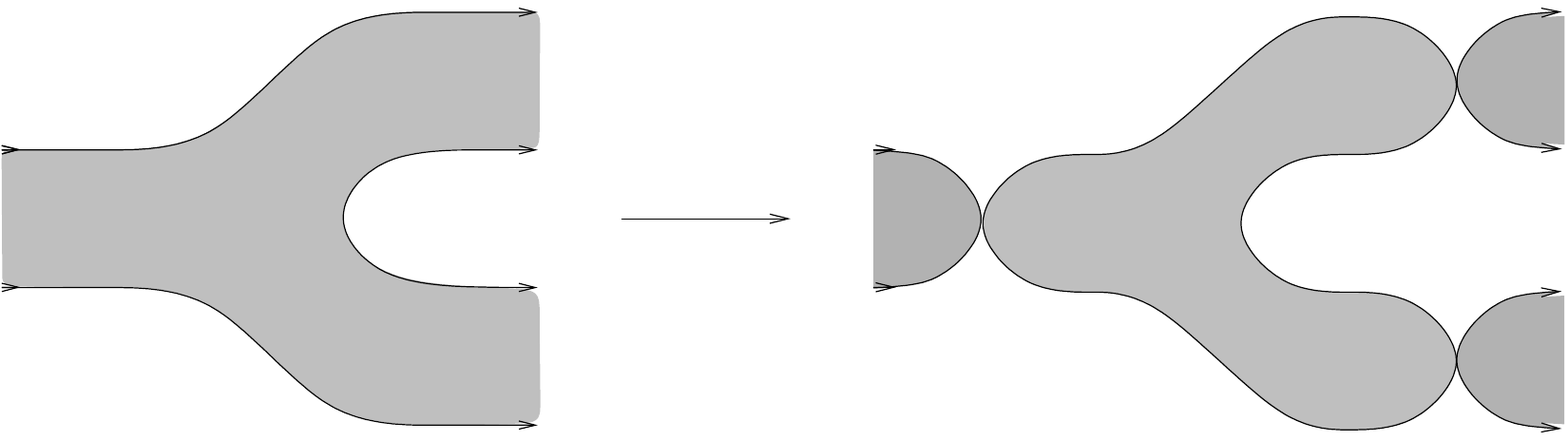}
\caption{Bubbling off the strip-like ends}
\label{closed}
\end{figure}
\noindent That is, for each self-intersection point
$(x_- \neq x_+) \in L^2, \phi(x_-) = \phi(x_+)$ choose a path of
Lagrangian subspaces
\begin{equation} \label{gammax}
\gamma_x:[0,1] \to \Lag(T_{\phi(x_-) = \phi(x_+)} X), 
\quad \gamma_x(0) = D_{x_-} \phi (T_{x_-} L )
\quad \gamma_x(1) = D_{x_+} \phi (T_{x_+} L).
\end{equation} 
Let $S$ be the unit disk with a single boundary marking
$1 \in \partial S$.  The path $\gamma_x$ defines a totally real
boundary condition on $S$ on the trivial bundle with fiber $T_x X$.
Let $\det(D_x)$ denote the determinant line for the Cauchy-Riemann
operator $D_x$ with boundary conditions $\gamma_x$ as in
\cite{orient}.  Denote by
\[ i(x) = \dim(\ker(D_x)) - \dim(\coker(D_x)) \in \bZ \]
the index of the operator $D_x$.  Let $\DD^+_{x_k,1} = \det(D_{x_k, 1}^+)$
and let $\DD^-_{x_k, 1}$ be the tensor product of the determinant line
$\det(D_{x_k}^-)$ for the once-marked disk with $\det(T_{x_k}L)$.
Because the once-marked disks with boundary conditions $\gamma_{x_k}$
and $\gamma_{\ol{x_k}}$ glue together to a trivial problem on the disk
with index $T_{x_k}L $, there is a canonical isomorphism
\begin{equation} \label{caniso} \DD^-_{x_k, 1} \otimes \DD^+_{x_k, 1} \to \R
  .\end{equation}
A choice of orientations $O_{x_k} \in \DD^\pm_{x_k, 1}$ for the
self-intersection points $x_k$ are {\em coherent} if the isomorphisms
\eqref{caniso} are orientation preserving with respect to the standard
orientation on $\R$.  Similarly for each critical point
$x \in \crit(m)$ let $W_x^\pm \subset L$ denote the stable and
unstable manifolds of $x$ under the flow of $-\grad(m) \in \Vect(L)$.
Choices of orientations $O_x$ on the determinant lines of the stable
resp.  unstable manifolds $\det(TW_{x}^\pm)$ are {\em coherent} if the
natural maps
\[\det(TW_x^-) \otimes \det(TW_x^-) \to \det(T_x L) \]
are orientation preserving.  Set
\[ \DD^+_{x_k,2} = \det(T_{x_k}W^+(x_k)) , \quad 
\DD^-_{x_k,2} = \det(T_{x_k}W^-(x_k))  \otimes \det( T_{x_k} L) \]
so that their tensor product is canonically trivial.  Given a relative
spin structure for $\phi:L \to X$ the orientation at $u$ is determined
by an isomorphism \label{familytransversality2}
 \begin{equation} \label{split2} \det(D_{{u}}) \cong \DD^+_{{x}_0,j_0}
   \otimes \DD^-_{x_1,j_1} \otimes \ldots \otimes \DD^-_{x_d,j_d},
 \end{equation}
 where each $j_k\in\{1,2\}$ depending on the type of end associated to
 $x_k$.  The isomorphism \eqref{split2} is determined by degenerating
 the surface with strip-like ends to a nodal surface. Thus each end
 $\eps_e, e \in \mE(S_v)$ of a component $S_v$ with a node $q_k$
 mapping to a self-intersection point is replaced by a disk
 $S_{v^\pm(k)}$ with one end attached to the rest of the surface by a
 node $q_k^\pm$.  After combining the orientations on the determinant
 lines on $S_{v^\pm(k)}$ with coherent orientations on the tangent
 spaces to the stable manifolds $W_{x_k}^\pm$ in the case of broken
 edges or semi-infinite edges $e \in \Edge(\Gamma), \ell(e) = \infty$,
 one obtains an orientation on the determinant line of the
 parameterized operator $\det(\ti{D}_u)$ and so orientations on the
 regularized moduli spaces $\M_\Gamma(\phi,D)$.  In particular the
 rigid component $\M(\phi,D)_\rho$ of the moduli space (that it, the
 component of expected dimension zero) inherits an orientation map
\[ \eps:  \M(\phi,D)_\rho \to \{ +1, -1 \} \]
comparing the constructed orientation to the canonical orientation of
a point.  

In general, Cieliebak-Mohnke perturbations \cite{cm:trans} are not
sufficient for achieving transversality if there are multiple markings
on ghost bubbles. Indeed, if there exists a sphere component
$S_v \subset S, v \in \Ver_{\black}(\Gamma)$ on which the map
$u |_{S_v}$ is constant and maps to the divisor so that
$u(S_v) \subset D$, the domain $S_v$ may meet any number of interior
leaves $T_e \subset T$.  Adding an bulk leaf $T_{e'}$ to the tree
meeting $S_v$ increases the dimension of a stratum
$\dim \M_\Gamma(\phi,D)$, but leaves the expected dimension
$\Ind(D_u), u \in \M_\Gamma(\phi,D)$ unchanged so $\M_\Gamma(\phi,D)$
is not of expected dimension for some types $\Gamma$.  A similar
phenomenon occurs in the immersed case for constant disks
$u|S_v: S_V \to X$ mapping to the self-intersection points $x,\ol{x}$
of the Lagrangian $\phi: L \to X$.  We call a holomorphic treed disk
$u: C \to X$ {\em uncrowded} if each such ghost component
$S_v \subset S$ meets at most one bulk leaf,
$T_e \cap S_v \neq \emptyset$. Cieliebak-Mohnke perturbations
$\ul{P} = (P_\Gamma)$ suffice to make strata for uncrowded types of
expected dimension at most one regular.  On the other hand, the
perturbations $P_\Gamma$ are chosen to satisfying the (Locality Axiom)
so that for any crowded component $\M_\Gamma(\phi)$ there exists a
non-empty uncrowded component $\M_{\Gamma'}(\phi)$ of the moduli space
obtained by forgetting all but one of the interior markings on such
ghost bubbles $S_v$.  Since combinatorial types $\Gamma'$ with sphere
bubbles are codimension two, such configurations $u: C \to X$ do not
appear in the components of the moduli space $\ol{\M}(\phi,D)$ of
expected dimension at most one.

The construction of regular perturbations now proceeds inductively by
combinatorial type.  More precisely, suppose that $d,e \ge 0$ are
integers.  Let $\Gamma$ be an uncrowded type of stable treed disk of
expected dimension at most one with $d$ incoming edges and $e$ edges
in total, and suppose that regular, stabilizing perturbation data
$P_{\Gamma'}$ have been chosen for all uncrowded boundary strata with
$d' \leq d$ incoming edges and $e' \ge e$ total edges, with
$(d',e') \neq (d,e)$. In particular, perturbations have been chosen
for any boundary stratum $ {\U}_{\Gamma'} \subset \ol{\U}_\Gamma$.  By
the gluing construction, we obtain regular perturbation data on a
neighborhood $V_\Gamma$ of the boundary of $\ol{\U}_\Gamma$.  Denote
by $\PP_\Gamma = \{ P_\Gamma \} $ the Baire space of perturbations
that agree with the given perturbations on $V_\Gamma$.

\begin{theorem} \label{comeager} Let $\Gamma$ be a combinatorial type
  of adapted holomorphic treed disk of expected dimension at most one.
  There exists a comeager subset $\PP^{\reg}_\Gamma$ of the space
  $\PP_\Gamma$ such that for $P_\Gamma \in \PP^{\reg}_\Gamma$,
\begin{enumerate} 
\item 
{\rm (Transversality)} every element of $\M_\Gamma(\phi,D)$ is
  regular;
\item \label{compactness} {\rm (Compactness)} the closure
  $\ol{\M}_\Gamma(\phi,D)$ is compact and contained in the adapted,
  uncrowded locus;
\item \label{bdes} {\rm (Boundary description)} The boundary of
  $\ol{\M}_\Gamma(\phi,D)$ is a union of components
  $\M_{\Gamma'}(\phi,D)$ where either $\Gamma'$ is a type with an edge
  of a length zero, a broken edge connecting two disk components, or a
  broken semi-infinite edge corresponding to bubbling off a Morse
  trajectory;
\item \label{tubular} {\rm (Tubular neighborhoods)} each uncrowded
  stratum $\M_\Gamma(\phi,D)$ of dimension zero has a tubular
  neighborhood of dimension one in any adjoining uncrowded strata of
  one higher dimension;
\item {\rm (Orientations)} the uncrowded strata $\M_\Gamma(\phi,D)$ of
  formal dimension one or two are equipped with orientations
  satisfying the standard gluing signs for inclusions of boundary
  components; in particular denote by $\eps(u) \in \{ \pm 1 \}$ the
  orientation sign associated to the zero-dimensional moduli spaces
  $\M(\phi,D)_\rho$.
\end{enumerate} 
\end{theorem}

We will not give a complete proof, since the arguments are very
similar to those for the embedded case as in \cite{cm:trans},
\cite{flips}, but rather sketch the arguments.  The transversality for
generic perturbations is a consequence of the Sard-Smale theorem.
Given integers $k,p$ determining the Sobolev class as in the proof of
Proposition \ref{prop:cutout} and some local trivialization of the
universal curve $\U_\Gamma^i \to \M_\Gamma^i$ denote by
$\M_\Gamma^{\univ,i}(\phi,D)$ the  {\em universal moduli space} 
\begin{multline} \M_\Gamma^{\univ,i}(\phi,D) = \{ (C,u,\partial
  u,u_1,u_2, P_\Gamma) \\ \in \M_\Gamma^i \times \Map(S^\circ,X)_{k,p}
  \times \Map(\partial S^\circ,L)_{k,p} \times \Map(T_1,L)_{k,p} \\
  \times \Map(T_2,L \times_\phi L - \Delta_L)_{k,p} \times
  \PP_\Gamma^l(X,D) | (*) \} \end{multline}
where $l \gg k$, $(*)$ includes the conditions \eqref{conds} together
with the requirement that the interior edges $T_e$ map to the
Donaldson hypersurface $D$; in the case with tangencies with the
Donaldson hypersurface there are additional conditions at the nodes
$w_e$ requiring given order of vanishing at $D$ as explained in
\cite{cm:trans}, \cite{flips}.  On a point $p$ in any non-constant
component $S_v \subset S, \d u | S_v \neq 0$, orthogonality of any
element $\eta$ of the cokernel $\ker(D_u^*)$ to variations of the
almost complex structure $K \in T_{J_\Gamma} \JJ_\Gamma(J_D)$ implies
that the cokernel element $\eta$ vanishes in an open neighborhood of
$p$, and so everywhere on $S_v$ by unique continuation.  Constant
spheres $S_v, \d u | S_v = 0$ have surjective linearized operators
$D_{u | S_v}$ since the trivial Cauchy-Riemann operator on the sphere
is surjective.  Finally constant disks $u|S_v: S_V \to X$ mapping with
corners alternating $x,\ol{x}$ for some ordered self-intersection
point $x \in X$ have surjective universal linearized operator since
any cokernel element $\eta$ must be orthogonal to a variation of the
Hamiltonian perturbation on an open subset, and so vanish again by
unique continuation.  The Sard-Smale theorem implies that the set
$\PP_\Gamma^{l,\reg,i}$ of regular values of the projection onto
$\PP_\Gamma^{l,i}$ is comeager; we denote by
\[ \PP_\Gamma^{\reg} = \cap_i \PP_\Gamma^{\reg,i} \]
so that if $P_\Gamma \in \PP_\Gamma$ then every element of the moduli
space $\M_\Gamma(\phi,D)$ is regular.  A standard argument due to
Taubes implies that for a comeager subset of the space of {\em smooth}
perturbations the moduli space is transversally cut out.  In
particular, let $\Gamma$ be a type of expected dimension at most one
and $\Gamma'$ the type obtained by adding a tangency requirement at a
marking or node.  Then the moduli space $\M_{\Gamma'}(\phi,D)$ is
empty.  On the other hand, $\Gamma$ and $\Gamma'$ have the same
underlying tree so by taking $P_\Gamma \in \PP_{\Gamma'}$ we may
assume that every element of $\M_\Gamma(\phi,D)$ meets the Donaldson
hypersurface $D$ transversally.

We sketch the proof of the compactness statement \eqref{compactness}.
The (Locality axiom) ensures that the moduli spaces
$\M_\Gamma(\phi,D)$ corresponding to types $\Gamma$ with more than one
interior marking $z_i',z_k'$ on a ghost bubble
$S_v \subset S, u_*[S_v] = 0$ admit forgetful maps to moduli space
$\M_{\Gamma'}(\phi,D)$ with at most one ghost marking on each sphere
bubble.  Suppose the boundary of $\M_\Gamma(\phi,D)$ contains a
configuration in a stratum $\M_{\Gamma'}(\phi,D)$ containing a sphere
bubble.  If the sphere bubble is non-trivial, then the expected
dimension of $\M_{\Gamma'}(\phi,D)$ is at most $-1$, and so empty.
Thus all sphere bubbles occurring in the limiting configuration are
ghost bubbles.  On the other hand, any configuration of ghost bubbles
has at least two markings to be stable, and so the configuration
$\M_{\Gamma'}(\phi,D)$ has intersection multiplicity $d(z)$ at least
two with the Donaldson hypersurface at at least one point $z \in S$.
But then $\M_{\Gamma'}(\phi,D)$ is of negative expected dimension, and
so empty.  Thus the boundary of the locus of expected dimension at
most one is the union of strata $\M_{\Gamma'}(\phi,D)$ where $\Gamma$
is a stable type or a Morse trajectory has bubbled off.

To prove \eqref{tubular}, each stratum for combinatorial type
representing a breaking has a tubular neighborhood given by a gluing
construction.  In the case of immersed Floer theory with treed disks
there are now three kinds of gluing constructions necessary for the
construction of tubular neighborhoods.  A detailed exposition of the
necessary estimates may be found in Schwarz \cite{sch:morse} for the
Morse gluing, Abouzaid \cite{abouzaid:exotic} for the disk gluing
estimates; also see Schm\"aschke \cite{clean}.  First, if two disks
$S_{v_k}, v_k \in \Ver(\Gamma), k \in \{ 1, 2 \}$ are joined at a node
corresponding to an edge $e \in \Edge(\Gamma)$ of length $\ell(e) = 0$
in a treed holomorphic disk $u: C \to X$ at which the branch of $L$ is
unchanged then a gluing procedure for holomorphic disks produces a
one-dimensional family of treed holomorphic disks
$u_\varkappa: C_\varkappa \to X$ depending on a gluing parameter
$\varkappa \in [0, \eps)$ for some small $\eps > 0$.  The gluing $u_\varkappa$
converges to the given configuration $u$ as $\varkappa \to 0$, and $C_\varkappa$
is obtained from $C$ by replacing the adjacent disks $S_{v'}, S_{v''}$
with a single disk $S_v'$.  Second, if the branch of $L$ changes at
the node $q_k$ so that $q_k$ maps to a self-intersection point
$x_k \in \cI^{\on{si}}(\phi)$ then a gluing construction for holomorphic
strips with transversely intersecting Lagrangian boundary conditions
produces a family $u_\varkappa: C_\varkappa \to X$ where again $C_\varkappa$ is
obtained from $C$ by replacing adjacent disks $S_{v'}, S_{v''}$ with a
single disk $S_v'$. Finally, for edges $e$ with infinite length
$\ell(e) = \infty$ the gluing construction for Morse trajectories
produces a family $u_\varkappa: C_\varkappa \to X$ converging to infinity, where
$C_\varkappa$ is obtained from $C$ by replacing a broken segment
$T_e = T_{e,1} \cup T_{e,2}$ with an unbroken one $T_e'$.

The boundary description in Theorem \ref{comeager} \eqref{bdes}
follows from the description of the tubular neighborhoods
\eqref{tubular} implies a description for the topological boundary of
the union of one-dimensional strata.  Any top-dimensional stratum
$\ol{\M}_\Gamma(\phi)$ has boundary strata $\ol{\M}_{\Gamma'}(\phi)$
corresponding (potentially) to Morse trajectories of length zero or
broken Morse trajectories $u: T_e \to L$ of length $\ell(e)$ infinity.
The former strata $\M_{\Gamma'}(\phi,D)$ are called {\em fake}
boundary components since there are two ways of desingularizing a
configuration $u \in \M_{\Gamma'}(\phi,D)$: either by gluing the two
adjacent holomorphic disks $u_i:S_i \to X, u_k:S_k \to X$ or by
deforming a zero length Morse trajectory so that $u_i,u_k$ are
connected by a segment $T_e$ of length $\eps> 0 $ instead of a node.
On the other hand, the strata with a broken segment
$u : T_e \to L, \ell(e) = \infty$ are {\em true} boundary components
of the one-dimensional component of the moduli space
$\cup_{i=1}^k \M_{\Gamma_i}(\phi,D)_1$ since in this case there is a
single way of moving into the interior, by making the length $\ell(e)$
of that segment $T_e$ finite, see Figures \ref{true} and \ref{fig:fake}.
\begin{figure}
\centering
\includegraphics[width=4in]{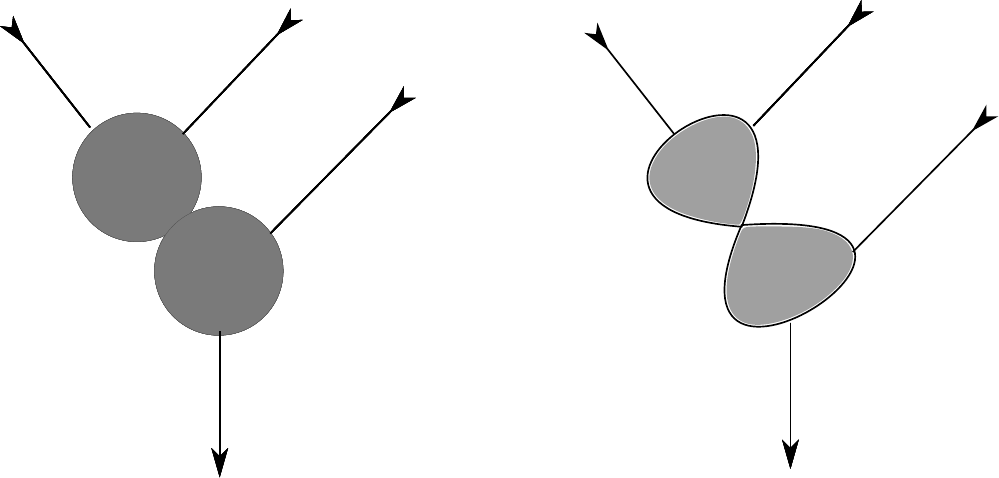}
\caption{Fake boundary components with an unbranched resp. branched node}
\label{fig:fake} 
\end{figure} 
\begin{figure}
\centering
\includegraphics[width=4in]{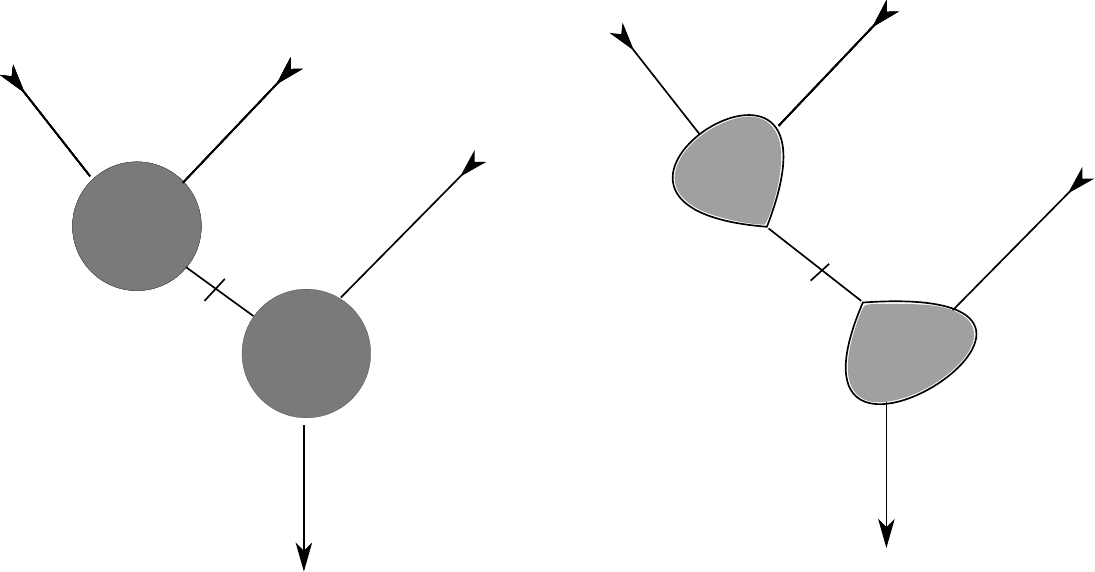}
\caption{True boundary components with an unbranched resp. branched segment}
\label{true} 
\end{figure} 
Because we chose the trajectories on the branched edges to be
constant, for the unperturbed almost complex structure $J_D$ and Morse
function $m$, once one has a fake boundary component $M_\Gamma(\phi)$
where a holomorphic disk breaks at a self-transverse intersection, it
also contains a true boundary component $\M_{\Gamma'}(\phi)$ with an
infinite edge mapping to the same self-transverse intersection, since
configurations with every possible edge length at the
self-intersection occur.  However, for the perturbed almost complex
structure and Morse function $J_\Gamma,m_\Gamma$ the perturbation
depends on the length $\ell(e)$ of this trajectory, and so the
existence of a non-empty stratum $\M_\Gamma(\phi)$ of a branched node
with zero length $\ell(e) = 0$ does not necessarily imply the
existence of a non-empty stratum $\M_{\Gamma'}(\phi)$ with infinite
length.

\begin{remark} {\rm (The exact case)}  
  An alternative construction of moduli spaces in the case of exact
  symplectic manifolds gives a definition of immersed Fukaya algebras
  over the integers.  Let $\Gamma$ be a stable combinatorial type of
  nodal marked disks $S$ with no sphere components, with {\em ordered}
  boundary markings $z_0,\ldots, z_d \in \partial S$ in
  counterclockwise order around the boundary and {\em unordered,
    possibly coinciding} interior markings $z_1',\ldots, z_c'$.  That
  is, two such disks $S_1,S_2$ are isomorphic if there exists a
  biholomorphism $S_1 \to S_2$ preserving the boundary marking
  $\ul{z}$ and preserving the interior markings $\ul{z}'$ up to
  reordering.  Let $\M^{\on{unord}}_\Gamma$ denote the moduli space of
  isomorphism classes of such disks, and $\ol{\M}^{\on{unord}}_{d,c}$ the union
  over combinatorial types $\Gamma$ with $d$ incoming leaves and $c$
  interior markings:
\[ \ol{\M}_{d,c}^{\on{unord}}
  = \bigcup_\Gamma \M^{\on{unord}}_\Gamma.\]
  Each moduli space $\M^{\on{unord}}_\Gamma$ is contractible, being
  identifiable with a subset of increasing elements
  $ 0 = z_1 < \ldots < z_{d-1} = 1$ of the boundary markings and a
  subset $z'_1,\ldots, z'_c \in \on{int}(\H)$ of the half-plane
  corresponding to the interior markings. The union
  $\ol{\M}^{\on{unord}}$ has a natural topology, similar to the moduli
  space of stable nodal disks $\ol{\M}_{d,c}$ with interior markings;
  in fact the latter admits a natural map to $\ol{\M}^{\on{unord}}_{d,c}$
  obtained by collapsing all sphere components.  Each moduli space
  $\M^{\on{unord}}_\Gamma$ is a smooth manifold equipped with a {\em
    universal disk bundle}
\[ \U^{\on{unord}}_\Gamma \to \M^{\on{unord}}_\Gamma .\]
Charts for $\M^{\on{unord}}_\Gamma$ near points $(C,\ul{z})$ where
some of the markings coincide are produced using the algebra
isomorphism $\C[z_1,\dots, z_n]^{\Sigma_n} \cong \C[p_1,\ldots,p_n]$
where $\Sigma_n$ denotes the symmetric group on $ \{ 1, \ldots, n \}$
and $p_1,\ldots, p_n$ are the elementary symmetric polynomials; taking
spectra produces local charts for $(\bC^n)/\bZ_n$. Since each
$\M^{\on{unord}}_\Gamma$ is contractible, the bundles $\U_\Gamma$ are
trivial.  Fix a trivialization $C \times \M^{\on{unord}}_\Gamma$.
(The existence of a global trivialization is not necessary; if the
bundle were only locally trivial we could perform the construction in
each local trivialization and take the intersection of the comeager
sets of perturbations.) Given a domain-dependent almost complex
structure and Morse function $P_\Gamma = (J_\Gamma,m_\Gamma)$, let
$J_\Gamma(C),m_\Gamma(C)$ denote the induced domain-dependent almost
complex structure and Morse function on $S,T \subset C$.  Denote the
moduli space of adapted treed disks by 
\begin{multline} \M^{\on{unord}}_\Gamma(\phi,D) = \left\{ 
\begin{array}{l} 
  C,\ul{z}, u_S: S \to X, \partial u_S: \partial S \to L, \\
  u_{T_1}: T_1 \to L , u_{T_2} :T_2 \to L \times_\phi L \
  \\  \olp_{J_\Gamma,H_\Gamma} u_S = 0, \left(\dds +
  \grad(m_\Gamma) \right) u_{T_2} = 0, \ u^{-1}(D) = \ul{z}
  \\  (\partial u_S)_{ \partial S \cap T_1} = \phi u_{T_1}
  |_{\partial S \cap T_1} , (\partial u_S)_{ \partial S \cap T_2} = \phi
  u_{T_1} |_{\partial S \cap T_2} \end{array} \right\}. \end{multline}
The same arguments that prove Theorem \ref{comeager} show that for a
comeager subset of perturbations $P_\Gamma$ the moduli space
$\M^{\on{unord}}_\Gamma(\phi,D)$ is cut out transversally; for
recursively defined perturbation data the locus of the moduli space
$\ol{\M}^{\on{unord},< E}(\phi,D)$ of energy at most $E$ is compact.
This ends the Remark.
\end{remark}

\section{Holomorphic disks with self-tangent boundary condition}
\label{stbc}

In this section we construct moduli spaces of holomorphic curves for
an immersed Lagrangian with a non-degenerate self-tangency.  For
simplicity we suppose that we are given a Lagrangian immersion
$\phi: L \to X$ such that self-intersection points
$L \times_\phi L - \Delta_L$ are transverse except for a single (up
to reordering) self-intersection
\[(x_0,x_1) \in L \times L, \quad \phi(x_0) = \phi(x_1) = x\] 
that has a non-degenerate self-tangency.  In particular the dimension
of the intersection of tangent spaces satisfies
\[
\dim ( (D_{x_0} \phi (T_{x_0} L)) \cap (D_{x_1} \phi (T_{x_1} L)) ) = 1
.\]
We assume that there exist local coordinates
$q_1,\ldots,q_n,p_1,\ldots,p_n$ on $X$ near the self-intersection
point $x$ with $x$ as the origin:
\[ q_1(x)  = \ldots  = q_n(x)  = p_1(x) = \ldots = p_n(x)  = 0 .\]
We also assume that the branches $L_1,L_2$ of $\phi(L)$ meeting $x$
are a plane and the product of a circular arc with a plane:
\begin{equation} \label{arcproduct}  L_1 = \{ p_1 = \ldots = p_n = 0 \}  \quad L_2 =  \{ p_1 = (1-
  q_1^2)^{1/2} + 1, q_2 = \ldots = q_n = 0 \} .\end{equation} 
Lemma \ref{admissible} shows that we may assume these conditions in
any generic Maslov flow.  We say that $x$ is a {\em standard
  self-tangency} if there exist such local coordinates and the almost
complex structure $J$ is chosen to be the standard complex structure
in these coordinates in a neighborhood of the self-tangency.

Once self-tangencies are allowed, one loses the exponential decay
property at the self-intersections.  Let $\Gamma$ be a combinatorial
type of treed disk with immersed Lagrangian boundary condition and let
$w \in C$ be a point at which a branched segment in $T_2$ meets the
surface part, so that $w \in T_2 \cap S$.  Let
\[ \eps_{w}: [0, \pm \infty) \times [0,1] \to S  \] 
be a holomorphic local coordinate on $S$ (with sign depending on
whether the end is incoming or outgoing) with limit
\[ \lim_{\tau \to \pm \infty} \eps_w(\tau,t) = w, \quad \forall t \in [0,1]
. \] 
Denote by $E_w \subset S$ the image of $[1,\infty) \times [0,1]$ under
$\eps_w$.  Denote by
\[E_w[\tau_0] := \eps_w( (\tau_0, \infty) \times [0,1]) \subset S \]
the part of the end with $\tau > \tau_0$.

\begin{lemma} \label{expdecay} {\rm (Exponential decay at
    self-intersections, similar to Proposition 4.6 in \cite{ees})} Let
  $u: C \to X$ be a holomorphic disk with boundary on a Lagrangian
  immersion $\phi: L \to X$ such that a point $z \in C$ maps to a
  self-intersection $x = \phi(x_0) = \phi(x_1), (x_0,x_1) \in L^2$.
  \begin{enumerate}
  \item \label{texpdecay} If the self intersection $x$ is transverse then there exists a
    local coordinate $\eps_w(\tau + it)$ near $z$ 
    so that if $\tau > \tau_0$ then (omitting $\eps_w$) for some
    $\theta > 0$, 
  \[ | u(\tau + it)| = O(e^{- \theta \tau}) ;\]
\item If the self-intersection $x$ is a standard self-tangency then
  either
\[ | u(\tau + it)| = O( e^{- \theta \tau}) \] 
as in the previous item or there exists a real number $c_0$ such that
if $\tau > \tau_0$ then up to involution
\begin{equation} \label{leftorright} u(\tau + it) = \left( \frac{2}{c_0 + \tau + it}, 0,\ldots, 0 \right) +
O(e^{-\theta \tau}) \end{equation} 
    where $O(e^{-\theta \tau})$ indicates exponential decay as in the
    first item.
  \end{enumerate}
\end{lemma} 

\begin{proof} The first item is Floer's exponential decay estimate
  \cite[Proof of Theorem 4]{floer:unreg}, see also \cite{floer}.  The
  second is also considered in Floer \cite[Lemma 4.2]{floer:unreg}, in
  the context of a study of the Fredholm index of the linearized
  operator; see also Ekholm-Etnyre-Sullivan \cite[Lemma 4.6]{ees}.
  The proof of the statement here is, as in Floer \cite{floer:unreg},
  simpler than the treatment in Ekholm-Etnyre-Sullivan \cite{ees},
  since we do not perturb the boundary condition $\phi(L)$ near the
  self-tangency $x$.  The self-tangent boundary conditions in $\C$
  given in the first coordinate in \eqref{arcproduct} may be
  transformed to affine linear boundary conditions via the conformal
  transformation $u \mapsto 1/u$.  
That is, the map $v(z) = 1/u_1(z)$
  is holomorphic and has boundary conditions given by
  $ \on{\Im}(v)(\tau) = 0, \on{\Im}(v) (\tau + i) = 1/2$.  Composing 
  with the holomorphic map $v \mapsto \exp( - 2 \pi v )$ we obtain a 
  finite energy holomorphic map with boundary conditions $\R$. 
 By 
  definition $v(z)$ has a limit as $\tau \to 0$ in the local 
  coordinate $z = \exp( \pi (\tau + i t))$,
 and must be given by a map 
  of the form 
  \[ \exp(- 2 \pi v(z)) = \pm \exp(2c_0) z + O(z^2) . \]
Assuming the positive sign we have 
\begin{eqnarray*}
 u_1(t,\tau) &=&  v(t,\tau)^{-1} \\
&=&  \pm ( \ln( \exp(    (c_0+t + i \tau)/2 + O(e^{-  \tau}))))^{-1} 
\\
u(t,\tau) &=& \left( \frac{  2}{c_0 + \tau + it}, 0,\ldots, 0 \right) +
O(e^{-\theta \tau})  
\end{eqnarray*} 
(where $\theta$ depends on the K\"ahler angles of the transverse
components) defines a holomorphic map from a neighborhood of $0$ in
the half plane $\on{Im}(z) \ge 0$ to $\C$ with boundary in $\R$ and
has the claimed expansion.
\end{proof}

\begin{proposition} If $\phi: L \to X$ has only non-degenerate
  tangencies then near any $u: C \to X$ of combinatorial type $\Gamma$
  the moduli space $\M_\Gamma(\phi)$ is cut out locally by a Fredholm
  map.
\end{proposition}

\begin{proof} The statement of the proposition is essentially that in
  Floer \cite[Theorem 4a]{floer:unreg}, who studies non-degenerate
  tangencies as a method for developing an index theorem for the
  transverse case, or Ekholm-Etnyre-Sullivan \cite[Corollary
  7.14]{ees}.  For simplicity assume a unique self-tangent point at
  $y \in X$. Let $C$ be a tree disk of type given by $\Gamma$ and let
  $q_1,\ldots, q_m \in C$ be the set of nodes corresponding to maps to
  the self-tangent point. We assume that the universal curve
  $\cU_\Gamma \to \cM_\Gamma$ is equipped with a smooth varying system
  of local coordinates $\eps_k: [0,1] \times \bR \to S$ near the nodes
  $q_k$ for any fiber $C \subset \cU_\Gamma$.  Given
  $\ul{c} = (c_1,\ldots, c_m) \in \bR^m$, fix a smooth {\em reference
    map} of the form
\begin{equation} \label{refmap} u_{\on{ref}}^{\ul{c}}: V \to X, \quad
  z \mapsto \left( -2 (z + c_k)^{-1}, 0,\ldots, 0
  \right) \end{equation}
near the node $q_k$. Choose a time-dependent exponential map on the
ends $\exp_{u,t}: T_x X \to X$ so that
\[\exp_{u,t=0}: T_{x} L_1 \to L_1, \quad \exp_{u,t=1}: T_{x} L_2(0) \to
L_2(0)\] 
in a neighborhood of the points $x_0,x_1 \in L$ mapping to $x \in X$.
For $p \ge 2, k \ge 1, kp > 2 $ and real number $\lambda$ let
$\Omega(C,u^*TX)_{k,p,\lambda}$ be the Sobolev space of $k,p$ maps
defined using the measure $e^{ \lambda \tau}$ on the strip-like end in
the coordinates $z = \tau + it$.  Denote by
$\Map^\Gamma_{k,p,\lambda}(C,X)$ be the space of maps $u$ of
combinatorial type $\Gamma$ of Sobolev class $W^{k,p}_{\on{loc}}$ of
the form $\exp_{u^{\ul{c}}_{\on{ref}}}(\xi)$ for
$\xi \in \Omega^0(C, u^* TX)_{k,p,\lambda}$ for some reference map
$u^{\ul{c}}_{\on{ref}}$. The class of maps that exponentially decay to
one of these reference maps is independent of choice of reference map
and local coordinates on the ends. Locally in a neighborhood of a map
$(u,\partial u): (C, \partial C) \to (X,L)$ the moduli space
$\M_\Gamma(\phi,D)$ is cut out by a Fredholm section of a Banach
vector bundle.  Given some local trivialization
$\U_\Gamma^i \to \M_\Gamma^i \times C$ of the universal treed disk
there is a bundle with base
\begin{multline} \B^i_\Gamma = \left\{ \begin{array}{l} (C,u, \partial
      u,u_{T_1},u_{T_2}) \in \left( \begin{array}{l} \M^i_\Gamma \times \Map_{k,p,\lambda}(S^\circ, X)
      \times \Map_{k,p,\lambda}(\partial S^\circ, L)\\  \times
                                       \Map_{k,p}(T_1,L) \times
                              \Map_{k,p}(T_2,L)  \end{array} \right) \\
                                       \phi \circ \partial u_S = u_S
                                       |_{\partial S}, \quad 
\partial u |_{ \partial S^\circ 
        \cap T_1} =  u_{T_1} |_{S \cap T_1} , \quad 
\partial u |_{ \partial S^\circ 
        \cap T_1} =  u_{T_1} |_{S \cap T_1} 
\end{array}
  \right\} \end{multline}
of maps of Sobolev class $k,p$ with Sobolev weight $\lambda$ together
with a treed disk of type $\Gamma$.  Local charts for $\B_\Gamma$ can
be constructed using geodesic exponentiation for some metric for which
$L$ is totally geodesic; such a metric exists as long as the
self-intersections of $L$ are self-transverse.  We suppose for
simplicity of notation that there is a single self-tangent node $q_k$
with $u(q_k) = x$.  Consider the map 
\begin{multline} \label{extra} T_c \bR \times \Omega^0( u^{c,*} TX,
  (\partial u^{c,*} TL))_{k,p,\lambda} \to \Map_{k,p,\lambda}(S,X) \\
  (\delta c,\xi) \mapsto \exp_{u^{c + \delta c}}(\Pi_{\delta
    c}\xi) \end{multline}
where 
\[\Pi_{\delta c}: 
 \Omega^0( u_{c,*} TX, 
 (\partial u^{c,*} TL))_{k,p,\lambda} \to
 \Omega^0( u_{c+\delta c,*} TX, 
 (\partial u^{c + \delta c,*} TL))_{k,p,\lambda} \]
 denotes parallel transport along the path of reference maps
 $u^{c + t \delta c}$.  The fiber of the bundle $\cE_\Gamma$ over some
 map $u$ is the vector space
\[ \cE_{\Gamma,u}^i :=  \Omega^{0,1}(S^\circ, u_S^* TX)_{k-1,p} \oplus 
\Omega^1(T_1, u_{T_1}^* TL)_{k-1,p} \oplus \Omega^1(T_2, u_{T_2}^*
TL^2)_{k-1,p} .\]
Local trivializations may be constructed using Hermitian parallel
transport using the metric used to construct the charts for the base.
The Fredholm map cutting out the moduli space locally is then
\begin{equation} \label{banachmap} \cF^i_\Gamma: \B^i_\Gamma \to
  \cE^i_\Gamma, (u_S,u_{T_1},u_{T_2}) \mapsto \left( \olp {u_S}, \dds u_{T_1} +
    \grad(m(u_{T_1})), \dds u_{T_2}  \right) ;\end{equation}
Here we use the assumption that the complex structure
$J_\Gamma: \cU_\Gamma \to J_\tau(X)$ is the standard one $J_0(z) = iz$
in a neighborhood of the self-tangency $x_0$ in the given local
coordinates.  This assumption implies that the one-form $\olp u_S$ has
the required exponential decay independent of the choice of reference
map.  The choice of Sobolev weight $\lambda \neq 0$ smaller than any
of the angles in \eqref{angles} guarantees that the linearized
operator $\ti{D}_u$ is Fredholm.  As before, after imposing a finite
energy condition $E(u) , E$ any holomorphic treed disk $u: C \to X$
has a well-defined limit $\lim_{s \to \infty} u( \eps_k(s,t))$ along
any strip-like end $\eps_k$ converging to a node $q_k$ mapping to a
self-intersection point $x_0 \in \phi( L \times_\phi L - \Delta_L)$.
The required moduli space $\M_\Gamma(\phi)$ is then the subset of the
zero set $\cF^{i,-1}_\Gamma(0)$ of the resulting Fredholm map
satisfying the matching condition at the nodes (quotiented by the
group of automorphisms of $C$, if the domain $C$ is unstable.)
\end{proof}

The moduli spaces with self-tangent boundary can be regularized as
before using stabilizing divisors.  Using Lemma \ref{dexists} choose a
Donaldson hypersurface $D \subset X$ so that $\phi(L) \subset X$ is
exact in $X - D$.  Requiring the bulk leaves $T_e$ to map to $D$,
and any component of $u^{-1}(D)$ containing an bulk leaf, we may
assume that the domains have stable surface part $S \subset T$ and use
domain-dependent almost complex structures and Morse functions.  As in
Theorem \ref{comeager}, for a comeager subset $\PP^\reg_{\Gamma,i} $
of domain-dependent perturbations the moduli space of treed
holomorphic disks $u: C \to X$ is transversally cut out for fibers
$C \subset \cU^i$ in a local trivialization
$\cU^i_\Gamma \to \cM^i_\Gamma$ of the universal curve; taking the
finite intersection $\PP_{\Gamma}^\reg = \cap_i \PP^{\reg}_{i,\Gamma}$
of these comeager sets produces the desired comeager subset of regular
perturbations.

There is also a moduli space denoted $\M_\Gamma^e(\phi)$ with
``exponential decay'' at the self-tangency, which will be empty for
regular perturbations. The moduli space $\M_\Gamma^e(\phi)$ is
constructed in the same way, but using a constant reference function
$u^{\on{ref}}(z) = \phi({y})$ near the self-tangency nodes $q_k$
instead of the reference function in \eqref{refmap}. For generic
choices of perturbation data $P_\Gamma$, the exponential decay moduli
spaces $\M_\Gamma^e(\phi)$ are, by construction, of lower dimension
than the moduli spaces $\M_\Gamma(\phi)$ that allow the reference
function \eqref{refmap}.  Thus,

\begin{proposition} {\rm (Similar to \cite[Remark 9.20]{ees})} Let
  $\Gamma^e$ denote a combinatorial type of treed adapted holomorphic
  disk with exponential decay at a tangency.  Let $\Gamma$ be the
  corresponding combinatorial type of treed holomorphic disk without
  the self-tangency requirement.  If $\M_\Gamma(\phi,D)$ has expected
  dimension zero, then for regular perturbation data the moduli space
  $\M^e_{\Gamma^e}(\phi,D)$ is empty.
\end{proposition}

Given a relative spin structure for the self-tangent immersion,
orientations for the moduli spaces $\ol{\M}(\phi,D)$ may be
constructed as in Fukaya-Oh-Ohta-Ono \cite{fooo}, see also
Wehrheim-Woodward \cite{orient}, by bubbling off boundary value
problems on disks with given paths $\gamma_x: [0,1] \to \Lag(T_xX)$.
For $x = {y} \in X$ the self-tangent intersection point of
$\phi: L \to X$, the natural gluing of the boundary value problems
associated to the once-punctured disk with marking ${y}$ and boundary
condition $\gamma_y$ and a marking ${\ol{v}}$ with boundary condition
$\gamma_{\ol{v}}$ produces a boundary value problem on the disk with
index one (since the Sobolev spaces are defined using a small {\em
  negative} Sobolev weight and so (after stabilization) kernel
isomorphic to $T_{{y}} L \oplus \R$.  For the self-tangent
self-intersection, we require as part of the coherence axiom for
orientations that
$o_{{y}} \in \det(D_{{y}}), o_{{\ol{v}}} \in \det(D_{{\ol{v}}})$
should be defined so that the natural gluing map
\begin{equation} 
\label{tglue}
\DD_{{y},2}^+ \otimes \DD_{{y},2}^- \to  \R\end{equation} 
produces the standard orientation on the determinant of the
one-dimensional index of the boundary condition where
$\DD_{{y},2}^+ = \det(D_{{y},2})$ and
$\DD_{{y},2}^- = \det(D_{{y},2}^-) \otimes \det(T_{{y}}L)$.  The
determinant line in \eqref{tglue} is an {\em odd} determinant line in
the language of Deligne-Freed \cite{signs}, so that permuting it with
other determinant lines produces signs that will contribute to the
gluing sign for deformation of the Lagrangian boundary condition.

Compactness results follow from the discussion in
Etnyre-Ekholm-Sullivan \cite{ees}.  In particular \cite[Theorem
11.2]{ees} construct compactified {\em parametrized} moduli spaces of
holomorphic curves as the Lagrangian boundary condition develops a
self-tangency in an isotopy $\phi_t, t \in [-\eps,\eps]$.  Energy
quantization holds uniformly in $t$ in the following sense:

\begin{lemma} \label{hbar} There $\hbar > 0$ such that any non-constant holomorphic
  polygon $u: S \to X$ with boundary in $\phi_t$ {\em except} for
  those contained in a small neighborhood of the self-tangency, so
  that any such holomorphic polygon not contained in a small
  neighborhood $U$ of $y$ has energy at least $E(u) > \hbar$.  
\end{lemma}

\begin{proof} For holomorphic spheres or disks, the standard argument
  using the mean value inequality applies, see \cite[Proposition
  4.1.4]{ms:jh}.  On the other hand, since the distance between the
  branches of the Lagrangian is bounded from below away from the
  self-intersections, there exists a constant $c > 0$ so that holomorphic
  polygons with a corner not contained in a neighborhood of the
  self-tangency contain a sub-domain of the form
  $u: [0,\infty] \times [0,1] \to X$ with the property that
  $d(u(T,0), u(T,\infty)) > c$.  By the mean value inequality again,
  this is impossible if the energy is sufficiently small, since the
  derivative $\partial_t u ( T, t)$ is bounded by a constant times the
  energy.
\end{proof} 

The lack of energy quantization for holomorphic strips in a
neighborhood of the self-tangency invalidates the argument in Gromov
compactness \cite[Proposition 4.7.1]{ms:jh} used to prevent energy
loss, see especially \label{espec} Step 2 in the proof in \cite{ms:jh}
which uses energy quantization. However, in the neighbourhood $U$ the
branches $L_0,L_1$ of $\phi(L)$ are exact:
$$      \omega_U = \d \alpha, \quad 
\alpha |_{L_0} = \d \zeta_0, \quad \alpha_{L_1} = \d \zeta_1 .$$
Energy loss can be ruled out using Stokes' theorem: For a sequence of
holomorphic polygons $u_\nu: S_\nu \to X$ Gromov converging to
$u: S \to X$ in the complement of the self-tangency $y$, let
$ V_\nu = u_\nu^{-1}(U)$ and $(\partial V_\nu)_1$ resp.
$(\partial V_\nu)_2 \cong \sqcup_{i=1}^{e_\nu} [0,1]$ the intersection
of $\partial V_\nu$ with $\partial S$ resp. the closure of complement
of $\partial S$ in $\partial V_\nu$.  Write the intersection
$ (\partial V_\nu)_1 \cap (\partial V_\nu)_2$ as a sequence
$z_{0,1}, z_{1,1}, \ldots, z_{0,e}, z_{1,e}$ corresponding to the
endpoints of the intervals in $\sqcup_{i=1}^{e_\nu} [0,1]$.  Since
$u_\nu$ converges to $u$, the number of ends $e_\nu$ converges to some
$e$ for $\nu$ sufficiently large.  The energy
$ E(u_\nu | u_\nu^{-1}(U))$ is determined by
$$\int_{u_\nu^{-1}(U)} u^* \omega = \int_{(\partial V_\nu)_1} u^*
\alpha + \sum_{i=1}^{e_\nu} \zeta_1( z_{1,i}) - \zeta_0( z_{0,i}) $$
which limits to $ \int_{u^{-1}(U)} u^* \omega$.  It follows that there
is no energy loss at the self-tangency $y$ either.

Once one has excluded holomorphic maps with exponential decay to a
self-tangency point, the moduli space of maps with a node mapping to a
self-tangency splits as a disjoint union depending on whether the map
approaches the tangency from the $x_1$ negative or positive, according
to the sign in \eqref{leftorright}.  We introduce two new symbols
$v_\pm$ and let
\[ \cI^{\on{si}}(\phi)  =  ( L \times_{\phi} L - \Delta_L)  - \{y \} \cup
\{ v_+ , v_- \} .\]
Then for $\ul{x} \subset \cI^{\on{c}}(\phi)\cup \cI^{\on{si}}(\phi)$ we denote by
$\M(\phi,D,\ul{x})$ the subset of the moduli space
$\M_\Gamma(\phi,D,\ul{x})$ with the given limits $\ul{x}$, which
distinguishes between holomorphic maps approaching from $x_1$ negative
or positive in the local model.

We define a moduli space of holomorphic treed disks by imposing the
following condition at nodes: If $e \in \Edge(\Gamma)$ is an edge
corresponding to a segment $T_e$ of $T$ that maps to the self-tangency
in $\phi_0$, then the limits of the disks on either end of $T_e$ are
opposite, that is, if $z_{\pm}$ are the endpoints of the segment
$ \lim_{z \mapsto z_+} u(z) = v_+$ (that is, $u(z)$ approaches $y$
from the left in the local model) then
$ \lim_{z \mapsto z_-} u(z) = v_-$ and similarly if
$ \lim_{z \mapsto z_+} u(z) = v_-$ (that is, $u(z)$ approaches $y$
from the left in the local model) then
$ \lim_{z \mapsto z_-} u(z) = v_+$.  Thus, the holomorphic treed disk
``passes through'' the self-tangency.  We also allow the tangency
points $v_\pm$ as input or output.  In the case that $\phi_t: L \to X$
is a Maslov flow with a self-tangency at $t = 0$ generating two new
self-intersection points, the labels specifying the limits at infinity
are
\begin{equation} \label{new} \cI^{\on{si}}(\phi_0) = \cI^{\on{si}}(\phi_{-\eps}) \cup \{ v_-,v_+\}, \quad \cI^{\on{si}}(\phi_\eps) = \cI^{\on{si}}(\phi_{-\eps}) \cup \{ \zz, \bzz, \bz,
  \mz \} .\end{equation}
In other words, each direction of the self-tangency evolves into two
ordered self-tangencies.

\section{Morse model for immersed Floer cohomology}
\label{mmodel}

In this section we apply the transversality results of the previous
section to construct immersed Floer theory in the Morse model.  The
generators of the immersed Floer space are critical points of a Morse
function on the Lagrangian together with ordered self-intersection
points of the immersion: Let 
\[ m: L \to \bR \]  
be a Morse function. Define
\begin{eqnarray} 
  \cI^{\on{c}}(\phi) &=& \crit(m) = \{ x  \in L \ | \ \d m (x) = 0 \} \\ 
  \cI^{\on{si}}(\phi) &=& \{ (x_1,x_2) \in L^2  \ | \ \phi(x_1) = \phi(x_2),
                     \ x_1 \neq x_2 \} \\
 \label{selfint} \cI(\phi) &=& \cI^{\on{c}}(\phi) \cup \cI^{\on{si}}(\phi) \cup \{ x^{\whitet},
                x^{\greyt} \} .\end{eqnarray}
The set of self-intersections $\cI^{\on{si}}(\phi)$ has a natural  
involution obtained by reversing the ordered pair:
\[\cI^{\on{si}}(\phi) \to \cI^{\on{si}}(\phi), \ x= (x_1,x_2) \mapsto \ol{x} =
(x_2,x_1) .\]
Thus $\cI(\phi)$ consists of the critical points of the Morse function
together with two copies of each self-intersection point, plus two
extra generators $x^{\whitet}, x^{\greyt} \in \cI(\phi)$ used to
construct strict units for the Fukaya algebra.  We assume for
simplicity $L$ is connected and $m: L \to \R$ has a single maximum
$x_+ \in \cI^{\on{c}}(\phi)$.  The
degrees are determined by
\[ \deg(x_+) = \deg(x^{\whitet}) = 0, \quad \deg(x^{\greyt}) = -1 
.\]

\begin{example} \label{ex1} We consider an example of an immersion of
  the circle in the two-sphere with a pair of self-intersection
  points.  Let $X \cong S^2 \cong \bR^2 \cup \{ \infty \}$.  Consider
  the immersion of a circle $L \cong S^1$ as in Figure
  \ref{generators}.
\begin{figure} 
\begin{center}
%% Creator: Inkscape inkscape 0.91, www.inkscape.org
%% Accompanies image file 'imex.pdf' (pdf, eps, ps)
\begingroup%
  \makeatletter%
  \providecommand\color[2][]{%
    \errmessage{(Inkscape) Color is used for the text in Inkscape, but the package 'color.sty' is not loaded}%
    \renewcommand\color[2][]{}%
  }%
  \providecommand\transparent[1]{%
    \errmessage{(Inkscape) Transparency is used (non-zero) for the text in Inkscape, but the package 'transparent.sty' is not loaded}%
    \renewcommand\transparent[1]{}%
  }%
  \providecommand\rotatebox[2]{#2}%
  \ifx\svgwidth\undefined%
    \setlength{\unitlength}{74.03606701bp}%
    \ifx\svgscale\undefined%
      \relax%
    \else%
      \setlength{\unitlength}{\unitlength * \real{\svgscale}}%
    \fi%
  \else%
    \setlength{\unitlength}{\svgwidth}%
  \fi%
  \global\let\svgwidth\undefined%
  \global\let\sv%gscale\undefined%
  \makeatother%
  \begin{picture}(1,3.22487084)%
    \put(0,0){\includegraphics[width=\unitlength,page=1]{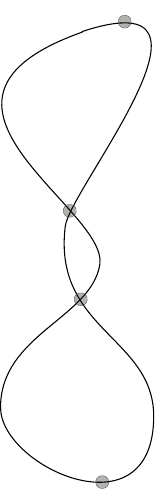}}%
    \put(2.06938101,2.55758581){\color[rgb]{0,0,0}\makebox(0,0)[lb]{\small \smash{}}}%
    \put(0.43701218,1.69888224){\color[rgb]{0,0,0}\makebox(0,0)[lb]{\small \smash{$\bzz$}}}%
    \put(0.74037698,3.15861453){\color[rgb]{0,0,0}\makebox(0,0)[lb]{\small \smash{$x_+$}}}%
    \put(0.75586417,-0.00000002){\color[rgb]{0,0,0}\makebox(0,0)[lb]{\small \smash{$x_-$}}}%
    \put(0.37449693,2.000991){\color[rgb]{0,0,0}\makebox(0,0)[lb]{\small \smash{$\zz$}}}%
    \put(0.46606091,1.40667391){\color[rgb]{0,0,0}\makebox(0,0)[lb]{\small \smash{$\bz$}}}%
    \put(0.46885573,1.11180859){\color[rgb]{0,0,0}\makebox(0,0)[lb]{\small \smash{$\mz$}}}%
    \put(0,0){\includegraphics[width=\unitlength,page=2]{imex.pdf}}%
  \end{picture}%
\endgroup%

  \end{center}
\caption{Generators for immersed Floer cohomology}
\label{generators}
\end{figure} 
Let $m:L \to \R$ be a Morse function with two critical points given by
the maximum $x_+$ and minimum $x_-$; the highest self-intersection
point contributes two elements $\zz, \bzz$ to $\cI(\phi)$ where the
ordering of the points in the self-intersection $L \times_\phi L$ is
given by the direction of the arrow in the Figure.  The second
self-intersection point contributes elements $\bz,\mz$
similarly. Thus
\begin{equation} \label{bzs}
\cI(\phi) = \cI^{\on{si}}(\phi) \cup \cI^{\on{c}}(\phi) \cup \{ x^{\whitet},
x^{\greyt} \} = \{ x_+ ,x_-, \zz,\bzz,\bz,\mz , x^{\whitet},
x^{\greyt} \} .\end{equation}
This ends the example. \end{example}

In order to obtain graded Floer cohomology groups a grading on the set
of generators is defined as follows.  Let $N$ be an even integer and
$\Lag^N(X) \to \Lag(X)$ an $N$-fold Maslov cover of the bundle of
Lagrangian subspaces as in Seidel \cite{se:gr}; we always assume that
the induced $2$-fold cover $\Lag^2(X) \to \Lag(X)$ is the bundle of
oriented Lagrangian subspaces.  A {\em grading} of $\phi:L \to X$ is a
lift $\phi^N$ to $\Lag^N(X)$ of the natural map
$L \to \Lag(X), x \mapsto \on{Im}(D_x\phi)$.  Given such a grading,
there is a natural $\bZ_N$-valued map
\[ \cI(\phi) \to \bZ_N, \quad x \mapsto |x| \]
obtained by assigning to any critical point the index mod $N$ and to
any self-intersection point $(x_-,x_+) \in L$ the Maslov index of the
image of any path from $\phi^N(x_-)$ to $\phi^N(x_+)$ in $\Lag_x(X)$.
Denote by $\cI^k(\phi)$ the subset of $x \in \cI(\phi)$ with
$|x| = k$.

The moduli space of holomorphic disks is non-compact, and to remedy
this the structure maps of the Fukaya algebra are defined over Novikov
rings in a formal variable.  Let
\[\Lambda = \Set{ \sum_{i=1}^\infty c_i q^{d_i} \ | \ \ c_i \in \C, d_i
  \in \R, \lim_{i \to \infty} d_i =  \infty }\]
denote the {\em universal Novikov field}.  The valuation by powers of
$q$
\[ \val_q: \Lambda  - \{ 0 \} \to \R, \quad 
\sum_{i=1}^\infty c_i q^{d_i} \mapsto \min_{c_i \neq 0 }(d_i) \]
is well-defined and the {\em Novikov ring}
\begin{equation} \label{novikov} \Lambda_{\ge 0} = \{ f \in \Lambda \
  | \ \val_q(f) \ge 0 \}, \quad \text{resp.} \quad \Lambda_{> 0} = \{
  f \in \Lambda \ | \ \val_q(f) > 0 \},
\end{equation} 
is the subspace with non-negative resp. positive $q$-valuation.  Let
$\Lambda^\times \subset \Lambda$ be the subset with zero
$q$-valuation.

The Floer cochain space is the free module over generators given by
Morse critical points, self-intersection points, and the two
additional generators from \eqref{selfint} necessary to achieve strict
units.  Let
\[CF(\phi) = \bigoplus_{ x \in \cI(\phi) } \Lambda x .\]
The space of Floer cochains is naturally $\bZ_N$-graded by 
\[ CF(\phi) = \bigoplus_{k \in \bZ_N} CF^k(\phi), \quad CF^k(\phi) = \bigoplus_{x \in \cI^k(\phi)} \Lambda x.\]
Put
\[ 1_\phi := x^{\whitet} \in CF(\phi) .\]
The $q$-valuation on $\Lambda$ extends naturally to $CF(\phi)$:
\[ \val_q:CF(\phi) - \{ 0 \} \to \R, \quad \sum c(x) x \mapsto
\min(\val_q(c(x)) .\]
We suppose that $L$ is equipped with a local system
$y: \pi_1(L) \to \Lambda^\times$, and denote for any holomorphic treed
disk $u: C \to X$ the holonomy of the local system around the boundary
of the disks components in $C$ by $y(u) \in \Lambda^\times$.  For
regular perturbations define {\em higher composition maps}
\[ \mu_d: CF(\phi)^{\otimes d} \to CF(\phi)[2-d] \]
on generators by 
\[ 
\mu_d(x_1,\ldots,x_d) = \sum_{x_0,u \in\ol{\M}_\Gamma(\phi,D,\ul{x})_\rho}
(-1)^{\heartsuit} (\sigma(u)!)^{-1} y(u) q^{E(u)} \eps(u) x_0 \]
where $\sigma(u)$ is the number of bulk leaves,
$\eps(u) \in \{ \pm 1 \}$ was defined in Theorem \ref{comeager}, and
\begin{equation} \label{heartsuit}  \heartsuit = {\sum_{i=1}^d i|x_i|} .\end{equation}

\begin{theorem} For any regular, coherent perturbation system
  $\ul{P} = (P_\Gamma)$ the maps $(\mu^d)_{d \ge 0}$ satisfy the
  axioms of a (possibly curved) \ainfty algebra $CF(\phi)$ with strict
  unit $1_\phi = x^{\whitet} \in CF(\phi)$.
\end{theorem}

In particular, $\mu_d$ satisfy the {\em \ainfty-associativity
  equations}
\begin{multline} \label{ainftyassoc} 
0 = \sum_{\substack{d_1,d_2 \ge 0 \\ d_1+d_2 \leq d}} (-1)^{ d_1 + \sum_{i=1}^{d_1} |x_i|}
\mu^{d-d_2+1}(x_1,\ldots,x_{d_1}, \\ \mu^m(x_{d_1+1},\ldots,x_{d_1+d_2}),
x_{d_1+d_2+1},\ldots,x_d) 
\end{multline}
for any $ x_1,\ldots, x_d \in \cI(\phi)$.  Up to sign the relation
\eqref{ainftyassoc} follows from the description of the boundary in
Theorem \ref{comeager} of the one-dimensional components, while the
sign computation in \cite{flips} is independent of whether the
Lagrangian is immersed or embedded.  The element
\[ \mu_0(1) \in CF(\phi) \]
is the {\em curvature} of the Fukaya algebra and has positive
$q$-valuation $\val_q(\mu_0(1))$.  The Fukaya algebra $CF(\phi)$ is
{\em flat} if $\mu_0(1)$ vanishes and {\em projectively flat} if
$\mu_0(1)$ is a multiple of the identity $1_\phi$.  Consider the
sub-space of $CF(\phi)$ consisting of elements with positive
$q$-valuation with notation from \eqref{novikov}:
\[CF(\phi)_+ = \bigoplus_{x \in \cI(\phi)} \Lambda_{>0} x .\]
Define the {\em Maurer-Cartan map}
\[ \mu: CF(\phi)_+ \to CF(\phi), \quad 
b \mapsto \mu_0(1) + \mu_1(b) + \mu_2(b,b) + \ldots .\]
Let $MC(\phi)$ denote the space of {\em weak solutions to the
  Maurer-Cartan space}  
\[ MC(\phi) = \{ b \in CF^{\on{odd}}(\phi) \ | \ \mu(b) = W(b) 1_\phi,
\quad W(b) \in \Lambda \} .\]
The value $W(b)$ of $\mu(b)$ for $b \in MC(\phi)$ defines the {\em
  disk potential}
\[ W: MC(\phi) \to \Lambda .\]
For $b \in CF(\phi)_+$ define  
\[  \mu^b_d(a_1,\ldots,a_d) = \sum_{i_1,\ldots,i_{d+1}} \mu_{d + i_1 +
  \ldots + i_{d+1}}(\underbrace{b,\ldots, b}_{i_1}, a_1,
\underbrace{b,\ldots, b}_{i_2}, a_2,b, \ldots, b, a_d,
\underbrace{b,\ldots, b}_{i_{d+1}}) \]
For $b \in MC(\phi)$, the maps $\mu^b_d, d \ge 1$ form a flat \ainfty
algebra.

\begin{remark} {\rm (Homotopy invariance)} The homotopy type of the
  the immersed Fukaya algebra $CF(\phi)$ is independent of the choices
  of $J,m,D$ and $\ul{P} = (P_\Gamma)$ up to \ainfty homotopy
  equivalences, by the immersed version of \cite[Corollary
  3.10]{flips} whose details we leave to the reader.  The homotopy
  equivalences between \ainfty algebras for perturbations
  $\ul{P}, \ul{P}'$ induce maps between the Maurer-Cartan moduli
  spaces $MC(\phi,\ul{P}) \to MC(\phi,\ul{P}')$ so that Floer
  cohomology (as a collection of vector spaces $HF(\phi,b)$ over the
  Maurer-Cartan space $b \in MC(\phi)$) is an invariant of the
  Lagrangian immersion $\phi$.

  The homotopy type of the immersed Fukaya algebra is also invariant
  under Hamiltonian diffeomorphism, but not exact isotopy.  Given a
  Hamiltonian diffeomorphism $\psi: X \to X$ \label{anyd} any Donaldson hypersurface
  $D \subset X$ in the complement of $\psi(\phi(L))$ pulls back to a
  Donaldson hypersurface $\psi^{-1}(D) \subset X$ in the complement of
  $\phi(L)$.  The domain-dependent almost complex structure $J_\Gamma$
  and Morse functions $m_\Gamma$ also pull back giving a bijection
  between treed holomorphic disks in the definition of the structure
  maps, so that $CF(\phi) \simeq CF(\psi \circ \phi)$.  The immersed
  Fukaya algebra $CF(\phi)$ is {\em not} invariant under arbitrary
  exact isotopy $\phi_t: L \to X$ (even if the number of intersection
  points is preserved) since the areas of the holomorphic disks
  $u: C \to X$ with boundary in $\phi_t$ may change, destroying a
  solution to the Maurer-Cartan equation.  See Example
  \ref{weinsteinex}.  This ends the Remark.
\end{remark}

\begin{example} \label{remind} We remind the reader of a basic
  computation of Floer cohomology in the Morse model for circles
  embedded in the two-sphere, which is a special case of the results
  of, for example, \cite{fooo:toric1}: Let the symplectic manifold $X$
  be the unit two-sphere $S^2$ in $\bR^3$ with its standard symplectic
  form with volume $4 \pi$.  The only embedded Lagrangian with
  non-trivial Floer cohomology is the equator, up to Hamiltonian
  isotopy.  Indeed, consider an embedding $\phi:L = S^1 \to X = S^2$
  such that $\phi(L)$ separates $X$ into regions of areas
  $A_-,A_+$. To compute the embedded Floer theory let $m: S^1 \to \R$
  denote a Morse function on $L$ with a pair of critical points $x_+$
  resp. $x_-$ the maximum resp. minimum of $m$.  We have
  \[ \mu_0(1) = q^{A_+} x_+ + q^{A_-} x_+, \quad \mu_1(x_-) = q^{A_+}
  x_+ - q^{A_-} x_+ .\]
  If the geometric unit was the same as the strict unit, $b = 0$ would
  give a weak Maurer-Cartan solution and the Floer cohomology of
  $\phi_0$ would be non-trivial if and only if $A_+ = A_-$. \label{wouldbe}

  We explain why the distinction between strict and geometric units is
  immaterial in this case. The Floer theory has a weakly bounding
  cochain
\[ b = (q^{A_+} + q^{A_-}) x^{\greyt} .\]
For reasons of dimension $\mu_n(b,\ldots, b)$ vanishes for
$b > 1$ while 
\[ \mu_1(x^{\greyt}) = x^{\whitet} - x_+ .\]
The graphs with a single edge and no vertex and weighting zero or
infinity give the two terms on the right, while there are no terms
with pseudoholomorphic disks since these are negative index. Hence
\[ \mu_0^b(1) = \mu_0(1) + \mu_1(b) = (q^{A_+} + q^{A_-}) x_+
+ (q^{A_+} + q^{A_-}) ( x^{\whitet} - x_+) = (q^{A_+} +
q^{A_-} ) x^{\whitet} .\]
With this choice of $b$ the Floer cohomology is again non-trivial iff
$A_+ = A_-$.
\end{example}

\begin{example} \label{teardrops} We consider an immersion of the
  circle in the two-sphere with a pair of self-intersection points.
  Let $X$ be again the symplectic two-sphere $S^2$ and $L$ the circle
  $S^1$ with immersion $\phi: L \to X$ as shown in Figure
  \ref{generators}. For the example shown in \ref{ex1}, recall that
  $\zz,\bzz$ and $\bz,\mz$ are the generators of the Floer cochains
  created by the new self-intersection points while $x_+ ,x_-$ are the
  maximum, resp. minimum of the Morse function, and
  $x^{\whitet}, x^{\greyt}$ are the additional generators added to
  obtain strict units.  The self-intersection points divide $\phi(L)$
  into regions of areas $A_1,A_2,A_3$ while the exterior has area
  $A_0$.  The curvature of the Fukaya algebra is, for some choices of
  trivializations of the determinant lines,
\[ \mu_0(1) = q^{A_3} \zz + q^{A_1} \mz + q^{A_0 + 2A_2}x_+ \] 
from contributions of the regions shown in Figure \ref{regions}, with
the final term the contribution from the ``outside'' of the circle.
\begin{figure} 
\begin{center}
%% Creator: Inkscape inkscape 0.91, www.inkscape.org
%% Accompanies image file 'imex2.pdf' (pdf, eps, ps)
\begingroup%
  \makeatletter%
  \providecommand\color[2][]{%
    \errmessage{(Inkscape) Color is used for the text in Inkscape, but the package 'color.sty' is not loaded}%
    \renewcommand\color[2][]{}%
  }%
  \providecommand\transparent[1]{%
    \errmessage{(Inkscape) Transparency is used (non-zero) for the text in Inkscape, but the package 'transparent.sty' is not loaded}%
    \renewcommand\transparent[1]{}%
  }%
  \providecommand\rotatebox[2]{#2}%
  \ifx\svgwidth\undefined%
    \setlength{\unitlength}{64.31480103bp}%
    \ifx\svgscale\undefined%
      \relax%
    \else%
      \setlength{\unitlength}{\unitlength * \real{\svgscale}}%
    \fi%
  \else%
    \setlength{\unitlength}{\svgwidth}%
  \fi%
  \global\let\svgwidth\undefined%
  \global\let\svgscale\undefined%
  \makeatother%
  \begin{picture}(1,3.24090954)%
    \put(0,0){\includegraphics[width=\unitlength,page=1]{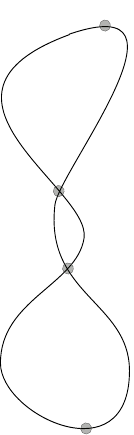}}%
      \put(0.68588225,3.16002938){\color[rgb]{0,0,0}\makebox(0,0)[lb]{\small \smash{$x_+$}}}%
    \put(0.27051352,1.97142873){\color[rgb]{0,0,0}\makebox(0,0)[lb]{\small \smash{$\zz$}}}%
    \put(0.38309095,1.08264629){\color[rgb]{0,0,0}\makebox(0,0)[lb]{\small \smash{$\mz$}}}%
    \put(0.53397912,0.10380593){\color[rgb]{0,0,0}\makebox(0,0)[lb]{\small \smash{$x_-$}}}%
    \put(0.05544301,1.53534477){\color[rgb]{0,0,0}\makebox(0,0)[lb]{\small \smash{}}}%
    \put(0.02558908,1.48843159){\color[rgb]{0,0,0}\makebox(0,0)[lb]{\small \smash{$A_0$}}}%
    \put(0.37104156,0.63119722){\color[rgb]{0,0,0}\makebox(0,0)[lb]{\small \smash{$A_1$}}}%
    \put(0.43501431,1.5012261){\color[rgb]{0,0,0}\makebox(0,0)[lb]{\small \smash{$A_2$}}}%
    \put(0.29853919,2.56317277){\color[rgb]{0,0,0}\makebox(0,0)[lb]{\small \smash{$A_3$}}}%
    \put(0,0){\includegraphics[width=\unitlength,page=2]{imex2.pdf}}%
  \end{picture}%
\endgroup%
 \end{center}
\caption{Teardrops obstructing Floer cohomology} 
\label{regions}
\end{figure} 
Thus $ b = 0$ is {\em not} a solution to the weak Maurer-Cartan
equation.

We look for a weakly bounding cochain that ``kills the teardrops''.
Working with the trivial local system $y \in \RR(L)$ the first
composition map has vanishing terms
\[ \mu_1(\mz) = 0, \quad \mu_1(\zz) = 0, \quad 
\mu_1(x_+) = 0 \]
while 
\[ 
\mu_1(\bzz) = - q^{A_2} \mz + q^{A_3} x_+, \quad \mu_1(\bz) = -  q^{A_2}
\zz, 
\quad \mu_1(x_-) = q^{A_1} \mz + q^{A_0 + 2A_2} x_+ ;\]
the last term arises from the holomorphic bigon
$u: \bR \times [0,1] \to X$ including the exterior that overlaps with
itself in the region with area $A_2$.  The negative signs depend on
the choice of relative spin structures; here we take the determinant
lines to be oriented by the tangent direction; then the relative spin
structure on completion of the middle region $A_2$ is non-trivial and
contributes a sign in the above formulas involving $A_2$.  There are
no other regions contributing to the \ainfty structure maps involving
$\bzz,\bz$.  If $\min(A_1,A_3) > A_2$ then in particular the element
\[ b_0 =  q^{A_1 - A_2} \bzz  + q^{A_3 - A_2} \bz\]
has positive $q$-valuation \label{posq} and solves the Maurer-Cartan-like equation
$$ \mu_0^{b_0}(1) \in \on{span} (x_+) .$$
Indeed, 
\begin{eqnarray*} 
  \mu_0^{b_0}(1) &=& \mu_0(1) + \mu_1(b_0) + \mu_2(b_0,b_0) + \ldots  \\ 
                 &=& q^{A_0 + 2A_2}x_+ + q^{A_1} \zz + q^{A_3} \mz + q^{A_1 -A_2} \mu_1(\bz) 
                     + q^{A_3 - A_2} \mu_1(\bzz)  \\
                 &=& q^{A_0 + 2A_2}x_+ +
                     q^{A_1 +
                     A_3 - A_2} x_+ .\end{eqnarray*} 

                 Ignoring the difference between the geometric and
                 strict units the element above gives a weak
                 Maurer-Cartan solution. To obtain an honest
                 Maurer-Cartan solution as in the previous example we
                 shift so that the zero-th composition map is a
                 multiple of the strict unit:
                 \[ b = b_0 + w x^{\greyt}, \quad w: = (q^{A_0 + 2A_2}
                 + q^{A_1 + A_3 - A_2} ) .\]
We claim that 
\[ \mu_0^b(1) = w x^{\whitet} .\]
In case all the incoming leaves are labelled by elements
$x \in \cI(\phi)$ not equal to the minimum $x_-$, then any
configuration $u: C \to X$ involving these elements with at least one
vertex survives perturbation of the Morse function, almost complex
structure, and metric.  Indeed, the holomorphic disks $u: S \to X$
with boundary in $\phi(L)$ are invariant under perturbation of the
almost complex structure, since the complex structure on the disk is
unique up to diffeomorphism, and the condition on semi-infinite
gradient trajectories on the leaves is that they flow up to the
maximum.  It follows that the only contributions to
$\mu_1^{b_0}(w x^{\greyt})$ are those with at most one vertex in the
corresponding graph, since otherwise the weighting parameter $\varkappa$ is
free to vary and the configuration cannot be rigid.  So as in the
previous example,
\begin{eqnarray*} \mu_0^b(1) &=& \mu_0^{b_0}(1) + \mu_1^{b_0}(w
                                 x^{\greyt}) \\
                             &=&
                                 w x_+ + \mu_1^0(w x^{\greyt}) \\
                             &=&
                                 w x_+ + w (x^{\whitet} - x_+) \\
                             &=& w x^{\whitet} .\end{eqnarray*}
Hence
\[ b \in MC(\phi) .\]

We compute the Floer cohomology for this bounding cochain.  For the
same reason as in previous paragraph, we have
$ \mu_1^{b_0} = \mu_1^b $ \label{mub} and so we may ignore the difference between
geometric and strict unit.  The first composition map (that is, the
differential) \label{differential} satisfies
\[ \mu_1^b(x_+) = 0, \quad 
\mu_1^b(x_- + q^{A_3 - A_2} \bz) =  (q^{A_0 + 2A_2} - q^{A_3 - A_2 +
  A_1}) x_+ .\]
Since $\mu_1^b$ squares to zero, we must have 
\[ 
\mu_1^b(q^{A_2} \mz + q^{A_1} x_+) = q^{A_2} \mu_1^b(\mz) = 0, \quad
\mu_1^b(q^{A_2} \zz) = q^{A_2} \mu_1^b(\zz) = 0 .\]
So $\mu_1^b(\mz) = \mu_1^b(\zz) = 0$.  Computing the cohomology of
$\mu_1^b$ we have 
\begin{equation} \label{weighted} 
 HF(\phi,b) = \begin{cases}  \Lambda^{\oplus 2}  & A_0  + 3A_2 = A_3 +
   A_1\\
\{0 \} & \text{otherwise} .\end{cases} \end{equation} 

We check that the computation is compatible with Hamiltonian
displaceability.  Note that the condition \eqref{weighted} implies
\begin{eqnarray*}   A_0  &=&  A_1 + A_2 + A_3 - 4 A_2 \\ 
 A_1 &=& A_0 + A_2 + A_3 + (2A_2 - 2 A_3) \\
 A_2 &=&  A_0 + A_1 + A_3  - 2 A_0 - 2A_2  \\
 A_3 &=& A_0 + A_1 + A_2  + (2A_2 - 2 A_1) 
.\end{eqnarray*}
These equalities preclude any equality
\begin{equation} \label{displ}  A_i > A_j + A_k + A_l , \quad i,j,k,l
  \ \text{distinct} .\end{equation} 
Equation \eqref{displ} is equivalent to displaceability by Moser's
theorem \cite{moser} since if one area $A_i$ is larger than the sum of
the other three $A_j + A_k + A_l$ then there exists a Hamiltonian
diffeomorphism $\psi: X \to X$ that moves the exterior of the largest
region $X - A_i$ into its interior $ \on{int}(A_i)$.  Thus the
non-triviality of the Floer cohomology $HF(\phi,b)$ is consistent
with non-displaceability.
 \end{example}

\begin{example} \label{figureeight} Consider an embedded circle with a single intersection
  point, dividing the sphere into regions $R_0$ (``outside''), $R_1$
  and $R_2$ as in Figure \ref{single}, with areas $A_0,A_1,A_2$.
\begin{figure} 
  \begin{center} 
%% Creator: Inkscape inkscape 0.91, www.inkscape.org
%% Accompanies image file 'eight.pdf' (pdf, eps, ps)
\begingroup%
  \makeatletter%
  \providecommand\color[2][]{%
    \errmessage{(Inkscape) Color is used for the text in Inkscape, but the package 'color.sty' is not loaded}%
    \renewcommand\color[2][]{}%
  }%
  \providecommand\transparent[1]{%
    \errmessage{(Inkscape) Transparency is used (non-zero) for the text in Inkscape, but the package 'transparent.sty' is not loaded}%
    \renewcommand\transparent[1]{}%
  }%
  \providecommand\rotatebox[2]{#2}%
  \ifx\svgwidth\undefined%
    \setlength{\unitlength}{100.16567993bp}%
    \ifx\svgscale\undefined%
      \relax%
    \else%
      \setlength{\unitlength}{\unitlength * \real{\svgscale}}%
    \fi%
  \else%
    \setlength{\unitlength}{\svgwidth}%
  \fi%
  \global\let\svgwidth\undefined%
  \global\let\svgscale\undefined%
  \makeatother%
  \begin{picture}(1,1.15957357)%
    \put(0,0){\includegraphics[width=\unitlength,page=1]{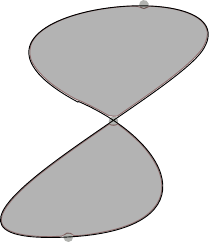}}%
    \put(0.80909476,0.62709015){\color[rgb]{0,0,0}\makebox(0,0)[lb]{\small \smash{}}}%
    \put(0.49714944,0.9064007){\color[rgb]{0,0,0}\makebox(0,0)[lb]{\small \smash{$A_1$}}}%
    \put(0.37173924,0.27747676){\color[rgb]{0,0,0}\makebox(0,0)[lb]{\small \smash{$A_2$}}}%
    \put(0.72165326,0.53556023){\color[rgb]{0,0,0}\makebox(0,0)[lb]{\small \smash{$A_0$}}}%
  \end{picture}%
\endgroup%
\end{center}
\caption{An example with vanishing Floer cohomology} 
\label{single}
\end{figure} 
Let $x_+$ be the maximum, $v,\ol{v}$ the generators associated to the
self-intersection points and $x_-$ the minimum.  The two lobes of the
figure eight contribute to $\mu_0(1) = (q^{A_1} + q^{ A_2}) v$ with
the same sign by invariance of the picture (up to deformation) under
rotation by $\pi$.  The first \ainfty map for $b = 0$ is given by
\[\mu_1(x_+) = 0, \quad \mu_1(v) = 0, \quad \mu_1(\ol{v}) =
q^{A_1}x_+, \quad \mu_1(x_-) = q^{A_2} v .\]
Indeed a configuration with an input $x_+$ can never be rigid; there
are two possible strips with input $v$ but these can be seen to cancel
by the \ainfty relations, or an explicit sign computation; the lobe
with area $A_2$ contributes a holomorphic strip causing
$\mu_1(x_-) = q^{A_2} v$.  The higher composition maps involving $x_-$
depend on a choice of perturbations.  There is a choice for which one
has a version of the divisor equation
\[ \mu_d(x_-,\ldots, x_-) = (d!)^{-1} \mu_1(x_-) = 
  (d!)^{-1} q^{ A_2} v \] 
  explained in \cite[Section 4.4]{pw:surger}.  The higher composition
  maps involving $v$ as input are trivial.  A weakly bounding cochain
  (albeit) with zero $q$-valuation must have for some constant $c$
  \[ b = \ln( - q^{A_1 - A_2}) x_- + c v = - \pi i x_- + c v \]
  and therefore exists only if $A_1 = A_2$ in which case
\[ \mu_0^b (1) = cx_+ .\]  
The reader who is uncomfortable with the assertion on the divisor
equation may take a local system whose holonomy around the lower lobe
is $-1$ and take $b = 0$.  The Floer operator has
\[ \mu_1^b(x_+) = 0, \quad \mu_1^b(v) = 0, \quad \mu_1^b(\ol{v}) =
q^{A_1}x_+, \quad \mu_1^b(x_-) = - q^{A_2} v .\]
Thus the Floer cohomology vanishes.  On the other hand, $\phi(L)$ is
non-displaceable if $A_i > A_j + A_k$ for some distinct $i,j,k$.  We
remark that Grayson \cite{grayson:eight} has conjectured that the mean
curvature flow of $\phi$ converges to a point if $A_1 = A_2$.
\end{example}

\begin{example} \label{weinsteinex} The following is an example of
  Weinstein \cite{weinstein} of an exact isotopy of maps from the
  circle $S^1$ to the cylinder $T^* S^1$.  The family includes
  immersions with both obstructed and unobstructed Floer cohomology.
  It also shows that non-triviality of immersed Floer cohomology is
  not an obstruction to displaceability by exact deformation. Let
  $\phi_t: S^1 \to T^* S^1$ be an immersion of the form shown in
  Figure \ref{disjoint}.  If the areas satisfy the equality
  $A_1 + A_3 = A_2$, then $\phi_1(L)$ is an exact isotopy of
  $\phi_0(L)$ by standard Moser arguments.
\begin{figure} 
\begin{center}
%% Creator: Inkscape inkscape 0.91, www.inkscape.org
%% Accompanies image file 'weinstein.pdf' (pdf, eps, ps)
\begingroup%
  \makeatletter%
  \providecommand\color[2][]{%
    \errmessage{(Inkscape) Color is used for the text in Inkscape, but the package 'color.sty' is not loaded}%
    \renewcommand\color[2][]{}%
  }%
  \providecommand\transparent[1]{%
    \errmessage{(Inkscape) Transparency is used (non-zero) for the text in Inkscape, but the package 'transparent.sty' is not loaded}%
    \renewcommand\transparent[1]{}%
  }%
  \providecommand\rotatebox[2]{#2}%
  \ifx\svgwidth\undefined%
    \setlength{\unitlength}{188.55726828bp}%
    \ifx\svgscale\undefined%
      \relax%
    \else%
      \setlength{\unitlength}{\unitlength * \real{\svgscale}}%
    \fi%
  \else%
    \setlength{\unitlength}{\svgwidth}%
  \fi%
  \global\let\svgwidth\undefined%
  \global\let\svgscale\undefined%
  \makeatother%
  \begin{picture}(1,0.99300233)%
    \put(0,0){\includegraphics[width=\unitlength,page=1]{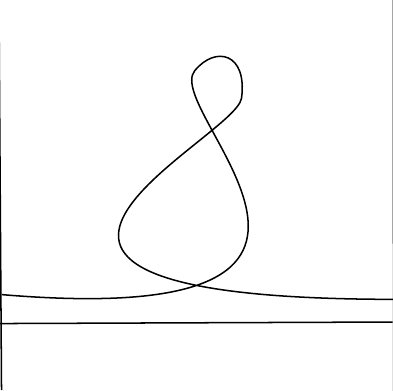}}%
    \put(0.38961377,0.18807638){\color[rgb]{0,0,0}\makebox(0,0)[lb]{\small \smash{$A_1$}}}%
    \put(0.37417594,0.43474253){\color[rgb]{0,0,0}\makebox(0,0)[lb]{\small \smash{$A_2$}}}%
    \put(0.50498664,0.76180027){\color[rgb]{0,0,0}\makebox(0,0)[lb]{\small \smash{$A_3$}}}%
    \put(0,0){\includegraphics[width=\unitlength,page=2]{weinstein.pdf}}%
    \put(0.61715766,0.6643017){\color[rgb]{0,0,0}\makebox(0,0)[lb]{\small \smash{$v_+,\ol{v}_+$}}}%
    \put(0.61715766,0.27000651){\color[rgb]{0,0,0}\makebox(0,0)[lb]{\small \smash{$v_-,\ol{v}_-$}}}%
  \end{picture}%
\endgroup%
\end{center}
\caption{An exact isotopy of the zero section with 
obstructed Floer theory}
\label{disjoint}
\end{figure} 
Let $\bz,\mz$ resp. $\zz,\bzz$ be the generators of the Floer complex
corresponding to the lower resp. higher self-intersection point (with
bar if the ordered intersection point is odd)
\[\mu_0(1) = q^{A_3} \zz, \quad \mu_1(\bz) = q^{A_2} \zz .\]
For unobstructedness the weakly bounding cochain is forced to equal
\[b = q^{A_3 - A_2} \bz \in MC(\phi) \] 
as long as $A_3 > A_2$ since the Maurer-Cartan space requires positive
$q$-valuation $\val_q(b) > 0$.  This is impossible if
$A_2 = A_1 + A_3$. \label{impossible}

On the other hand, consider an exact isotopy of $\phi_1(L)$ with the
property that $A_1$ is negative, so that $A_3 > A_2$.  In this case,
$b$ is well-defined and the Floer cohomology is non-trivial.  This
ends the example.
\end{example}

\section{Invariance for transverse self-intersection}

In this section we prove that Maslov flows not changing the number of
self-intersection points leave the Floer cohomology invariant.  This
generalizes an argument in Alston-Bau \cite{alston} which assumed a
monotonicity condition on the Lagrangian.  Since the structure maps in
our setting are weighted by areas, we study how the areas of the disks
contributing to the Fukaya algebra change as we deform the immersion
under a Maslov flow, similar to the study in Angenent
\cite{angenent}.  We show that immersed Lagrangian Floer theory
possesses a canonical $\R$-grading, similar to the $\bQ$-grading of the
twisted sectors on orbifold quantum cohomology as in Chen-Ruan
\cite{cr:orb}, (which is preserved only to leading order by the
differential) which determines the change in the areas of the
holomorphic polygons under Maslov flow.

First, we explain that for small times there exists a bijection
between pseudoholomorphic disks with the given boundary condition.

\begin{lemma} \label{bijlema} Let $\phi_t: L \to X, t \in [0,T]$ be an
  isotopy of Lagrangians with only self-transverse intersections and
  no triple intersections.  For $t_2$ sufficiently close to $t_1$,
  there exists a diffeomorphism $\psi_{t_1}^{t_2}$ of $X$ mapping
  $\phi_{t_2}(L)$ to $\phi_{t_1}(L)$ and for which for any
  combinatorial type $\Gamma$ of treed disk the almost complex
  structure $\psi_{t_1}^{t_2} J_{\Gamma,t_1}$ is admissible.  This
  diffeomorphism induces a bijection between moduli spaces
  $\M_\Gamma(\phi_{t_1})$ and $\M_\Gamma(\phi_{t_2})$, defined using
  $J_{\Gamma,t_1}$ resp.  $\psi_{t_1}^{t_2,*} J_{\Gamma,t_2}$.
\end{lemma}

\begin{proof} As in Alston-Bao \cite{alston}, the fact that there are
  only self-transverse intersections means that there exists a family
  of diffeomorphisms $\psi_{t_1}^{t_2}$ mapping $\phi_{t_1}(L)$ to
  $\phi_{t_2}(L)$.  The assumption of transverse self-intersections
  and no triple self-intersections implies that
  $\cup_{t \in [0,T]} (\phi_t( L \times_{\phi_t} L - \Delta_L)) \times
  \{ t \} $
  is an embedded submanifold of the product $X \times [0,T]$.
  Standard arguments imply that there exists a tubular neighborhood
  $B_\eps(0, \bR^{2n}) \times [0,T] \to X \times [0,T]$ of
  $\{ x _t , t \in [0,T] \}$ so that the branches of $\phi_t(L)$ near
  any self-intersection are the images of
  $B_\eps(0, \bR^n \times \{ 0 \})$ resp
  $B_\eps(0, \{ 0 \} \times \bR^n)$.  The family $\psi_{t_1}^{t_2}$ is
  defined as the restriction of the flow of a vector field $v_t$ on
  $X \times [0,T]$ to $X \times \{ t_1 \}$, with $v_t$ defined so that
  (a) $D \pi v_t = \ddt$, where $\pi: X \times [0,T] \to [0,T]$ is the
  projection (b) $v_t$ is tangent to the submanifold of
  self-intersections
  $\cup_{t \in [0,T} (\phi_t( L \times_{\phi_t} L - \Delta_L)) \times
  \{ t \} $
  (c) $v_t$ is tangent to $\phi_t(L)$ away from the self-intersections
  $\phi_t(L \times_{\phi_t}L - \Delta_L)$.  Such a vector field may be
  defined in stages, first in the tubular neighborhood of the
  self-intersections and then extended to $\phi_t(L)$ and $X$.  Since
  tameness is an open condition, $\psi_{t_1}^{t_2,*} J_{\Gamma,t_2}$
  is also tamed for $t_1$ sufficiently close to $t_2$, and similar for
  the stabilizing condition.  The map $\psi_{t_1}^{t_2}$ induces the
  claimed bijection of moduli spaces of pseudoholomorphic treed disks.
\end{proof}

We give two definitions of the $\R$-grading. The first definition of
the grading is an action-index difference. 

\begin{definition} Let
\[(x_-,x_+) \in L^2, \quad u(x_-) = u(x_+) = x \in X\]
be a self-intersection point.  Choose a path
\[ \gamma_x : [-1,1] \to \Lag(T_x X), \quad \gamma(\pm 1) =
D\phi_{x_{\pm}} (T_{x_\pm} L) \]
as in \eqref{gammax}.  Consider the determinant map 
\[ \on{det}: \Lag(T_x X) \to \Lag(K_x^{-1}) \cong S^1 \]
given by the top exterior product.  Let 
\[\d \theta \in \Omega^1(S^1), \quad \int_{S^1} \d \theta = 2 \pi \]
denote the standard angular form that is circle-invariant.  Define the 
{\em action} of the path
\begin{equation} \label{pathaction} a(x) = \int_{-1}^1 (\on{det}
  \gamma_x)^* \d \theta /\pi \end{equation}
obtained by integrating the pull-back of the
angular form.  If $\ol{x}= (x_+,x_-)$ is the reversed
self-intersection point denote by $\gamma_{\ol{x}}$ the opposite of
the path $\gamma_x$ so that
$$ a(x) + a(\ol{x}) = 0, \quad i(x) + i(\ol{x}) = n .$$
Define 
\begin{equation} \label{theta} \theta(x) = i(x) - a(x) \end{equation} 
independent of the choice of path $\gamma_x$.
\end{definition} 

Alternatively, the $\R$-grading may be given in terms of the sum of
the K\"ahler angles of intersection of the Lagrangians.  

\begin{definition} Define an {\em angle}
\[ \theta_1 = \min \{ \theta \in (0,1) \ | \ e^{\pi i \theta}
D\phi(T_{x_-}L) \cap D\phi(T_{x_+}L) \neq \{ 0 \} \} .\]
Replacing $D\phi( T_{x_\pm} L)$ with their orthogonal complements in
$T_x X$ and applying the construction iteratively to the complements,
one obtains a basis 
\[\{ y_1,\ldots,y_n \}  \subset D\phi(T_{x_+} L)\] 
and angles
\[\theta_1 < \ldots < \theta_n \in (0, 1 )\] 
so that bases for the Lagrangian tangent spaces are given by 
\begin{equation} \label{angles} D\phi (T_{x_+}L)  = \on{span} \{ y_1,\ldots, y_n \}, \quad 
D\phi(T_{x_-}L) = \on{span} \{ e^{i \theta_1} y_1, \ldots, e^{i \theta_n }
y_n \} .\end{equation} 
Then as in \cite[Equation (3)]{alston}
\[ \theta(x) = n - \sum_{i=1}^n \theta_i \in (0,n) .\]
\end{definition} 

We these definitions in place we prove a variational formula for the
areas under Maslov flow.  As in \eqref{maslovflow} suppose that the
anticanonical bundle $K^{-1}$ is equipped with a connection $\alpha_t$
so that the infinitesimal variation $\cdot{\phi_t}$ is given by the
corresponding one-form $\phi_t^* \alpha_t$ representing the connection
on $\phi_t^* K^{-1}$.  The following is a generalization of a formula
of Angenent \cite[bottom of page 1214]{angenent} for holomorphic
polygons with a single corner in the dimension two case, to arbitrary
dimensions and polygons.

\begin{lemma} \label{change} The rate of change in the area of any family of
  rigid treed disks $u_t: C \to X$ with boundary in $\phi_t(L)$ and
  with vertices at infinity $z_0,\ldots, z_d$ (with only $z_0$
  outgoing) mapping to intersection points
  $x_0,\ldots, x_d \in \cI(\phi_t)$ is
  \[ \ddt \int_S u_t^* \omega_t =
 \sum_{k > 0 } \theta(x_k) - d + 2 -
  \theta(x_0) \]
with $\theta(x_k) := i(x_k) $ if $x_k \in \cI^{c}(\phi_t)$. 
\end{lemma}

\begin{remark} \label{noself} If there are no self-intersections at
  the corners $z_k$ and no incoming markings $d = 0$ the right-hand
  side is simply $2 - i(x_0)$.  For example, in the case of a single
  outgoing marking with output $x_0$ so that $i(x_0) = 0$ each disk in
  the equation is Maslov index $I(u_t) = 2$, while if only Maslov
  index two disks appear then the Fukaya algebra is projectively flat,
  that is, $\mu_0(1) \in \on{span}(1_{\phi_t})$.
\end{remark}

\begin{proof}[Proof of Lemma] We wish to put ourselves in the
  situation where there are no branch changes in the boundary
  condition, by adding additional disks at each of the branch changes.
  For simplicity we consider a configuration $C = S \cup T $ with a
  single disk $S$ with $d+1$ leaves $T = T_0 \cup \ldots T_d$ with
  points at infinity $z_0,\ldots, z_d \in C$ mapping to
  self-intersection points $x_0,\ldots, x_d \in \cI^{\on{si}}(\phi)$.
  Let $\ti{S}$ be the disk equipped with complex bundle $E$ with
  totally real Lagrangian boundary condition $F$ obtained by gluing
  together $S$ equipped with the bundle pair
  $(u^* TX, (\partial u)^* TL)$ with the disks $S_0,\ldots,
  S_d$
  \label{capping} equipped with bundles $T_{\phi(x_i)} X$ and boundary
  conditions corresponding to the paths
  \[\gamma_{x_k}: (0,1) \to \Lag( T_{\phi(x_k)} X), \quad j =
  0,\ldots, d .\]
  See Figure \ref{corners}; note that the capping disks
  $S_0,\ldots, S_d$ are not mapped to $X$ but rather are only equipped
  with bundle pairs.
\begin{figure}
\centering
%% Creator: Inkscape inkscape 0.91, www.inkscape.org
%% Accompanies image file 'corners.pdf' (pdf, eps, ps)
\begingroup%
  \makeatletter%
  \providecommand\color[2][]{%
    \errmessage{(Inkscape) Color is used for the text in Inkscape, but the package 'color.sty' is not loaded}%
    \renewcommand\color[2][]{}%
  }%
  \providecommand\transparent[1]{%
    \errmessage{(Inkscape) Transparency is used (non-zero) for the text in Inkscape, but the package 'transparent.sty' is not loaded}%
    \renewcommand\transparent[1]{}%
  }%
  \providecommand\rotatebox[2]{#2}%
  \ifx\svgwidth\undefined%
    \setlength{\unitlength}{101.12587396bp}%
    \ifx\svgscale\undefined%
      \relax%
    \else%
      \setlength{\unitlength}{\unitlength * \real{\svgscale}}%
    \fi%
  \else%
    \setlength{\unitlength}{\svgwidth}%
  \fi%
  \global\let\svgwidth\undefined%
  \global\let\svgscale\undefined%
  \makeatother%
  \begin{picture}(1,1.20720742)%
    \put(0,0){\includegraphics[width=\unitlength,page=1]{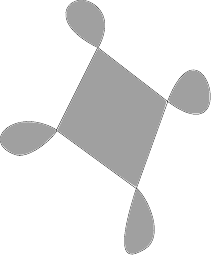}}%
    \put(1.12610199,-0.32407637){\color[rgb]{0,0,0}\makebox(0,0)[lb]{\small \smash{}}}%
    \put(0.86176921,0.75908287){\color[rgb]{0,0,0}\makebox(0,0)[lb]{\small \smash{$S_0$}}}%
    \put(0.38530993,1.06894806){\color[rgb]{0,0,0}\makebox(0,0)[lb]{\small \smash{$S_1$}}}%
    \put(0.07211299,0.52918297){\color[rgb]{0,0,0}\makebox(0,0)[lb]{\small \smash{$S_2$}}}%
    \put(0.61520971,0.12602524){\color[rgb]{0,0,0}\makebox(0,0)[lb]{\small \smash{$S_3$}}}%
    \put(0.44195188,0.61581213){\color[rgb]{0,0,0}\makebox(0,0)[lb]{\small \smash{$S$}}}%
  \end{picture}%
\endgroup%
\caption{Gluing a polygon with once-punctured disks} 
\label{corners}
\end{figure} 
First assume that all vertices of $S$ map to self-intersections of
$\phi$.  The bundle $u_t^* K^{-1}$ extends to a bundle on $\ol{S}$
obtained by taking top exterior powers.  The one-form $u_t^* \alpha$
extends to a one-form $\ol{\alpha}_t$ on the extension; via the given
trivialization of the anticanonical bundle $\phi_t^* K^{-1}$ over the
Lagrangian these connection one-forms become ordinary one-forms still
denoted $\ol{\alpha}_t \in \Omega^1(\ol{S})$ on the base.  Since the
Maslov index of the glued problem $(E,F)$ is the winding number of the
boundary condition, 
\[  \int_{\partial \ol{D}} \ol{\alpha_t} 
  =  I( E,F). \]
  Since $\mu_d$ is a map of degree $2 - d$, the Maslov index of any
  rigid disk contributing to $\mu_d$ must be the shifted difference in
  degrees
\[ I(E,F) = - i(x_0) + \sum_{i=1}^d i(x_i) - d + 2 .\] 
On the other hand, the integral of
$\ol{\alpha_t}$ around the boundary of any of the disks
$S_i$ is by \eqref{pathaction} equal \label{onthebase} to $\pm
a_t(x_i)$, with a minus sign for the self-intersection point that is
outgoing.  Putting everything together the integral of the connection
one-form $\alpha$ around the boundary $\partial S$ of the disk is
\begin{eqnarray*} \sum_{i=0}^d \int_{\partial S} u_t^* \alpha
  &=& \int_{\partial \ol{D}} \ol{\alpha_t} + a_t(x_0) - \sum_{i=1}^d a_t(x_i) \\
  &=& -  i(x_0) +  \sum_{i=1}^d i(x_i) - d + 2 + a_t(x_0) - \sum_{i=1}^d a_t(x_i)\\
  &=& \sum_{i=1}^d \theta_t(x_i) - \theta_t(x_0) -d + 2
      .\end{eqnarray*}
  The case that some of the points at infinity map to critical points
  $\cI^{\on{c}}(\phi)$ is similar, using the fact that each element
  $x_k \in \cI(\phi)$ of index $i(x_k)$ that is not a
  self-intersection point cuts down the dimension of the moduli space
  by dimension $i(x_k)$.  For the case that the configuration $C$ has
  several disk components, the computation follows by considering each
  disk separately and taking the sum.
\end{proof}

The following Euler flow on Floer cochains cancels out the change in
the areas of the disks under mean curvature flow of the immersion.
Combine the $\R$-gradings on self-intersection points and critical
points by
\[  \cI(\phi) \to \R, \quad x \mapsto |x| := \begin{cases} \theta(x) &  x \in \cI^{\on{si}}(\phi)\\ 
\dim(W_x^-) & x \in \cI^{\on{c}}(\phi) 
 \end{cases}  \]
 where $W_x^-$ is the unstable manifold of $\grad(m)$ at $x$.  

\begin{definition} 
  The {\em Euler vector field}
 \[ e_t \in \Vect(CF(\phi_t)), \quad e_t = \sum_{x \in \cI(\phi)} (1 -
 |x|) \partial_x .\]
 Identify $CF(\phi_t) \cong CF(\phi_{t_1})$ for all $t \in [t_1,t_2]$
 and let
\[ E_{t_1}^{t_2} : CF(\phi_{t_1}) \to CF(\phi_{t_2}) \]
denote the flow of the Euler vector field $e_t$ from time $t_1$ to
time $t_2$.  On generators we have
\begin{equation} \label{ongen}
 E_{t_1}^{t_2}: x \mapsto \begin{cases} q^{(t_2 - t_1)(1 - |x|)} x & x
  \in
  \cI^{\on{c}}(\phi)\cup \{ x^{\whitet}, x^{\greyt} \}  \\
  q^{ \int_{t_1}^{t_2} (1 - \theta(x))} x
 & x \in
  \cI^{\on{si}}(\phi) \end{cases} .\end{equation} 
\end{definition} 

\begin{theorem} \label{related} Let $\phi_t:L \to X$ be a Maslov flow
  of Lagrangian immersions with only transverse self-intersections and
  no triple intersections.  Suppose that $t_2$ is sufficiently close
  to $t_1$ so that the bijection between the moduli spaces
  $\M(\phi_{t_1})$ and $\M(\phi_{t_2})$ is induced by a family of
  diffeomorphisms $\psi_{t_1}^{t_2} : X \to X$ mapping $\phi_{t_1}(L)$
  to $\phi_{t_2}(L)$.  Then the A-infinity structure maps
  $\mu^{t_1}_d$ for $\phi_{t_1}$ and $\mu^{t_2}_d$ for $\phi_{t_2}$
  are related by
 \begin{multline}
 q^{(t_2 - t_1)} E_{t_1}^{t_2}( \mu^{t_1}_d(x_1,\ldots, x_d) ) =
  \mu^{t_2}_d( E_{t_1}^{t_2} x_1, \ldots ,E_{t_1}^{t_2} x_d ), \\
  \forall d \ge 0, \ x_1,\ldots, x_d \in \cI(\phi_{t_1}) \cong 
  \cI(\phi_{t_2}). \end{multline}
\end{theorem} 

\begin{proof} Since $\mu_d$ is defined by counting holomorphic disks,
  we have for all $x_k$'s self-intersection points
\begin{multline}  \label{deriv}
\ddt   q^{(t - t_1)} (E_{t_1}^t)^{-1}  
\mu^{t}_d( E_{t_1}^{t} x_1, \ldots ,E_{t_1}^{t} x_d )
\\ =  \ddt 
 q^{t - t_1}  q^{ \int_{t_1}^{t} (1 - \theta(x_0))} 
\mu^{t}_d \left(  q^{ \int_{t_1}^{t} (1 - \theta(x_1))} x_1,\ldots, 
  q^{ \int_{t_1}^{t} (1 - \theta(x_d))} x_d \right)
\\ = \left(  \sum_{k > 0 } \theta(x_k) - d + 2 -
  \theta(x_0)  + (\theta(x_0) - 1) - \sum_{k > 0 } (\theta(x_k) - 1)
\right)  \\
\ln(q) (E_{t_1}^t)^{-1} \mu^t_d(E_{t_1}^t x_1,\ldots, E_{t_1}^t x_d) 
=  0 
\end{multline}
by Lemma \ref{change}.  The claim in the Lemma holds when \label{when}
$t = t_1$ since $E_{t_1}^{t_1}$ is the identity.  By \eqref{deriv}
\label{theclaim} the claim holds for all $t$.  The general case (when
some $x_k \in \cI^{\on{c}}(\phi_t))$ is similar.
\end{proof}

Because of the signs in the Euler flow above, the flow may not
preserve the space of Maurer-Cartan solutions which are required to
have positive $q$-valuation.  Let
$$MC^{> E}(\phi_t)= \Set{ b \in MC(\phi_t) \ | \ \val_q(b) > E } $$
  be the subset of solutions to the weak Maurer-Cartan equation having 
  $q$-valuation at least $E$.  

  \begin{corollary} Let $\phi_t:L \to X$ be a Maslov flow of
    Lagrangian immersions with transverse self-intersections for
    $t \in [t_1,t_2]$.  The Euler flow
    $E_{t_1}^{t_2}: CF(\phi_{t_1}) \to CF(\phi_{t_2})$ (resp. its
    inverse) maps 
\[ E_{t_1}^{t_2}:MC^{> E} (\phi_{t_1}) \to MC^{ >E - (\dim(L)-1)(t_2 -
  t_1)}(\phi_{t_2}) \] 
resp.   
\[ MC^{> E} (\phi_{t_2}) \to MC^{ >E - 2(t_2 - t_1)}(\phi_{t_2}) .\]  
The potentials are preserved \label{preserved} up to an overall power
of $q$, that is,
\[ (E_{t_1}^{t_2})^* W_{t_2} = q^{(t_2 - t_1)} W_{t_1} \]
and the Euler flow \label{eulerflow} lifts to an isomorphism
\[ HF(\phi_{t_1},b_{t_1}) \to HF(\phi_{t_2}, b_{t_2}) . \]
\end{corollary}

\begin{proof} We have
\begin{eqnarray*} 
  \mu_0^{t_2}(1) &+& \mu_1^{t_2}(E_{t_1}^{t_2}b) +
  \mu_2^{t_2}(E_{t_1}^{t_2}b,E_{t_1}^{t_2} b) + \ldots  \\
  &=& q^{(t_2 - t_1)} E_{t_1}^{t_2} \mu_0^{t_2}(1) + q^{(t_2 - t_1)}
  E_{t_1}^{t_2}\mu_1^{t_2}(b) + q^{(t_2 - t_1)} E_{t_1}^{t_2}
  \mu_2^{t_2}(b, b) + \ldots \\ &=& q^{(t_2 - t_1)} E_{t_1}^{t_2} \left(
    \mu_0^{t_2}(1) + \mu_1^{t_2}(b) + \mu_2^{t_2}(b, b) + \ldots 
  \right) . 
\end{eqnarray*}
Since the real-grading on $CF(\phi_t)$ takes values between $0$ and
$n = \dim(L)$, the Euler flow
$E_{t_1}^{t_2}: CF(\phi_{t_1}) \to CF(\phi_{t_2})$ satisfies
\[ \ddt \val_q( (E_{t_1}^{t})^{-1}(x)) = |x| - 1 \in [ -2 ,\dim(L) -
1] \]
for self-intersection points $x \in \cI(\phi_t)$ by \eqref{ongen}.
Hence the forward flow maps $MC^{> E}(\phi_{t_1})$ to
$MC^{> E - (t_2 - t_1)(1 - \dim(L))} (\phi_{t_2})$, and preserves the
potentials up to an overall power of $q$.  Similarly, the reverse
Euler flow maps $MC^{> E}(\phi_{t_2})$ to
$MC^{> E - 2(t_2 - t_1)} (\phi_{t_2})$, \label{gt} and preserves the
potentials up to an overall power of $q$.  Taking the directional
derivative of $\mu_0^b(1)$ in the direction of a class $c$ implies
that $HF(\phi_{t_1},b_{t_1})$ is mapped to $HF(\phi_{t_2}, b_{t_2})$
isomorphically.
\end{proof}

In order to complete the proof of Theorem \ref{transmain}, it remains
to deal with triple intersections, that is, triples
$x_1,x_2,x_3 \in L$ with $\phi(x_1) = \phi(x_2) = \phi(x_3) $.

\begin{figure} 
\begin{center}
%% Creator: Inkscape inkscape 0.91, www.inkscape.org
%% Accompanies image file 'triangle.pdf' (pdf, eps, ps)
\begingroup%
  \makeatletter%
  \providecommand\color[2][]{%
    \errmessage{(Inkscape) Color is used for the text in Inkscape, but the package 'color.sty' is not loaded}%
    \renewcommand\color[2][]{}%
  }%
  \providecommand\transparent[1]{%
    \errmessage{(Inkscape) Transparency is used (non-zero) for the text in Inkscape, but the package 'transparent.sty' is not loaded}%
    \renewcommand\transparent[1]{}%
  }%
  \providecommand\rotatebox[2]{#2}%
  \ifx\svgwidth\undefined%
    \setlength{\unitlength}{273.18852222bp}%
    \ifx\svgscale\undefined%
      \relax%
    \else%
      \setlength{\unitlength}{\unitlength * \real{\svgscale}}%
    \fi%
  \else%
    \setlength{\unitlength}{\svgwidth}%
  \fi%
  \global\let\svgwidth\undefined%
  \global\let\svgscale\undefined%
  \makeatother%
  \begin{picture}(1,0.4446314)%
    \put(0,0){\includegraphics[width=\unitlength,page=1]{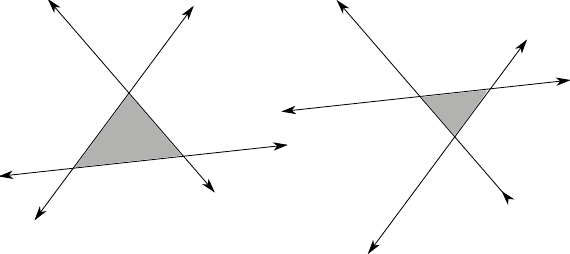}}%
    \put(0.20418892,0.32730756){\color[rgb]{0,0,0}\makebox(0,0)[lb]{\small \smash{$x_1$}}}%
    \put(0.34469919,0.1557371){\color[rgb]{0,0,0}\makebox(0,0)[lb]{\small \smash{$x_2$}}}%
    \put(0.06072052,0.12467694){\color[rgb]{0,0,0}\makebox(0,0)[lb]{\small \smash{$x_3$}}}%
    \put(0.79433208,0.14834179){\color[rgb]{0,0,0}\makebox(0,0)[lb]{\small \smash{$x_1$}}}%
    \put(0.66121712,0.29181024){\color[rgb]{0,0,0}\makebox(0,0)[lb]{\small \smash{$x_2$}}}%
    \put(0.91709368,0.312517){\color[rgb]{0,0,0}\makebox(0,0)[lb]{\small \smash{$x_3$}}}%
  \end{picture}%
\endgroup%
\end{center}
\caption{Behavior under a triple intersection} 
\label{triangle}
\end{figure} 

\begin{lemma} {\rm (Triple intersection lemma)}  
  \begin{enumerate} 
  \item \label{four} {\rm (Dimension greater than two)} Suppose that
    $\dim(X) \ge 4$.  Let $\phi_0: L \to X$ be a Lagrangian immersion
    without triple intersections and with only transverse
    self-intersections $L \times_\phi L - \Delta_L$.  Then for any $l \ge 2$, there exists an open
    $C^l$-dense set of Maslov flows $\phi_t: L \to X, t \in [0,T]$
    between $\phi_0$ and $\phi_T$ for which there are no triple
    self-intersection points;
  \item \label{twoinv}  {\rm (Dimension two)}  In dimension $\dim(X) = 2$, immersed Lagrangian Floer theory
    $HF(\phi_t,b_t)$ is invariant under a Maslov flow
    $\phi_t: L \to X, t \in [0,2]$ that develops a triple intersection
    (but no quadruple intersections).  \end{enumerate}
\end{lemma}

\begin{proof} First we give the proof of \eqref{four}, which uses a
  Sard-Smale argument and a dimension count.  Suppose that
  $\dim(X) \ge 4$.  Consider the universal space of triple points
  \begin{multline} \label{triple} \M^{\univ,\on{trip}} := \Set{  (t_0, \phi , H, x_1,x_2,x_3) |
\begin{array}{l} x_1,x_2,x_3 \ \text{distinct} \\  
   \phi_{t_0}(x_1) = \phi_{t_0}(x_2) = \phi_{t_0}(x_3) \\ \dot{\phi_t} =
    \phi_t^*(\d H_t + 
    \alpha_t  - \alpha_L)  \}  \end{array} }  \\ \subset [0,1] \times
    \Lag(L,X)_l \times C^l([0,1] \times X)  \times L^3 .\end{multline}
  By assumption the self-intersections of $\phi_0$ are transverse.  As
  in Theorem \ref{sardsmale}, Hamiltonian variations generated by
  $H_t \in C^l(X)$ together with $(D_x \phi)( T_x L)$ span the tangent
  space $T_{\phi(x)} L $ at any point $x \in L$.  It follows that the
  locus \eqref{triple} is cut out transversally.  By the Sard-Smale
  theorem,  regular values of the projection
\[ \pi: \M^{\univ,\on{trip}} \to C^l([0,1] \times X)  \] 
are comeager.  For regular values $H$, the submanifold of triple
intersections $\pi^{-1}(H)$ is of expected dimension $-1$, hence
empty.

The argument for \eqref{twoinv} is rather different, since triple
intersections cannot be avoided in families; instead we must examine
the holomorphic triangles more carefully.  Suppose that $\dim(X) = 2$.
For a triple intersection point of $\phi_0$, denote the intersection
points $x_1,x_2,x_3 \in L $ with
$\phi_0(x_1) = \phi_0(x_2) = \phi_0(x_3)$.  If a holomorphic polygon
$u: S \to X$ with boundary in $\phi$ has only one of these points
$x_k, k \in \{ 1, 2, 3 \}$ as a vertex, then that $u$ extends in a
continuous family $u_t$ with boundary in $\phi_t$ for all $t$ past the
triple intersection time $0$.  On the other hand,
$ b_t \in MC(\phi_t)$ for $t < 0$ implies that if the coefficient of
two or more $\ol{x}_i,\ol{x}_j, i,j \in \{ 1,2,3 \}$ of the odd
generators $\ol{x}_1,\ol{x}_2,\ol{x}_3$ in the element
$b_t \in CF(\phi_t)$ are non-zero, then the triangle in Figure
\ref{triangle} implies that the Floer cohomology is obstructed.
Indeed, if say the coefficients $b(\ol{x}_1),b(\ol{x}_2)$ are non-zero
and the triangle has area $A$, smaller than the area of any other
holomorphic polygon contributing to $\mu_0^b(1)$ then we would have
\begin{equation} \label{smalltri} \mu_0^b(1) = \mu_2(b,b) =b(\ol{x}_1)
  b(\ol{x}_2)  q^A x_3 + \text{higher\ order\ in } q.  \end{equation}
So at most one element of $x_1,x_2,x_3$ has non-zero coefficient in
$b_t$.  It follows that the disks $u_t: C \to X$ that meet two or more
of the generators $x_1,x_2,x_3$ do not contribute to the maps
$\mu_d^{b_t}(1)$.
\end{proof}

\section{Curve shrinking and gluing at a tangency} 
\label{curveshrinking} 

In this section we modify estimates of Ekholm-Etnyre-Sullivan
\cite{ees} to show invariance under the birth or death of a pair of
self-intersection points. The Floer differential contains a term of
small $q$-valuation that connects the two new generators.  This makes
the Floer cohomology \label{tobe}  invariant under the birth or death, while the
modification of the curvature of the Fukaya algebra can be cancelled
by a suitable weakly bounding cochain. 

The main result \label{rresult} of this section is a description of how the
holomorphic disks change as the immersion passes through a
self-tangency, and is a modification of the results of \cite[Chapter
10]{ees} to the setting of treed disks.  Let
$(\phi_t:L \to X)_{t \in [-\eps,\eps]}$ be a family of Lagrangian
immersions with a single value $t=0$ for which there is a birth of a
pair of self-intersection points at ${v} \in \phi_0(L)$.  Let
\[\zz,\bzz, \bz,\mz \in \cI^{\on{si}}(\phi_\eps) \] 
denote the additional generators in the Floer cochains for $\phi_\eps$
in relation to $\phi_{-\eps}$ depicted in the local model in Figure
\ref{generators}, where $\bzz,\bz$ are distinguished as the generators
that are the incoming corners of the small strip in Figure
\ref{small}; that is, with the curved branch $L_2$ in
\eqref{arcproduct} ordered before the flat branch $L_1$ for the
intersection with negative $q_1$ coordinate, and the reverse for the
intersection with positive $q_1$ coordinate.  We say that
$\ul{x}_t \subset \cI^{\on{si}}(\phi_t)$ is an {\em admissible} family
of generators if and only if \label{admfam}
\begin{center}
  $\phi(x_{i,t}) \to {v} , x_{i,t} \in \ul{x}_t, i > 0 $ resp. $i = 0$
  implies $x_{i,t} \notin \{ \zz, \mz \} $ resp.
  $x_{i,t} \in \{\zz,\mz \}$; \end{center}
that is, the generators $\bzz, \bz$ are excluded as outputs and
$\zz,\mz$ are excluded as inputs.  In the following the reference to
the Donaldson hypersurface $D$ may be dropped to simplify notation so
that the moduli spaces of holomorphic treed disks will be denoted
$\ol{\M}(\phi_t)$.

\begin{theorem} \label{smallstrip} Let $\phi_t: L \to X, t \in [-T,T]$
  be an admissible family of Lagrangian immersions with a
  self-tangency at $t = 0$ at $y \in X$.  Suppose that admissible
  perturbations $\ul{P} = (P_\Gamma)$ have been chosen for $\phi_0$
  using some Donaldson hypersurface $D$.  Then the same hypersurface
  $D$ and collection $\ul{P} = (P_\Gamma)$ of perturbations are
  regular for adapted holomorphic disks with boundary in $\phi_t, t
  \in [-\eps,\eps]$ for $\eps$ sufficiently small and 
  \begin{enumerate}
  \item \label{sstrip} {\rm (Small strips, similar to Lemma 2.14 in
      \cite{ees})} There exists $\eps > 0$ such that for any
    $t \in (0,\eps)$ there exist holomorphic strips of index one
\[ [ u_+(t): C \to X]  \in \M(\phi_t,\bzz,\mz), 
\quad [ u_-(t): C \to X] \in \M(\phi_t,\ol{v}_-,\zz),
 \]
 connecting $\bzz$ to $\mz$ resp. $\bz$ to $\zz$ with area
 $A(u_\pm(t)) \to 0$ as $t \to 0$, as in Figure \ref{small}, and
 any strip $u: C \to X$ with boundary in $\phi_t$ between $\bzz$ and
 $\mz$ of sufficiently small energy $E(u)$ is equal to $u_+(t)$ or
 $u_-(t)$.

\item \label{cshort} {\rm (Curve shrinking, similar to Proposition
  2.16 in \cite{ees})} Let $\ul{x}_s \in \cI(\phi_s)$ for
  $s \in [0,\delta]$ be a admissible family of self-intersection
  points/critical points.  There exist $\varkappa > 0$ such that for any
  $s \in (-1/\varkappa,0)$ there exists an orientation-preserving bijection
  \[ G_s: \M_{\Gamma_0}(\phi_0, \ul{x}_0)_\rho \to U_s \subset
  \M_{\Gamma_s}(\phi_s, \ul{x}_s )_\rho - \{ u_\pm \} \]
  where $u_\pm$ are the strips from the previous item, $U_s$ is a
  $C^0$ neighborhood of $\M_{\Gamma_0}(\phi_0,\ul{x}_0)_\rho$ in
  $ \M_{\Gamma_s}(\phi_s, \ul{x}_s )_\rho$, with the property that
  $\lim_{s \to 0} G_s(u) = u$.  See Figure \ref{split}.

\begin{figure}
\begin{center}
\includegraphics[width=4.5in]{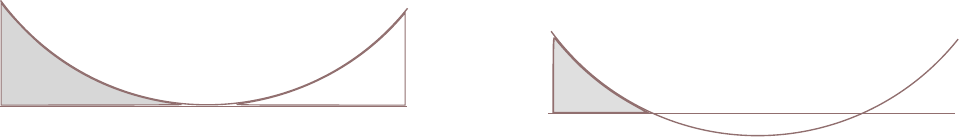} 
\end{center}
\caption{Shrinking a holomorphic curve at a tangency} 
\label{split} 
\end{figure} 

\item \label{gtang} {\rm (Gluing at a tangency, similar to Proposition
    2.17 in \cite{ees})} Let $\Gamma_0$ be a combinatorial type of
  tree disk $E_0 \subset \Edge(\Gamma)$ a subset of the finite edges
  corresponding to nodes $q_e \subset C$ mapping to the self-tangency
  point $x$, and no semi-infinite edges $T_e \subset T$ map to the
  self-tangency $x$.  Let $\Gamma_t$ denote the combinatorial type
  obtained by collapsing the edges $E_0$.  Suppose that the almost
  complex structure $J$ is in standard form near the self-tangency.
  Then there exists $\varkappa > 0$ such that for $s \in (0,1/\varkappa)$ there
  exists a bijection
\[ G_s: 
\M_{\Gamma_0}(\phi_0)_\rho \to U_s \subset \M_{\Gamma_s}(\phi_s)_\rho \]
where $U_s$ is a Gromov neighborhood of
$\M_{\Gamma_0}(\phi_0,\ul{x}_0)$ in
$ \M_{\Gamma_s}(\phi_s, \ul{x}_s )_\rho$ , with the property that
$\lim_{s \to 0} G_s(u) = u$.  See Figure \ref{glue}.  After
multiplying the given orientations by the sign \eqref{heartsuit} in
the definition of the structure maps the map $G_s$ is orientation
preserving.  Furthermore, given any sequence
$u_\nu \in \M_{\Gamma_s}(\phi_s)_\rho$ there exists a subsequence
converging to some $u \in \M_{\Gamma_0}(\phi_0)_\rho$.
\end{enumerate} 
\begin{figure}
\begin{center}
\includegraphics[width=3in]{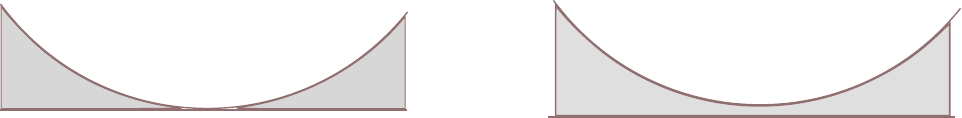} 
\end{center}
\caption{Gluing a holomorphic curve at a tangency} 
\label{glue} 
\end{figure} 
\end{theorem}

\begin{proof}[Proof of Theorem \ref{smallstrip} \eqref{sstrip}]
  The proof is essentially that of Lemma 2.14 in \cite{ees}, but we
  reproduce a proof for completeness.  
See also \cite{yjl} for similar gluing estimates. 

Note first that the Lagrangians
  in the local model \eqref{localmodel} are products.  Hence any
  holomorphic curve of $u: C \to X$ may be written in a neighborhood
  of $y$ as a product of maps $u_1: C \to X$ and $u_2 : C \to X$ with
  boundary in $L_1',L_1''$ and $L_2(s)', L_2(s)''$.  Let
  $u_1: \bR \times [0,1] \to \C$ denote the holomorphic strip with
  boundary on $L_1',L_2(s)''$ whose image is the region
\[x_1^2 + (y_1 - 1 \pm t) \in \left[0,\sqrt{ 1 - x_1^2} \right]  .\] 
This is the shaded region in Figure~\ref{small}, which exists by the Riemann
Mapping Theorem, see \cite{warsch}.  Let $u_2,\ldots, u_n$ be the
constant maps in the other components. Then $u = (u_1,\ldots, u_n)$ is
the desired map.

We claim that there are no other maps of equally small energy.  By
energy quantization Lemma \ref{hbar} there exists $\hbar$ so that any
non-trivial holomorphic polygon mapping a corner to an element
$x \in \cI(\phi_t)$ not equal to $y$ has energy at least $\hbar$.  Fix
$\eps$ sufficiently small so that $A(u(\eps)) < \hbar/2$.  Any other
map $u':C' \to X$ of area at most, say, $2A(u(\eps))$ must be such
that each component of the surface part $u'_{S'}$ meets the given
neighborhood of $y$.  The surface part $u'_{S'}$ is contained in a
small neighborhood of the intersection point, by a diameter estimate
involving the mean value inequality explained in Sikorav
\cite[4.4.1]{sikorav} (for the closed case; the case with Lagrangian
boundary is the same). Then the components $u'_2,\ldots, u'_n$ must be
constant by the maximum principle, while the first component is, up to
translation, the map $u$ given above up to an automorphism of
$\bR \times [0,1]$.
\end{proof}

The proofs of the shrinking and gluing parts Theorem \ref{smallstrip}
are generalizations of results of Ekholm-Etnyre-Sullivan \cite{ees}
that rely on a local model for the tangency. 
We prove parts~\eqref{cshort} and~\eqref{gtang} of Theorem~\ref{smallstrip}
in the remainder of this section.
As for similar gluing
results for pseudoholomorphic curves, the results depend on a
quantitative version of the implicit function theorem used by Floer
to prove gluing of trajectories:

\begin{lemma} \label{picard} {\rm (Floer's Picard Lemma,
    \cite[Proposition 24]{floer:monopoles})} Let $f : V_1 \to V_2$ be
  a smooth map between Banach spaces that admits a Taylor expansion
  $f(v) = f(0) + df(0)v + N(v)$ where $df(0): V_1 \to V_2 $ is
  Fredholm and has a right inverse $G:V_2 \to V_1$ satisfying the
  uniform bound
\[ \Vert GN(u) - GN(v) \Vert \leq C( \Vert u\Vert  + \Vert v \Vert)\Vert
u - v \Vert \]
for some constant $C$. Let $B(0, \eps)$ denote the $\eps$-ball
centered at $0 \in V_1$ and assume that
\[\Vert Gf(0) \Vert \leq \frac{1}{8C} .\]
%see e..g 8.2 in https://projecteuclid.org/download/pdf_1/euclid.jdg/1143651770
Then for $\eps< \frac{1}{4C}$, the zero-set of $f^{-1}(0) \cap B(0,\eps)$ is a
smooth submanifold of dimension $\dim(Ker(df(0)))$ diffeomorphic to
the $\eps$-ball in $\Ker(df(0))$. \end{lemma}

We review the Sobolev spaces and estimates from Etnyre-Ekholm-Sullivan
\cite{ees} necessary to prove shrinking.  Given a curve $u: C \to X$
with boundary in $\partial u: \partial C \to L$.  In the local model
of a self-tangency move \eqref{localmodel} $u$ is a product of components $u_1,\ldots,u_n$ to
$\C$ with the given Lagrangian boundary conditions.  By the local
model, we may assume that $L$ is locally the union
$\phi(L) \cap U \cong L_1 \cup L_2$ of components $L_1 , L_2$ with
$L_1 \subset \bC^n$ linear and $L_2 \subset \bC^n$ the product of a part
of a circle of radius one with a linear space, whose center will move
down as one performs the isotopy.  Up to rescaling this is a mean
curvature flow.

\begin{definition} {\rm (Pre-shrinking)}
  We first define a simplified family of immersions developing a
  self-tangency.  Let $U$ be a ball of radius $(R + \hh R^a)^{-1}$
  around the self-tangency point ${v}$, the origin in local
  coordinates, where $a \in (0,1)$ is close to $1$.  Let
  $b_R \in \bC^\infty(\R)$ be a non-decreasing function with support in $[0,R^{-1})$
  and
\[ b_R(r)=(R+R^a)^{-2}, \quad \forall r \in \left[ 0,(R+\hh R^a)^{-1} \right) \]
and
\[ \sup |D^k b_R| \leq O(R^{2k -
    (2k+1)a}), k \leq 3 .\] 
Let
\[h_R: \bC^n \to \R, \quad z \mapsto \on{Re}(z_1)  b_R(|z|) .\]
For $s > 0$ let 
\[ \Psi_R^s:X \to X \]  
denote the time $s$ Hamiltonian flow of $h_R$.  Thus $\Psi_R^s$ is a
translation by $(R+R^a)^{-2}s$ \label{prefactor} in a small ball of
radius $(R + \hh R^a)^{-1}$ around the origin, and the identity
outside a ball of radius $R^{-1}$. Let
\[ L_2^R(s) = \Psi_R^s(L_2), \] 
and let
\[ \phi_s^{\loc}: L \to
X \]  
denote the immersion obtained by patching together $\phi$ outside of
the ball around $x$ with the immersions with images $L_1, L_2^R(s)$.
The family $\phi_s^{\loc}$ is a Maslov flow up to rescaling locally,
by Example \ref{elem}, since the change in radius in the circle is
equivalent up to dilation centered at a point on the branch $L_1$ to a
translation.  \label{elemref}

We now define an approximately holomorphic disk ending at one of the
self-intersection points near the tangency.  Let
$\beta:[0,1] \to [0,1]$ be a smooth surjective map which is constant
in a small neighborhood of $\{ 0, 1 \}$.  Let $u: C \to X$ be a treed
holomorphic disk and $u_R$ the deformation of $u$ in a neighborhood of
the self-tangency given by
\[u_{R}: \bR \times [0,1] \to X, \quad u_{R}(\tau + it) =
\Psi_R^{\beta(t)} (u( \tau + it)) \]
so that $u_{R}$ has the deformed Lagrangian boundary conditions. 

The pre-shrunk map is defined by cutting off the variation of the
local model of the previous paragraph near the intersection point. Let
$\varkappa = \varkappa(R)$ be such that the intersection points of $L_1$ and
$L^{R(\varkappa)}_2(s)$ 
are
\[ L_1 \cap L^{R(\varkappa)}_2(s) = \{ \pm ( (\varkappa
+ \varkappa^a)^{-1},0,\ldots, 0 ) \}  .\]
Choose a family of metrics $g(R,s,t)$ such that the Lagrangian $L_1$
is totally geodesic for $t =0$ and $L^{R(\varkappa)}_2$ is totally geodesic for
$t = 1$. Locally near each node $q_k$ mapping to the self-tangency the
map $u_{R}$ may be written as the geodesic exponential of a section of
the tangent bundle
\[ \xi_{R,k}:  E_{p_k}(-\varkappa) \to  \bC^n , \quad 
\exp^{R,t}( \xi_{R,k}(\tau + it)) = u_R(\tau + it), \quad \forall j
.\]
Define a treed disk $w_\varkappa: C \to X$ by patching as follows:
\[ w_\varkappa(\zeta) = \begin{cases} 
u(\zeta)                               & \zeta \in C \backslash 
\cup_k E_k(\varkappa)  \\
\exp^{R,t}(\xi_{R,k}(\zeta))  & \zeta \in E_k(\varkappa) \\
v (\zeta)                                           & \zeta \in \cup_k 
E_k(-\varkappa - \varkappa^a)
\end{cases}. 
\]
This ends the definition. 
\end{definition} 

We claim that there is a unique holomorphic treed disk near the
approximately-holomorphic disk defined in the previous paragraph.
Since the almost complex structure $J_\Gamma$ in a neighborhood $U$ of
the self-tangency ${v}$ is assumed to be standard
$J_\Gamma |_U = J_{\bC^n}$, and so invariant under rescaling, it will
suffice \label{suffice} to prove Theorem \ref{smallstrip}
\eqref{cshort} for the family $\phi_s^{\loc}$.  The Sobolev spaces
used for shrinking are defined using a weight function constant on the
ends.

\begin{definition} {\rm (Weighted Sobolev spaces)} Let $\lambda > 0 $
  be a small weight and $e_\varkappa: C \to \R$ a function with
\begin{equation} \label{weightfunction} e_\varkappa(\tau + it) 
  = \begin{cases} e^{- \lambda \tau} & \tau \in (0,
    \varkappa) \\
    e^{- \lambda \varkappa} & \tau > \varkappa \end{cases} .\end{equation}
Let $ \B_{2,-\lambda,\varkappa}$ denote the Banach manifold of treed disks
$u: C \to X, \partial u: \partial C \to L$ bounding $\phi$ such that
on each component, $u$ and $\partial u$ are of finite $e_\varkappa$-weighted
$W^{2,2}$ norm as in the proof of Proposition \ref{prop:cutout}.  Let $\cE_{1,-\lambda,\varkappa}$ be
the corresponding bundle of one-forms as in \eqref{bundlefiber}.
\end{definition} 

The following are the necessary ingredients in the Picard lemma:

\begin{proposition} \label{shprop} {\rm (Zeroth, first, and second
    order estimates for shrinking)}
\begin{enumerate} 
\item \label{zerothsh} There exists a constant $C_0 > 0$ such that for all
  $\lambda > 0$ sufficiently small and $\varkappa > 0$ sufficiently large,
  the pre-glued section $w_\varkappa$ defined above has norm
\[ \Vert \cF_{w_\varkappa}(0) \Vert_{1,-\lambda, \varkappa} 
\leq C_0 e^{-\lambda \varkappa} \varkappa^{ -1 - a/2} .\]
\item  \label{shc1} There exists a constant $C_1 > 0$ such that the linearized operator at
  the pre-glued solution satisfies the uniform estimate
  \begin{equation} \label{uniform} \Vert \xi \Vert_{2,-\lambda,\varkappa}
    \leq C_1 \varkappa^{1 + \delta} \Vert \ti{D}_{w_\varkappa} \xi
    \Vert_{1,-\lambda,\varkappa} .\end{equation}
\item There exists a constant $C_2 > 0$ such that the non-linear term
  in the map $\cF_{w_\varkappa}$ satisfies
\[   \Vert N(\xi_1) - N(\xi_2,) 
\Vert_{1,-\lambda,\varkappa} 
\leq C_2 e^{\lambda \varkappa} ( \Vert \xi_1 \Vert_{2,-\lambda,\varkappa} + \Vert \xi_2 
\Vert_{2, - \lambda, \varkappa} )  \Vert \xi_1 - \xi_2 \Vert_{2 , -\lambda ,
  \varkappa} .\]
\end{enumerate} 
\end{proposition}

\begin{proof}[Sketch of proof] We sketch the arguments, without going
  into full detail as in Ekholm-Etnyre-Sullivan \cite[Chapter
  8]{ees}. The exponential decay factor in the zero-th order estimate
  \eqref{zerothsh} arises from the choice of Sobolev weighting
  function and the second factor arises from the cutoff used to define
  the map $w_\varkappa$, see \cite[Lemma 8.10]{ees}.  The first order
  estimate is similar to \cite[Lemma 8.11]{ees}.  Suppose otherwise so
  that there exists a sequence $\xi_\nu, \nu \in \bZ_{> 0}$ with norm
  one such that the right-hand-side of \eqref{uniform} approaches zero
  for some sequence of gluing parameters $\varkappa_\nu$. Let
  $\alpha_1: C \to \R$ be a cutoff function equal to $0$ in a
  neighborhood of the markings mapping to the self-intersection and
  with first and second derivatives of order $\varkappa^{-a}$; then
\[ \Vert \alpha_1 \xi_\nu \Vert_{2,-\lambda,\varkappa} \leq C \Vert (\olp \alpha_1) 
\xi_\nu \Vert_{1,-\lambda,\varkappa} + \Vert \alpha \ti{D}_{w_\varkappa} \xi_\nu
\Vert_{1,-\lambda,\varkappa} .\]
The first term is order $\varkappa^{-a}$ because of the estimates on the
cutoff function while the second is order $\varkappa^{-1 - \delta} $ by
assumption on $\xi_\nu$. On the other hand, let $\alpha_2 : C \to \R$
be a cutoff function equal to $1$ on the strip-like ends, with first
and second derivatives of order $\varkappa^{-a}$. Let
$\hat{\varphi}: (-\infty, -\varkappa + \varkappa^a) \to \R$ be the function that
equals the difference of angles $\theta(\varkappa) - \theta(\tau)$ between
the tangent line $L_2^\varkappa(1)$ and the real line. The function
$\hat{\varphi}$ extends to a cutoff function
\[\varphi: (-\infty ,-\varkappa + \varkappa^a) \times [0,1] \to \C\] 
with 
\begin{eqnarray*}
  0 &=& \varphi(\tau + it), \quad (\tau,t) \in (-\infty, -\varkappa + \varkappa^a) \times \{ 0\}\\ 
  \hat{\varphi}(\tau) &=& \varphi(\tau + it) , \quad (\tau,t) \in   (-\infty, -\varkappa + \varkappa^a) \times \{ 1 \} \end{eqnarray*}
with the same derivative estimates. Let 
\[\mu = \diag(\varphi,1,\ldots, 1) .\]
Then  we have 
\begin{eqnarray*} \Vert \alpha_2 \mu \xi_\nu
  \Vert_{2,-\lambda,\varkappa_\nu} &\leq& C \Vert e^{-\lambda
    \varkappa_\nu} \alpha_2 \mu \xi_\nu \Vert_{2,\varkappa_\nu} \\ &\leq& C \varkappa ( \Vert
  e^{-\lambda \varkappa_\nu } (\olp \alpha_2 \mu) \xi_\nu \Vert_{1,\varkappa_\nu} + \Vert
  e^{-\lambda \varkappa_\nu} \alpha_{2,\varkappa_\nu} \mu \ti{D}_u \xi_\nu
  \Vert_{1,\varkappa_\nu} \end{eqnarray*}
The first coefficient of $\xi_\nu$ arises from the fixed angle
$\theta(\varkappa_\nu) $ which is approximately $1/\varkappa_\nu$ for
$\varkappa_\nu$ large.  Thus we obtain that
$\Vert \xi_\nu \Vert_{2, -\lambda, \varkappa_\nu}$ converges to zero
as $\nu \to \infty$ which is a contradiction.  Finally the estimate
for the non-linear term is similar to \cite[Lemma 8.17]{ees} and
follows from
\[ \sup_{z \in C} |\xi (z) |  \leq e^{\lambda \varkappa} \Vert \xi \Vert_{2,
  -\lambda, \varkappa} .\qedhere \]
\end{proof} 

Floer's version of the Picard Lemma \ref{picard} and Proposition
\ref{shprop} produce the desired family of holomorphic treed disks
with Lagrangian boundary condition in $\phi_s^{\loc}$.  The map
$\phi_s^{\loc}$ is related to $\phi_s$ by an isotopy of
self-transverse Lagrangians of $X$, and so $\phi_s^{\loc}$ and
$\phi_s$ are related by a family of diffeomorphisms of $X$ which
approaches the identity as $\varkappa \to \infty$.  As in Corollary
\ref{implicit}, in such a setting the spaces of holomorphic curves are
in bijection.  This proves the gluing part in \eqref{cshort} of
Theorem \ref{smallstrip}, that is, the existence of the map.  The
Gromov convergence 
\[ \lim_{s \to 0} G_s(u) = u \]  
follows since $E(G_s(u)) \to E(u)$ as $s \to 0$ and each component of
the limit is obtained by a suitable rescaling sequence, by
construction.

The existence of an inverse (that is, surjectivity of the above
construction onto the space of nearby solutions) follows from the
exponential decay estimates in Lemma \ref{expdecay} \eqref{texpdecay}.
We break the region bounded by $\on{Re}(u_{s,1}) \in \{ 1, 2/R \}$ the
branches $L_1,L_2(s)$ into regions
$\Omega_s, \Omega^+_s, \Omega^{++}_s$.  On the region $\Omega^{++}_s$
on which $\on{Re}(u_{s,1}) \ge 1$ we have uniform convergence in all
derivatives.  So 
\[ u_s = \exp_{w_s} (\xi(s)), \text{ for some} \ \xi(s), \quad \Vert
\xi(s) \Vert < C \exp(- \lambda \varkappa) .\]   
On the region $\Omega^+_s \cong [-T(s),T(s)] \times [0,1]$ we write
$u_{s,1}$ in terms of its Fourier coefficients
\begin{multline} \Theta_s \circ u_{s,1}(\tau + i t) = (\tau -
  \tau_0(s)) + i t + \sum_{n > 0} c_n(s) \exp( 2 \pi n (\tau - T(s))
  +i t) \\ + \sum_{n < 0} c_n(s) \exp( 2 \pi n (\tau + T(s)) +i t)
  .\end{multline}
Uniform convergence at $\on{Re}(u_{s,1}) = 1$ implies the Fourier
constants $c_n(s), n > 0 $ converge to the Fourier constants $c_n(0)$
of $u_{0,1}$ as $s \to 0$, while the Fourier coefficients
$c_n(s), n < 0$ converge to zero by exponential decay of $u_{s,1}$.
It follows that
\[ \sup | (u_{s,1} - w_{s,1})(\tau + it)  |^p  + | (\d u_{s,1} - \d
w_{s,1})(\tau + it) |^p \leq  C \exp(-p |\tau - T(s)|) \quad \text{on} \ \Omega^+_s .\]
Thus for any $\eps$, there exists $s_0 > 0$ such that for $s < s_0$,
\[ \Vert ( u_s - w_s ) |_{\Omega+_s} \Vert_{1,p,-\lambda} \leq \eps 
\exp( - \varkappa \lambda) .\]
Finally the restriction of $u_{s,1} - w_{s,1} | \Omega_s$ to
$\Omega_s \cong [-T(s), T(s)] \times [0,1]$ converges to zero
uniformly, as well as its derivative.  Thus for any $\eps > 0$, there
exists $s_0 > 0 $ such that if $s < s_0$ then
\[ \sup | (u_{s,1} - w_{s,1})(\tau + it) |^p + | (\d u_{s,1} - \d 
w_{s,1})(\tau + it) |^p \leq \eps \exp( - \varkappa \lambda) .\]
It follows that
\[ \Vert u_s - w_s |_{\Omega_s} \Vert_{1,p,-\lambda} \leq \eps \exp( -
\varkappa \lambda) .\]
Putting everything together we find that for any $\eps > 0$, there
exists $s_0 > 0 $ such that if $s < s_0$ then 
\[ u_s = \exp_{w_s} (\xi(s)), \text{ for some} \ \xi(s) \ \text{with}
\ 
\ \Vert \xi(s) \Vert_{1,p,\lambda} \leq \eps \exp( - \varkappa \lambda)
. \] 
Since this quantity is less than the quantity $1/4C$ in Floer's Picard
Lemma \ref{picard}, for $\eps$ sufficiently small, the map $u_s$ is
the solution given by the Picard iteration.

The gluing at a tangency part \eqref{gtang} of Theorem
\ref{smallstrip} follows from a similar application of Floer's Picard
Lemma \ref{picard}. Similar to the set-up for shrinking, for $R \gg 0$
let $a_R: [ 0,\infty] \to \R$ be a smooth non-increasing function with
support in $[0, R^{-1}/2)$ with the properties
\[ a_R(r) = 1/R, \quad r \in [0,1/R^2]; \quad 
\sup_r | D_r a_R | = O(1); \quad \sup_r | D^2_r a_R | = O(R) .\]
Let 
\[h_R:\bC^n \to \R, \quad h_R(z) = \on{Re}(z_1) a_R(|z_1|) .\] 
For $s> 0$ let $\Phi_R^s$ denote the time $s$ Hamiltonian flow of
$h_r$. Let $L_2^R(s)$ denote the Lagrangian submanifold obtained by
applying $\Phi_R^s $ to $L_2$, and let $\gamma^1(s), \gamma^2_R(s)$ be
paths tracing out $L_1$ and $L_2^R(s)$. Let $\phi_s^{\loc}: L \to X$
be the family of immersions obtained by gluing in the local family to
the immersion $\phi: L \to X$.  Let $u: C \to X$ be a rigid
holomorphic tree disk with boundary on $\phi$ and
$I \subset \{ 0,\ldots, n \}$ such that the points $z_i, i \in I$ on
the boundaries of the disk map to the self-intersection points. Let
$s = (K \varkappa)^{-1}$ as below or in \cite[Lemma 10.18]{ees} and
let $\Xi_\varkappa$ denote the neck region bounded by the curves
$\gamma^2_R(s), \gamma^1(s)$ and $u^{i_\pm(k)}(\varkappa + it)$ for
$t \in [0,1]$.  By the Riemann mapping theorem for a unique
$A(\varkappa)$ there exists a biholomorphic map from the interior of
rectangle
\begin{equation} \label{phirho} \varphi_\varkappa: [-A(\varkappa), A(\varkappa)] \times
  [0,1] =: \Omega_\varkappa \to \C
\end{equation} 
with a continuous extension with the properties 
\begin{eqnarray*} 
 \varphi_\varkappa( [-A(\varkappa),A(\varkappa)] \times \{ 0 \}) &\subset&
\on{Im}(\gamma_1 ) \\
 \varphi_\varkappa( [-A(\varkappa),A(\varkappa)] \times \{ 1 \}) &\subset&
\on{Im}(\gamma_2(s) ) \\
\varphi_\varkappa(0 + i[0,1]) &\subset& \{ \on{Re}(z_1) = 0 \} .\end{eqnarray*}

The pre-glued map is defined as follows. Let $u_{i_{\pm}(k)}^1$ denote
the first components of the maps $u_{i_\pm(k)}$ meeting the tangency
at nodes $q_k$, contained in components $C_{i_\pm(k)}$. By gluing
$C_{i_\pm(k)} - E_{q_k}[ -((1 + a)\varkappa ] $ with $\Omega_\varkappa$ one obtains a treed holomorphic disk on which the maps $u_{k,\pm}^1$
glue together in a neighborhood of the node $q_k$. Let $u_{k,\pm}'$
denote the remaining components of $u$ in the local model. Using the
standard metric on $\bC^{n-1}$ define maps $u'_{k,\pm}$ to $\bC^{n-1}$
by
\[ u'_{k,\pm}(\tau + it) = \exp ( \xi_{k,\pm}'(\tau + it) ) \]
where in this setting the geodesic exponentiation $\exp$ is the
identity map since the metric is constant. Let $E_{q_k,\pm}$ denote
the image of the local coordinate near $q_k$ on the component
$C_{i_\pm(k)}$ adjacent to the node $q_k$. Pick a cutoff function
$\alpha$ equal to $1$ on $C - \cup_{k,\pm} E_{q_k,\pm} [-\varkappa + 5]$
with support in the complement of
$\cup_{k,\pm} E_{q_k, \pm} [-\varkappa + 3]$, with derivatives uniformly
bounded for $\varkappa \gg 0$.  Define $w'_\varkappa: C_\varkappa \to \bC^{n-1}$ by
\[ w'_\varkappa(\zeta) = \begin{cases} 
  u'_{k,\pm}(\zeta)  &  \zeta \in C_{i_\pm(k)} - E_{q_k,\pm}[-\varkappa + 5] \\
  \exp_0^t(\alpha(\zeta) \xi_{\pm,j}(\zeta)) & E_{q_k,\pm}[-\varkappa + 5] - E_{q_k,\pm}[-\varkappa] \\
  0 & \zeta \in \Omega_\varkappa .
\end{cases} 
\]
Combining the components we define $w_\varkappa = (w^1_\varkappa, w_\varkappa')$.

A treed holomorphic disk nearby the preglued solution above is found
using the Picard lemma.  For a function $f: C_\varkappa \to \R$ and for
each component $C_{i_\pm(k)}$ adjacent to a node $q_k$ mapping to the
self-tangency we denote by $f_{i_\pm(k)}$ the restriction of $f$ to
$C_{i_\pm(k)} - E_k[-(1 + a)\varkappa]$.  Let $f_k$ denote the restriction
to the region $\Omega_k(\varkappa)$ near the node $q_k$.  For $\lambda > 0$
let $e_\lambda^{i_\pm(k)}$ denote the weight function on
$C_{i_\pm(k)}$ which equals $1$ on $C_{i\pm(k)} - \cup E_{q_k}$ and
equals $e^{\lambda |\tau|}$ in $E_{q_k}$.  Let
$\Vert \cdot \Vert_{k,\lambda,i}$ denote the Sobolev norm on
distributions with $k$ square-integrable derivatives obtained from the
weight function $e_\lambda^{i_\pm(k)}$.  Let $e_\lambda^j$ denote the
weight function which equals $e^{\lambda | \tau| } $ on $\Omega_k$ and
$\Vert \cdot \Vert_{k,\lambda,j}$ the resulting Sobolev norm.  Define
\begin{equation} \label{secondweightfunction} \Vert f
  \Vert_{k,\lambda,\varkappa} = \sum_{i \neq i_\pm(k)} \Vert f_i
  \Vert_{k,2} + \sum_j \Vert f_i \Vert_{k,\lambda,i_\pm(k)} + \sum_j
  \Vert f_j \Vert_{k,\lambda,j} .\end{equation}
Let 
\[ \cF_u: T_u \B_\Gamma \to
\cE_{\Gamma,u} \]  
denote the map from \eqref{banachmap} obtained from the given local
trivializations of $\B_\Gamma$ and $\cE_{\Gamma}$ near $u$. Let
$\theta \gg \lambda $ be the smallest non-zero K\"ahler angle of the
intersection where the angles are described in \eqref{angles}; see
also \cite[Lemma 8.13]{ees}.

%Let $M_\varkappa$ be the
 % complex manifold obtained by
 % gluing together $[ \pm (1-a)\varkappa,\pm \infty) \times [0,1]$ with 
 % $\Omega_\varkappa = [-A(\varkappa), A(\varkappa)] \times [0,1]$.  An 
 % $\lambda$-weighted norm is defined as in the previous paragraph, by 
 % summing weighted Sobolev norms on the three regions. 
%With respect to this Sobolev norm  one has the estimate 
%
%%\[ \Vert f \Vert_{k,\varkappa,\lambda} \leq C \Vert \olp f \Vert_{k-1,
%  \varkappa,\lambda} \]
%%
%for any function $f: M_\varkappa \to \C$ with real values on the boundary. 

  \begin{proposition} \label{gprop} {\rm (Zeroth, first, and second
      order estimates for gluing)}
\begin{enumerate} 
\item There exists $C_0 > 0$ such that for all $\varkappa > 0$
  sufficiently large the pre-glued solution $w_\varkappa$ satisfies
\[ \Vert \cF_{w_\varkappa}(0) \Vert_{1,\lambda,\varkappa} 
\leq C_0 e ^{(- \theta + 2 K \lambda) \varkappa} \]
\item \label{gc1} There exists $C_1$ such that for all $\varkappa$
  sufficiently large the linearization $\ti{D}_{w_\varkappa}$ of
  $\cF_{w_\varkappa}$ satisfies \cite[Equation (8.74), (8.90)]{ees}
\[   \Vert \xi \Vert_{2,\lambda,\varkappa} \leq C_1 \Vert \ti{D}_{w_\varkappa} \xi
\Vert_{1,\lambda,\varkappa} .\]
\item There exists $C_2 > 0 $ such that for all $\varkappa$
  sufficiently large the non-linear part of the map $\cF_{w_\varkappa}$
  satisfies a uniform estimate \cite[Lemma 8.18]{ees}
\[   \Vert N(\xi_1) - N(\xi_2) 
\Vert_{1,\lambda,\varkappa} 
\leq C_2 ( \Vert \xi_1 \Vert_{2,\lambda,\varkappa} + \Vert \xi_2 
\Vert_{2, \lambda, \varkappa} )  \Vert \xi_1 - \xi_2 \Vert_{2 , \lambda ,
  \varkappa} .\]
\end{enumerate}
\end{proposition} 

\begin{proof} The first is similar to \cite[Lemma 8.13]{ees}, the
  second similar to \cite[Lemma 8.15]{ees}, and the third to
  \cite[Lemma 8.16]{ees}.  All estimates are local and so apply to
  holomorphic treed disks with almost complex structure standard in
  a neighborhood of the tangency.  
\end{proof} 

Floer's version of the Picard Lemma \ref{picard} and Proposition
\ref{gprop} produce the claimed family of holomorphic treed disks
$u_s: C_s \to X$ with boundary in $\phi_{s}^{\loc}$ converging to $u$
as $\varkappa \to 0$.  As before, $\phi_{s}^{\loc}$ and $\phi_{s}$ are
related by a family of diffeomorphisms of $X$ which approaches the
identity in all derivatives as $s \to 0$.  The existence of an inverse
to $G_s$ (that is, ``surjectivity of gluing'') follows from
exponential decay for pseudoholomorphic curves of small energy.
Suppose that $u_s$ Gromov converges to $u$; we must show that for $s$
sufficiently small the solution $u_s$ is the one produced by the
Picard iteration, up to equivalence.  Because the almost complex
structure is constant and split near the self-tangency, it suffices to
consider each component $u_{s,i}, i = 1,\ldots, n$ of $u_s$
separately.  For the components $u_{s,i}$ with $i \neq 1$, the
Lagrangian boundary conditions are transverse and the Lagrangian
version of the annulus lemma (see for example, Frauenfelder-Zemisch 
\cite[Lemma 3.1]{totreal}) implies that for $s$ sufficiently large,
$u_{s,i}$ is arbitrarily close to the approximate solution $w_{s,i}$
in the Sobolev norm defined by \eqref{secondweightfunction}.  For the
first coordinate, we identify the region $\Omega'_s$ bounded by
$\on{Re}(z) \in \{ -1,1 \}$ and $\phi_s(L)$ with a rectangle
$[-T(s),T(s)] \times [0,1]$ of length $2T(s)$ by a biholomorphism
$\Theta_s$.  Write $\Theta_s \circ u_{s,1}$ in Fourier coefficients on
the neck region $[-T(s), T(s)] \times [0,1]$
\begin{multline} \Theta_s \circ u_{s,1}(\tau + i t) = (\tau -
  \tau_0(s)) + i t + \sum_{n > 0} c_n(s) \exp( 2 \pi n (\tau - T(s))
  +i t) \\ + \sum_{n < 0} c_n(s) \exp( 2 \pi n (\tau + T(s)) +i t)
  .\end{multline}
Let $c_n(0)$ be the similar Fourier coefficients of the derivative of
$\Theta_s \circ w_{s,1}$, where $w_{s,1}$ is the approximate solution,
and $\tau_0(0)$ the translation factor of the leading order term of
$\Theta_s \circ w_{s,1}$.  Convergence of $u_{s,1}$ to $u_{0,1}$ on
compact subsets and agreement of $\Theta_s \circ u_{0,1}$ with
$\Theta_s \circ w_{0,1}$ on a neighborhood of $\tau = \pm T(s)$ (where
both are defined) implies that $c_n(s) \to c_n(0)$ for all $n$.  After
a conformal variation that eliminates the difference
$\tau_0(s) - \tau_0(0)$ (that is, a small variation in length of the
neck) the difference satisfies
\[  | \Theta_s \circ u_{s,1}(\tau + it) - \Theta_s \circ w_{0,1}
(\tau+it) | < C ( \exp( - 2\pi (T(s) - \tau)) + \exp( - 2\pi (\tau - T(s))))
.\]
and similarly for the first derivatives.  We break $\Omega'_s$ into
three regions separated by the curves $\on{Re}(u_{0,1}) = s$: The
region $\Omega_s$ with $s = 1/\varkappa$ and the regions $\Omega_s^{\pm}$
between $\on{Re}(z)(u_{0,1}(s)) \in \{ -s, s\}$ and
$\on{Re}(z)(u_{0,1}(s)) \in \{ -1,1 \}$.  Since $u_{0,1}(z) \sim 1/z$,
the conformal modulus of $\Theta_s$ is order $1/z$ on $\Omega^\pm_s$
while $\Theta_s$ has conformal modulus bounded by $s$ on $\Omega_s$.
It follows that for any $\eps > 0$, there exists $s_0 > 0$ such that
for $s < s_0$
\begin{multline} \int_{ [-T(s),T(s)] \times [0,1]} ( | u_{s,1}(\tau +
  i t) - w_{s,1}( \tau + it) |^p \\ + | \d u_{s,1}(\tau + i t)
  - \d w_{s,1}( \tau + it) |^p  ) e_\varkappa(\tau + it) \d t \d \tau \\
  < C_1 \eps + C_2 \eps + C_3 \eps / (a - \lambda)
 \end{multline} 
 where the first term $C_1 \eps$ arises from the contribution of the
 complement of $\Omega'_s$ where $u_s$ converges uniformly in all
 derivatives, the second $C_2 \eps$ from the integrals over
 $\Omega^\pm_s$, which have exponentially decaying integrand, and the
 third $C_3 \eps / ( a - \lambda)$ from the integral over
 $\Omega'_\varkappa$, on which the integrand is also exponentially decaying
 at a rate depending on the minimum non-zero angle $a$ between the
 branches of the Lagrangian.  Combining the results for the various
 components implies that for $t$ sufficiently large, $u_t$ is
 arbitrarily close to the approximate solution $w_t$ in the Sobolev
 norm defined by \eqref{secondweightfunction}.  This proves part
 \eqref{gtang} of Theorem \ref{smallstrip} up to sign.

 The statement on regularity in Theorem \ref{smallstrip} follows from
 Lemmas \ref{shprop} \eqref{shc1} and \ref{gprop} \eqref{gc1}.
 Indeed, for $s$ sufficiently small, the linearized operators are
 surjectivity so all rigid holomorphic disks with boundary on $\phi_s$
 are regular.

 Finally we show that the gluing map in Theorem \ref{smallstrip}
 \eqref{gtang} is orientation preserving.  The sign computation for
 the gluing map is similar to that for the \ainfty associativity
 relation \eqref{ainftyassoc}, but in this case the determinant lines
 associated to the nodes $q_k \in S$ mapping to the self-tangency
 ${v} \in X$ have index shifted by one, which creates additional
 signs.  For the purpose of computing orientations it suffices to
 consider the case that there is a single node $q_k$ mapping to a
 tangency.  Recall that for any self-intersection point
 $x_j \in \cI^{\on{si}}(\phi_t)$ the notation $\DD^{\pm}_{x_j,2}$
 denotes the determinant line associated to $x_j$ in the Fredholm
 operator on the once-punctured disk associated to a choice of path
 from $\gamma_x$.  By assumption in \eqref{caniso}, the orientations
 on $\DD^\pm_{x_j,2}$ are defined so that the tensor products
 $\DD^-_{x_j,2} \otimes \DD^+_{x_j,2} \cong \R$ with the standard
 orientation on the trivial vector space, in the case of a
 self-transverse boundary condition, or a one-dimensional vector
 space, in the case of self-tangent boundary condition $x_j = {v}$; to
 simplify notation we assume that each vertex maps to a
 self-intersection point; otherwise the following discussion holds by
 the same argument replacing $\DD_{x_j,2}$ with $\DD_{x_j,1}$ in the
 notation for each vertex mapping to a Morse unstable manifold. By
 deforming the parametrized linear operator $\ti{D}_u$ of
 \eqref{linop} to the linearized operator $D_u$ plus a trivial
 operator, one obtains an isomorphism of determinant lines
\begin{equation} \label{convent} \det( T {\M}_d (x_0,\ldots,x_d)) \to  
  \det(T\M_d) \DD^+_{x_0,2} \DD^-_{x_1,2} \ldots \DD^-_{x_d,2}
  .\end{equation}
The gluing map for a single node
takes the form (omitting tensor products from the notation to save
space)
\begin{multline} 
  \det(T\M_m) \DD^+_{{v},2} \DD^-_{x_{n+1},2} \ldots \DD^-_{x_{n+m},2}
  \det(T\M_{d-m+1,2}) \DD^+_{x_0,2} \DD^-_{x_1,2} \ldots \DD^-_{{v},2}
  \ldots \DD^-_{x_d,2} .\end{multline}
To determine the sign of this map, first note that the gluing map
\[ (0,\eps) \times \M_m \times \M_{d-m+1} \to \M_d \]
on the associahedra $\M_k$ is given in coordinates (using the
automorphisms to fix the location of the first and second point in
$\M_m$ to equal $0$ resp. $1$ and $\M_{d- m + 1}$) by
\begin{multline} \label{signs} (\delta,(z_3,\ldots,z_m)
  ,(w_3,\ldots,w_{d-m+1}) ) \\ \to (w_3,\ldots, w_{n+1}, w_{n+1} +
  \delta, w_{n+1} + \delta z_3, \ldots, w_{n+1} + \delta z_m,
  w_{n+2},\ldots, w_{d-m}) .\end{multline}
The map \eqref{signs} acts on orientations by a sign of $-1$ to the power
\begin{equation} \label{signa} (m-1)(n-1). \end{equation}
These signs combine with the contributions
\begin{equation} \label{signb} 
\sum_{k=1}^n k |x_k| +
(n+1)(|{v}| + 1) + 
 \sum_{k=n+m+1}^d 
  (k-m+1) |x_k| 
+ 
\sum_{k=n+1}^m (k-n) |x_k|
 \end{equation} 
in the definition of the structure maps, and a contribution
\begin{equation} \label{signc}
(d-m+1)m + m \left(  |{v}| + \sum_{i \ge n } |x_i| \right) 
 \end{equation} 
 from permuting the determinant lines
 $\DD_{x_j,2}^-, j =n+1,\dots,n+m, \DD_{{v},2}^+$ with
 $\det(T \M_{d-m+1})$ and permuting these determinant lines with the
 $\DD_{x_i,2}^-, i \leq n, \DD_{{v},2}^-$.  The extra factor of $\R$
 created by the gluing at the self-tangency point can be moved to the
 first position at the cost of creating signs in the number of
\begin{equation} \label{signd} \sum_{k=0}^n |x_k| + d   - 1  .\end{equation}
Combining the signs \eqref{signa}, \eqref{signb}, \eqref{signc},
\eqref{signd} one obtains mod $2$
\begin{multline} 
\left(   mn + n + m \right) +
\left( \sum_{k=1}^n k |x_k| + (n+1) (|{v}| + 1) + \sum_{k=n+m+1}^d (k-m+1)
  |x_k| +  \sum_{k=n+1}^{n+m} 
(k-n) |x_k|  \right) \\ 
 + 
(d-m+1)m + m \left( |{v}| +  \sum_{i \leq n} |x_i| \right)
+ \sum_{k=1}^n |x_k| + d + 1  \\
\equiv 
( mn + m + n) + 
\sum_{k=1}^d k |x_k| + (n+1) (|{v}| + 1) + 
 \sum_{k=n+m+2}^d (m-1) |x_k| + \sum_{k=n+1}^{n+m}  n |x_k| 
\\   + 
(d-m+1)m + m \left( d + \sum_{i \ge n + m + 1} |x_i| \right)
+ \sum_{k=0}^n |x_k| + d + 1 \end{multline} 
\begin{multline}
  \equiv mn + m + \sum_{k=1}^d k |x_k| + |{v}| + \sum_{k=n+m+2}^d |x_k|
  + nm +
  (d-m+1)m + md + \sum_{k=0}^n |x_k|   \\
  \equiv m + \sum_{k=1}^d k |x_k| + \sum_{k=n+1}^{n+m} |x_k| + m +
  \sum_{k=n+m+2}^d |x_k| + \sum_{k=0}^n |x_k| .
\end{multline}
After incorporating the sign \eqref{heartsuit} the number of signs is
congruent mod $2$ to
\begin{equation} \label{assocsign} 
\sum_{k=0}^d  |x_k| + d   \equiv 0 .
 \end{equation}
This completes the proof of Theorem \ref{smallstrip}.

\begin{corollary} \label{gcor} Let $\phi_s: L \to X$ be an admissible
  Maslov flow with a self-tangency at $s = 0$ as above.  For any
  $E > 0$ and $\ul{x} = (x_0,\ldots, x_d)$ there exists an $\eps_0$ so
  that for $\eps < \eps_0$ there is an orientation-preserving
  bijection
\begin{multline} \label{bij}
  \Set{ u \in \M_d(\phi_{-\eps},\ul{x})_\rho \ | \ E(u) < E } \\ \to
  \Set{ \ul{u}
  \in \M(\phi_{\eps}, \ul{x}_0)_\rho \times_{\cI(\phi_\eps)^r}
  \prod_{i=1}^r \M(\phi_\eps, v_{\eps_i}, \ul{x}_i)_\rho \ | \ E(\ul{u}) <
  E } \end{multline} 
where $\ul{x}_0 = (x_0,v_{-\kappa_i})_{i=1}^r$ with
$\kappa_i \in \{ \pm 1 \}$ and
$\ul{x} = \ul{x}_1 \cup \ldots \cup \ul{x}_r$, where $r$ runs over
non-negative integers; see Figure \ref{lots}.
\end{corollary}

\begin{figure} 
\begin{center}
\includegraphics[width=5in]{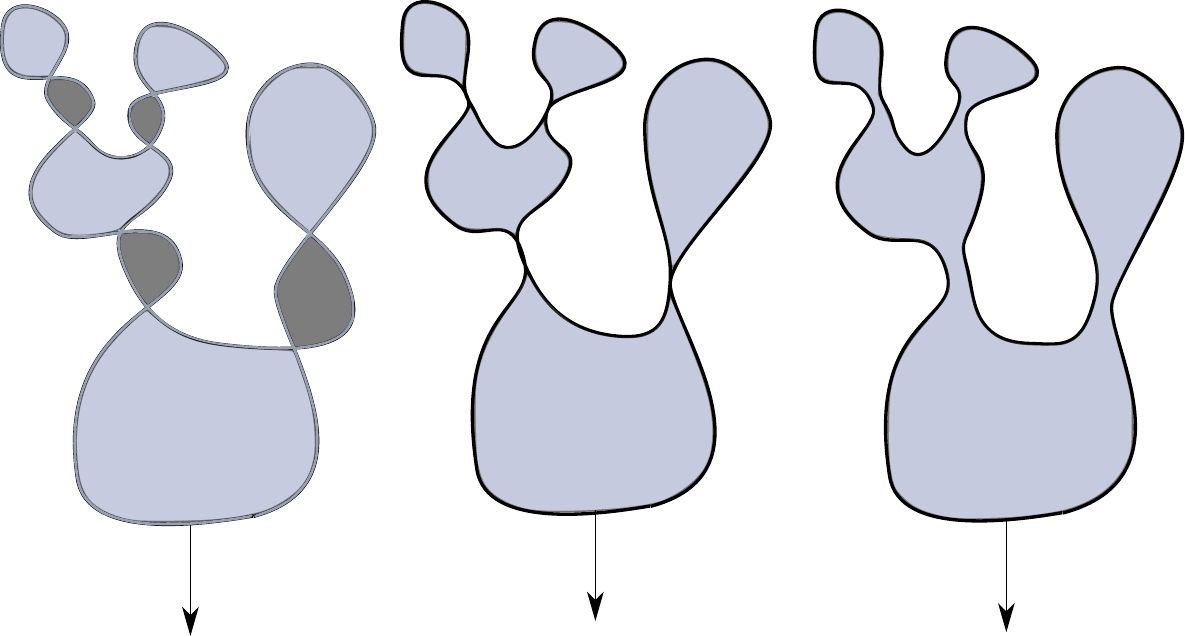}
\end{center}
\caption{Disks on an immersion developing self-tangent boundary}
\label{lots}
\end{figure} 

\begin{proof} We combine \label{combinate} the different parts of Theorem
  \ref{smallstrip} : Given a family $u_\eps \in \M^{< E}(\phi_\eps)$,
  we obtain in the limit $\eps \to 0$ an element
  $u \in \M^{< E}(\phi_0)$. Removing the preimages $u^{-1}(x_j)$ of
  the tangency points in $\phi_0(L)$ we obtain a disconnected domain
  $C - \{ u^{-1}(x_j) \}$.  Define $u_0$ to be the closure of the
  component of $C - \{ u^{-1}(x_j) \}$ containing the outgoing edge
  $e_0$, and let $u_1,\ldots, u_r$ be the components of $u$ attached
  to the component $u_0$ at the self-tangency points.  By Theorem
  \ref{smallstrip} again, each of these shrinks to a component mapping
  to $v_{\pm \kappa_i}$ for $\eps < 0$, and these are connected by
  small strips connecting $v_{\pm \kappa_i}$ with $v_{\mp
    \kappa_i}$. The bijection \eqref{bij} follows.
\end{proof} 

\section{Invariance for birth-death singularities}
\label{bdeath}

In this section we use the correspondence between holomorphic curves
with Lagrangian boundary condition evolving under Maslov flow in
Theorem \ref{smallstrip} to construct the claimed map in Theorem
\ref{tanmain} between Maurer-Cartan moduli spaces as the immersion
develops (or loses) a pair of self-intersection points. 

We first introduce some notation for counting curves in the immersion
with self-tangency.  Let $\ol{\M}(\phi_0)$ denote the moduli space of
treed holomorphic disks with boundary condition on $\phi_0$ and let
$\ol{\M}(\phi_0)_{0} $ denote the part of the moduli space of expected
dimension zero with the property that any combinatorially-finite edge
$T_e \subset T, e \in \Edge_{< \infty}(\Gamma)$ mapping to the
self-tangency $v$ has length zero: \label{mapping}
\[   (u(T_e) = v)   \implies ( \ell(e)  = 0 ) .\]
(Without this condition, treed disks splitting at the tangency cannot
be rigid, since one could add an edge of arbitrary length.)  Theorem
\ref{smallstrip} implies that for any energy bound $E$, the subset of
rigid holomorphic treed disks $\ol{\M}^E(\phi_0)_\rho$ of energy at most
$E$ is compact, since the correspondence in Theorem \ref{smallstrip}
is energy preserving up to a small constant determined by the change
in action at the self-intersection points.  Each element $u: C \to X$
of $\ol{\M}(\phi_0)_\rho$ has underlying combinatorial type a treed nodal
disk $C$ where the nodes of $C$ map to the self-tangency point
${v} \in X$.  For any such map let $E(u)$ denote the energy,
$\sigma(u)$ the number of interior markings, and $y(u)$ the monodromy
of the local system.

\begin{definition} {\rm (Composition maps for self-tangent
    immersions)} Define maps similar to the \ainfty structure
maps
\begin{multline} \label{fake}
  \mu_d(\phi_0): CF(\phi_0)^d  \to  CF(\phi_0), \\
  \mu_d(\phi_0;x_1,\ldots,x_d) := \sum_{u \in
    \ol{\M}(\phi_0;x_0,\ldots,x_d)_\rho } (-1)^{\heartsuit}
  (\sigma(u)!)^{-1} y(u) q^{E(u)} \eps(u) x_0 \end{multline}
\end{definition} 
where $\heartsuit$ is as in~\eqref{heartsuit}.

\begin{remark} The maps \eqref{fake} fail to satisfy the \ainfty
  axiom.  Suppose that the orientations on the determinant lines at
  the self-tangency points $\det(D_{x_\varkappa}^\pm)$ are chosen so that
  the natural maps $\det(D^\pm_{x_0}) \to \det(D_{x_\varkappa}^\pm)$
  induced by the deformation of boundary condition is orientation
  preserving.  If we preclude the tangencies as inputs or outputs then
  the maps $\mu_d(\phi_0)$ capture the limit of the Fukaya composition
  maps for $\phi_t$ as $t \to 0^-$: Let
  \begin{equation} \label{piv} \Pi: CF(\phi_0) \to CF(\phi_0), \quad x
    \mapsto \begin{cases} x & x
      \in \{ v, \ol{v} \} \\
    0 & \text{otherwise} \end{cases} \end{equation}
denote projection onto $\on{span} \{ v, \ol{v} \}$ so that $(1 - \Pi)$
is the projection of $CF(\phi_0)$ onto the image of $CF(\phi_t)$.
Under the natural inclusion  $CF(\phi_t) \to CF(\phi_0)$ we have 
\[ \lim_{t \to 0^-}  \mu_d(\phi_t)  = (1- \Pi) \mu_d(\phi_0)
(1- \Pi)^d .\]
In particular, \label{inpart} the maps $(1 - \Pi) \mu_d(\phi_0) (1-\Pi)^d$ satisfy the
\ainfty associativity axiom, by Theorem \ref{smallstrip}
\eqref{cshort}.

However, because of the failure of gluing at a self-tangency without
moving the Lagrangian, one cannot expect the maps $\mu_d(\phi_0)$ to
satisfy the \ainfty axiom if some of the semi-infinite edges map to \label{simplify}
the self-tangent point ${v}$.  For $b \in CF(\phi_0)$ define
$b$-deformed maps
\begin{multline}  \mu^b_d(\phi_0;a_1,\ldots,a_d) \\
= \sum_{i_1,\ldots,i_{d+1}} \mu_{d + i_1 +
  \ldots + i_{d+1}}(\phi_0,\underbrace{b,\ldots, b}_{i_1}, a_1,
\underbrace{b,\ldots, b}_{i_2}, a_2,b, \ldots, b, a_d,
\underbrace{b,\ldots, b}_{i_{d+1}}) .\end{multline}
\end{remark} 

\begin{definition} \label{tancount} {\rm (Curve counts at the tangencies)}  
  Define two new maps composing resp. precomposing with projection
  onto the span of the self-tangency points: Define
\[ \nu_{d,\pm}^b : CF(\phi_0)^{\otimes d} \to CF(\phi_0), 
\quad \nu_{d,+}^b = \Pi \circ \mu_d^b, \quad \nu_{d,-}^b = \mu_d^b
\circ \Pi^d.
 \]
 The maps $\nu_{d,\pm}^b$ count holomorphic tree disks with boundary
 insertion such that the only allowed outputs resp. inputs are the
 self-tangency points.  Define
\[ \iota:  CF(\phi_0) \to CF(\phi_t), \quad v_{0,\pm} \mapsto
\eps_\pm \ol{v}_{t,\mp} \]
where $\eps_\pm \in \{ \pm 1 \}$ are the signs of the small strips in
Theorem \ref{smallstrip}.  Let $A_t$ denote the area of the small
strip in Theorem \ref{smallstrip}.  Since
\[ \lim_{t \to 0} A_t = 0, \quad 
\lim_{t \to 0} \val_q(\nu_{0,\pm}^b(1)) > 0\]
we have   positive $q$-valuation 
\[\val_q( q^{-A_t} \iota(\nu_{0,+}^{b}(1)) ) > 0 .\] 
\end{definition} 

We assume that the orientations for the determinant lines at the
self-intersection points are chosen so that the contribution of the
small strip connecting the self-intersection points is negative.

\begin{definition} \label{corrmap} The {\em correction map between Floer cochain
    spaces}  is the map
\begin{equation} \label{psi}  \psi:  CF(\phi_{-\eps}) \to CF(\phi_\eps), \quad b_{-\eps} \mapsto b_\eps:= b_{-\eps} +
q^{-A(\eps)} \iota(\nu_{0,+}^{b_{-\eps}}(1)) .\end{equation} 
\end{definition} 

The informal version of this definition was described in the
introduction in \eqref{informal}.  The derivative of the map
\eqref{psi} is the identity plus corrections from holomorphic tree
disks with output on the tangency and inputs the Maurer-Cartan
element, except for a single input with the given tangent vector
$x \in \cI(\phi_{-\eps})$:
\begin{equation} \label{derivpsi} 
D_{b_{-\eps}} \psi (x) = x + \iota \circ
\nu^{b_{-\eps}}_{0,+}(x) .\end{equation} 
We show below that \eqref{derivpsi} gives an isomorphism of Floer
cohomologies.  

\begin{example}  
  Two weakly bounding cochains $b_{-\epsilon}$ (left) and $b_\epsilon$
  (right) related by $\psi$ are shown in Figure \ref{psifig}, with
  areas shown in the limit $\eps \to 0$ (that is, without any Maslov
  flow).  In this case the area of the small strip $A(\eps)$ between
  the new self-intersection points is equal to the area labelled
  $A_5$.  The original weakly bounding cochain $b_{-\eps}$ is
  supported at the two self-intersection points $\ol{v}_{\eps,\pm}$ in
  the boundary of the region with area $A_1$, and so the correction in
  this case is
  \[ \iota(\nu_{0,+}^{b_{-\eps}}(1)) =  q^{A_6 + ( A_ 3 - A_1 ) +
    (A_2 - A_1) } \ol{v}_{\eps,+} + q^{A_4} \ol{v}_{\eps,-} . \]
This ends the example.  
\end{example}

\begin{figure}
\begin{center}
\includegraphics[height=3in]{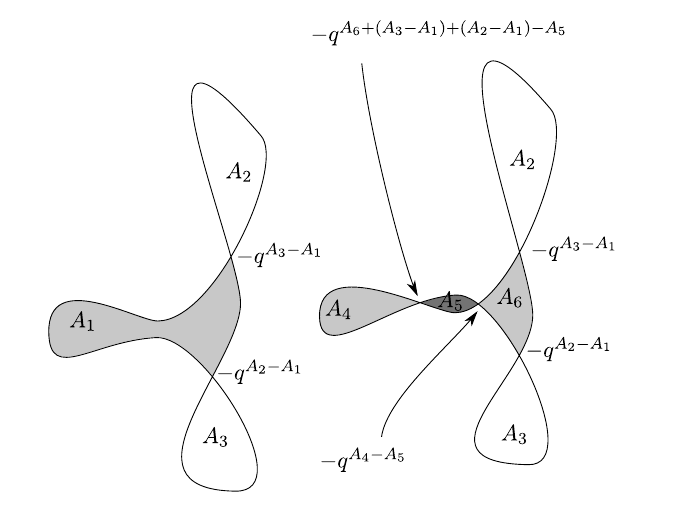}
\end{center}
\caption{Two weakly bounding cochains related by $\psi$}
\label{psifig}
\end{figure}

In preparation for the proof of Theorem \ref{tanmain}, we suppose
without loss of generality that $\phi_t$ undergoes a birth of two new
self-intersection points at $t =0$; the case of a death is similar.

\begin{lemma} \label{claimlem}  The equality holds
\begin{equation} \label{claim} 
\lim_{\eps \to 0} \mu_0^{b_{-\eps}}(\phi_{-\eps};1) =
\lim_{\eps \to 0} \mu_0^{b_{\eps}}(\phi_\eps;1) \end{equation}
using the inclusion $CF(\phi_{-\eps} ) \to CF(\phi_\eps)$. 
\end{lemma} 

\begin{proof} \label{moredetail} By definition
\begin{eqnarray*} 
 \mu_0^{b_\eps}(\phi_\eps;1) &=& \sum_{d \ge 0} \mu_d(\phi_\eps;
                                 b_\eps,\ldots, b_\eps) \\
 &=& \sum_{d \ge 0} \mu_d(\phi_\eps;
                                 (b_{\eps} - b_{-\eps}) + b_{-\eps}   ,\ldots,   (b_{\eps} - b_{-\eps}) + b_{-\eps}) \\
&=& \sum_{r \ge 0}
\mu_r^{b_{-\eps}}(\phi_\eps; b_\eps - b_{-\eps}, \ldots, b_\eps - b_{-\eps})
                        .\end{eqnarray*} 
                      By Definition \ref{corrmap}, this is equal to
\[ 
\mu_0^{b_\eps}(\phi_\eps;1)= \sum_{r \ge 0}
\mu_r^{b_{-\eps}}(\phi_\eps;\iota
\nu_{0,+}^{b_{-\eps}}(1),\ldots,\iota \nu_{0,+}^{b_{-\eps}}(1)) . \]

\vskip .1in \noindent {\em Step 1: The coefficient of any generators
  not equal to $v_{\pm \eps}, \ol{v}_{\pm, \eps}$ in \eqref{claim} are
  equal. }  Corollary \ref{gcor} implies that, accounting for both
curves that degenerate to curves passing through the self-tangency as
well as the change in bounding cochain, the coefficient of any
$x \in \cI(\phi_{\eps}) - \{ v_{\pm,\eps}, \ol{v}_{\pm,\eps} \}$ are
equal.  Indeed, any configuration contributing to
$\mu_0^{b_{-\eps}}(1)$ transforms into a collection of holomorphic
treed disks with output at $v_{\pm,\eps}$, a collection of small
strips connecting to $\ol{v}_{\pm,\eps}$, and holomorphic treed disks
with the original output.  The difference in areas and holonomies are
accounted for by the definition of \eqref{psi}.  An
example \label{anexample} of a configuration counted for $\eps > 0 $
and the matching configuration for $\eps < 0$ is given in Figure
\ref{lots} where the contribution for $\eps > 0$ is of the form
$\mu_2(\nu_{0,+},\nu_{0,+})$; disks with boundary in the given
immersion are lightly shaded while the contributions from $\nu_{0,+}$
are darkly shaded.  The signs agree by Theorem \ref{smallstrip}.

\label{theexample}

  \vskip .1in \noindent {\em Step 2: The coefficient of $v_{\pm,\eps}$ in \eqref{claim}
    vanishes on both sides. }  Recall that the generators
  $v_{\pm,\eps}$ are the outputs of the small strip in Theorem
  \ref{smallstrip}.  If there are configurations contributing to
  $\mu_0^{b_\eps}(\phi_\eps;1)$ with output $v_{\pm,\eps}$ with
  positive coefficient, then the definition of $\psi$ in \eqref{psi}
  implies that these configurations are cancelled by
  $\mu_1(q^{-A(\eps)} \iota(\nu_{0,+}^{b_{-\eps}}(1)) $.  Furthermore,
  there are no other composition maps $\mu_r, r > 1$ containing
  $ \iota(\nu_{0,+}^{b_{-\eps}}(1)) $ as an input, since any polygon
  projects to a polygon in the first coordinate and so the only
  polygon with input $\ol{v}_{\pm,\eps}$ is the small strip.

  \vskip .1in \noindent {\em Step 3: The coefficient of
    $\ol{v}_{\pm,\eps}$ in \eqref{claim} vanishes on both sides. }
  Suppose otherwise, so that without loss of generality the
  coefficient of $\ol{v}_{+,\eps}$ is equal to $c \in \Lambda$.  The
  Bianchi \ainfty relation for $\phi_\eps$ reads
  \[ \mu_1^{b_\eps} \mu_0^{b_\eps}(1) = 0 . \]
Contributions to this identity include once-broken treed disks with no
inputs and output at $v_{-,\eps}$, as in the larger shaded region in
Figure \ref{bianchi} only a part of whose boundary is shown.

\begin{figure}
\begin{center}
\includegraphics[height=2in]{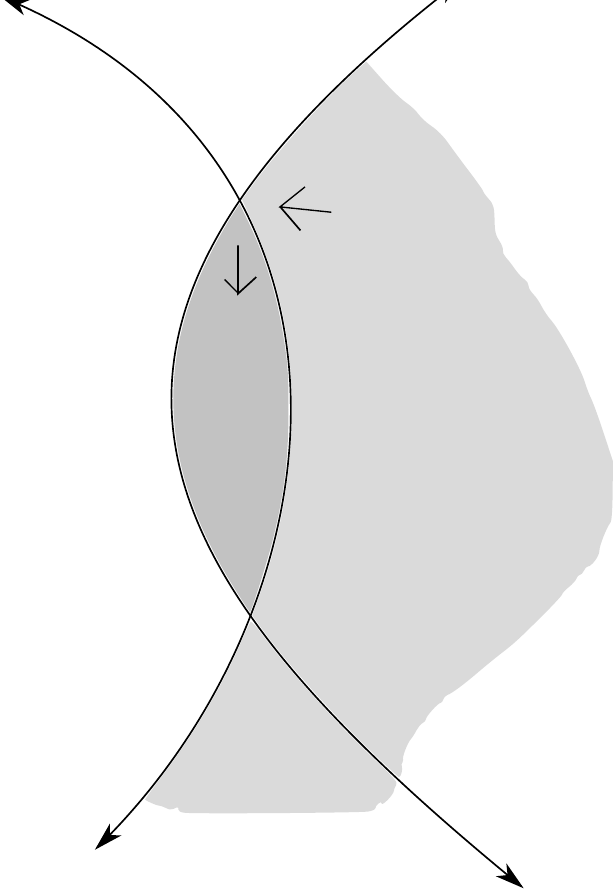}
\caption{Part of a disk with output at $\ol{v}_{+,\eps}$}
\label{bianchi}
\end{center}
\end{figure} 

The contribution from configurations with breaking at
$\ol{v}_{+,\eps}$ is equal to $c q^{A(\eps)}$, since the only
holomorphic treed disk from $\ol{v}_{+,\eps}$ to $v_{-,\eps}$ is the
small strip discussed above, shown darkly shaded in Figure
\ref{bianchi}.  To see this, note that the projection on the first
coordinate is holomorphic and the only holomorphic polygon in the
two-dimensional picture is the bigon in Theorem \ref{smallstrip},
shown in Figure \ref{bianchi}.  Contributions also include
configurations with breaking at some other generator
$x \in \cI(\phi_\eps)$, not equal to $v_{\pm,\eps}, \ol{v}_{\pm,\eps}$
and so corresponding to a generator of $\cI(\phi_{-\eps})$.  However,
by assumption, $b_{-\eps}$ is a solution the projective Maurer-Cartan
equation, and so the coefficient of $x$ in
$\mu_0^{b_{-\eps}}(\phi_{-\eps};1)$ vanishes.  But then by Step 1, the
coefficient of $x$ in vanishes as well.  Hence the coefficient of
$\ol{v}_{+,\eps}$ vanishes.
\end{proof} 

\begin{proof}[Proof of Theorem \ref{tanmain}] 
  By Lemma \ref{claimlem}
\begin{eqnarray*}  b_{-\eps} \in MC(\phi_{-\eps})   
&\implies &  \mu_0^{b_{-\eps}}(1) \in \on{span}(1_{\phi_{-\eps}}) \\ 
&\implies & \mu_0^{b_\eps}(1) \in \on{span}(1_{\phi_\eps}) \\ 
&\implies & b_\eps \in MC(\phi_\eps)  \end{eqnarray*} 
with the same value of the potential. 

To prove the claim on Floer cohomology first note that the Floer
differential $\mu_1^{b_{\pm \eps}}$ is the derivative of
$\mu_0^{b_{\pm \eps}}$, which implies by Lemma \ref{claimlem} that
$D \psi$ is a chain map, where $\psi$ is as in Definition~\ref{corrmap}.  Indeed we have
\begin{eqnarray*} 
\lim_{\eps \to 0}  D \psi (\mu_1^{b_{-\eps}}) (x) &=& \lim_{\eps \to 0} \ddt \psi( \mu_0^{b_{-\eps} + tx} (1))
  \\
                                   &=& \lim_{\eps \to 0}  \ddt \mu_0^{\psi(b_{-\eps}) + t D \psi_{b_{-\eps}}(x)}(1) \\
                                   &=& \lim_{\eps \to 0}  \mu_1^{\psi(b_{-\eps})}( D
                                       \psi_{b_{-\eps}}(x)) .\end{eqnarray*}
To show that $D\psi$ induces an isomorphism, we apply the spectral sequence for the cohomology induced by
the filtration by powers of $q^{\delta}$, for $\delta$ larger than
$\eps$ but smaller than the area of any other non-constant holomorphic treed disk.  
The differential for the first page of the spectral sequence is
the Morse differential, together with the map $\ol{v}_{\eps,\pm} \to
q^{A(\eps)} v_{\eps,\mp}$ induced by the small strip, and any 
cohomology class for this differential has a representative supported in the image of $CF(\phi_{-\eps})$ in
$CF(\phi_\eps)$.    In particular, $D\psi$ is an isomorphism on the
first page, and it
follows
that $D\psi$ is an isomorphism $HF(\phi_{-\eps}) \to HF(\phi_{\eps})$
on the limits of the spectral sequence (see \cite[5.2.12]{weibel}.)

Conversely we claim that the limit of any family of Maurer-Cartan
solutions is obtained by the correction formula in \eqref{psi}.  Given
$b_{\eps} \in MC(\phi_{\eps})$ write
\[ b_{\eps} = b_{-\eps} + \delta b, \quad \delta b = \sum_{x \in \{ \zz,
  \bzz, \bz,\mz \}} \lan \delta b, x\ran x .\]
Consider the equations
\[ \lim_{\eps \to 0} \mu_0^{b_\eps}(1) = \lim_{\eps \to 0} 
\sum \mu_d^{b_{-\eps}}(\delta b,\ldots, \delta b) \in \Lambda
1_{\phi_{1}} .\]
Re-write this equation as
\begin{multline}  
- \mu_1( \lan \delta b , \bz \ran \bz 
+ \lan \delta b , \bzz \ran \bzz \ran ) \\  
= \sum_{d \ge 0} \mu_d^{b_{-\eps}}(\delta b,\ldots, \delta b)  - 
\mu_1(
\lan \delta b ,\bz \ran \bz  + 
\lan \delta b , \bzz \ran \bzz \ran )  \ \text{mod} \
\on{span}(1_\phi) \end{multline} 
The terms in $\delta b$ mod $q^{n\eps}$ determine the terms in
$\delta b$ mod $q^{n\eps + \hbar}$ for some energy quantization
constant $\hbar$ and any integer $n$.  Therefore this equation has a
unique solution $\delta b$, necessarily equal to the one given above
in \eqref{psi}.

%Therefore it suffices to show that the weakly bounding cochain just
%before death is gauge equivalent to a weakly bounding cochain with
%vanishing coefficients of the disallowed generators.  We claim that
%there exists a solution $b_\eps'$ to the gauge equation
%%
%\begin{equation} \label{tosolve} \mu_1^{b_\eps,b_\eps'}(c) = b_\eps' -
 % b_\eps \end{equation}
%%
%such that the coefficients of $\zz, \mz$ in $b_\eps'$ vanish.  The
%initial step is as follows. Suppose that the highest order term in the
%coefficient of $y_\pm$ is $q^{A_0}y_\pm$.  Let
%$c = q^{A_0 - A} \ol{v}_\mp$.  Setting $b_\eps' = b_\eps + c$ produces
%a solution modulo the subspace generated by $q^{A_0 + \eps}$, where
%$\eps > 0$ is small.

%The gauge equation \eqref{tosolve} equation can be now solved
%inductively: Assuming that the equation \eqref{tosolve} has been
%solved mod $q^{A_0 + n \eps} CF(\phi)$.  The equation
%%
%%\begin{eqnarray} \label{first}
  %\mu_1^{b_\eps,b_\eps' + \delta b}(c + \delta c) 
 % &=& b_\eps' + \delta b - b_\eps  \ \text{mod} \ q^{A_0 + (n+1) \eps}
 % \\ 
%  0 &=&  \label{second} \lan \mz, b'_{1} + \delta b 
%        \ran  = \lan \zz, b'_{1} + \delta
%        b \ran   \ \text{mod} \ q^{A_0 + (n+1) \eps} 
%%  \end{eqnarray}
%
 % admits a solution
%  $\delta b,\delta c \in q^{A_0 + n\eps} CF(\phi)/q^{A_0 +
%    (n+1)\eps} CF(\phi)$.
%  Indeed, let $(\delta b)'$ be such that the first equation
 % \eqref{first} is solved
%  with $\delta c = 0$, and then let $\delta c$ be such that
%  $\mu_1(\delta c)$ is the projection of $(\delta b)'$ onto the span
 % of $\zz,\mz$.
 \end{proof}

\begin{remark} In Etnyre-Ekholm-Sullivan \cite{ek:or} the authors give
  a second proof of invariance of Legendrian contact homology under a
  self-tangency move by using a ``bifurcation analysis''.  It would be
  interesting to know whether a stabilization or cobordism argument
  would remove the need for explicit gluing results used above.
\end{remark}

\begin{proof}[Proof of Theorem \ref{main}]
  Theorem \ref{main} follows from Theorems \ref{transmain} and
  \ref{tanmain} by breaking up the domain of the isotopy $[0,T]$ into
  sub-intervals $[0,t_1], [t_1,t_2], \ldots, [t_{k-1}, T]$.  By
  induction, we may assume that the theorem holds for any Maslov flow
  $\phi_t$ on the interval $t \in [t_1,T]$.  Suppose that
  $[b_0] \in \ol{MC}(\phi_0)$ with
  \[ \min_{i=0}^k A_i - (\dim(L) - 1)( T - t_i) > 0 .\]
  Since $A_0 = \val_q(b_0)$, Theorem \ref{transmain} implies that
  there exists a family $b_t = E_0^t(b_0), t \in [0,t_1]$ and an
  identification $HF(\phi_0,b_0) \cong HF(\phi_t,b_t)$ for all $t$.
  Theorem \ref{tanmain} implies that the family $b_t$ continues as a
  family $b_{t,+} \in MC(\phi_t)$ for $t \in [t_1,t_1+\eps]$ for
  $\eps$ small with a $q$-valuation
  \begin{eqnarray*} \val_q(b_{t_1,+}) &=& \val_q(\psi(b_{t_1,-})) \\
&=&   \min(\val_q(b_{t_1,-}), A_1) \ge \min(\val_q(b_0) - t_1 (\dim(L) -
  1), A_1) .\end{eqnarray*} 
Now $b_{t_1,+}$ satisfies the assumption of Theorem \ref{main} for the
interval $[t_1,T]$ since, letting $A_1' = \val_q(b_{t_1,+})$, we have
\begin{eqnarray*} \min  && \left(A_1' - (\dim(L) - 1) (T - t_1) , \min_{i=2}^{k}
  A_i -  (\dim(L) - 1)( T - t_i) \right)\\  &\ge& \min \left( \min(\val_q(b_0) - T
  (\dim(L) - 1), A_1 - (\dim(L) - 1) (T - t_1)),  \right. \\ & \ &
                                                                   \left. \min_{i=2}^{k}
                                                                   A_i   -  (\dim(L) - 1)( T - t_i)
  \right) 
  \\ &=& \min_{i=0}^k A_i - (\dim(L) - 1)( T - t_i) > 0
  .\end{eqnarray*}
By the inductive hypothesis, there exists $b_T \in MC(\phi_T)$ with
\[ W(\phi_T,b_T) = q^{2(T - t_1)}W(\phi_{t_1},b_{t_1}) \cong q^{2T}
W(\phi_0, b_0)  \] 
\[ HF(\phi_T,b_T) \cong HF(\phi_{t_1}, b_{t_1}) \cong
HF(\phi_0,b_0).\]
The statement of the Theorem for the forward direction follows.  The
result for the reverse Maslov flow (mean curvature flow) is similar,
using inverse Euler flow.
 \end{proof}

 We end the paper with further open questions:

 \begin{remark} \label{comms}
\begin{enumerate} 
\item The first question concerns existence of non-trivial objects in
  the Fukaya category.  We conjecture that any rational compact
  K\"ahler manifold $(X,\omega)$ contains a Floer non-trivial
  Lagrangian $\phi: L \to X, \ HF(\phi) \neq \{ 0 \}$, that is, a
  non-trivial object $[\phi]$ in the immersed Fukaya category $\F(X)$,
  and the definition of Lagrangian is taken to be sufficiently
  general.  This paper shows that this property is preserved under (at
  least small) K\"ahler-Ricci flows $(X,\omega_t)$, while the paper
  \cite{flips} shows that sufficiently nice surgeries in the flow
  generate non-trivial objects.
\item The next question concerns the existence of mean curvature
  flows.  Joyce's conjecture \cite{joyce:conjectures} predicts that
  after allowing certain singular Lagrangians and surgeries, mean
  curvature flow $\phi_t: L_t \to X, t \in [0,\infty)$ exists and
  preserves Floer cohomology $HF(\phi_t,b_t)$ if $b_0 \in MC(\phi_0)$.
  Note that this approach requires surgery at times $t_i \in (0,T)$
  that are {\em not} singular times of the mean curvature flow.  This
  conjecture is open even for immersed curves $\phi: L \to X$ on
  surfaces $X, \dim(X) = 2$.  Some analytic results relevant to the
  surgeries necessary when the $q$-valuation of the weakly bounding
  cochain become zero are contained in the manuscript \cite[Chapter
  10]{fooo}.
\item To what extent is non-displaceability equivalent to Floer
  non-triviality?  In the case of immersions in the two-sphere
  $\phi: L \cong S^1 \to X \cong S^2$, \label{X2}non-vanishing of Floer
  cohomology $HF(\phi,b)$ seems to be stronger than
  non-displaceability
  $\exists \psi: X \to X, \psi(\phi(L)) \cap \phi(L)$ empty.  Is there
  a version of Floer theory (such as bulk deformations) that detects
  non-displaceability?
\item The structure of the space of Maurer-Cartan solutions $MC(\phi)$
  seems poorly understood, even in the case of immersions
  $\phi: L \cong S^1 \to X \cong S^2$ in two-sphere.  For example,
  what is the dimension of $MC(\phi)$?
\item Immersed Lagrangian Floer theory is somewhat limited in
  applicability by the requirement that the transverse
  self-intersection condition: $HF(\phi,b)$ is defined only for
  immersions $\phi: L \to X$ with self-transverse, or at least clean, \label{cleanint}
  self-intersection $L \times_\phi L - \Delta_L$.  Of course, one can
  perturb $\phi$ to achieve a self-transverse immersion
  $\tilde{\phi}$. How does the Maurer-Cartan moduli space
  $MC(\tilde{\phi})$ depend on the choice of perturbation
  $\tilde{\phi}$?  In particular, is there a ``good locus''
  $MC(\tilde{\phi})$ of weakly bounding cochains $\tilde{b}$ so that
  the Floer cohomology
  $HF(\tilde{\phi}, \tilde{b}), \tilde{b} \in MC(\tilde{\phi})$ is
  independent of the choice of perturbation?
\item As a special case of the previous question, does there exist a
  good Floer theory for Lagrangian immersions that are finite covers
  of their image?  Let $\phi: L \to X$ be a Lagrangian embedding and
  let $\ti{\phi}: \ti{L} \to X$ be the immersion obtained by composing
  $\phi$ with a finite cover $\pi: \ti{L} \to L$.  Is the immersed
  Lagrangian Floer theory of a self-transverse perturbation of the
  composition $\psi \circ \ti{\phi}: \ti{L} \to X$ independent of the
  choice of perturbation $\psi: X \to X$?
\item What is the sub-category of immersed Fukaya category $\F(X)$ \label{removeto}
  generated by embedded Lagrangians $\phi: L \to X$?  The mapping cone
  construction in Fukaya-Oh-Ohta-Ono \cite[Chapter 10]{fooo} implies
  that, at least in some cases, self-intersection points
  $x \in L \times_\phi L $ may be replaced with mapping cones
  $\Cone(x), x \in \Hom(\phi,\phi)$; see also Mak-Wu \cite{mak:dehn}
  and Fang \cite{fang}.
\end{enumerate}
\end{remark} 

\appendix

\section{Erratum on triple intersections}

This erratum is join with Hadi Azizi, who kindly pointed out that the
proof of Lemma 7.9 (b) omitted a discussion of the case that only one
of the interior angles is even.  In this case, the map of
Maurer-Cartan spaces is non-trivial.  The claim in question is that
immersed Floer cohomology is invariant under Maslov flow when the
Maslov flow passes through a triple intersection; i.e.~a time when
three branches of the Lagrangian intersect.  Generically, this happens
only at a finite number of times for two-dimensional symplectic
manifolds, and not at all in higher dimensions.  To introduce
notation, let $X$ be a compact, oriented surface with symplectic form
$\omega \in \Omega^2(X)$, and let $L_t \subset X$ be a Maslov flow of
immersed Lagrangians (e.g.~a curve shortening flow) as above.  We
assume that at time $t =0 $ the immersion $L_t$ has a triple
intersection, as shown in Figure \ref{fig:triple}.  Choose small times
$t_- < 0 < t_+$ and define
\[ L_- = L_{t_-}, \quad L_+ = L_{t_+} \]
as shown.    Let $U$ be a neighborhood of the triple intersection so that there are three local branches $L_{\pm,1}, L_{\pm,2}, L_{\pm,3} \subset L_\pm$ of the Lagrangians as shown 
\[  U  \cap L_\pm = 
L_{\pm,1} \cup L_{\pm,2} \cup L_{\pm,3} .\]
Assume that $L_-$ is equipped with a weakly bounding cochain $b_-$. 

\begin{theorem} \label{thm:inv} (Lemma 7.9 (b) above)
Under the above assumptions, there exists a weakly bounding cochain $b_+$ for $L_+$
so that
\begin{enumerate} 
\item \label{it:sameW} $W(b_+) = q^{2(t_+ - t_-)} W(b_-)$ and 
%\red{[\textbf{Comment:} I'm a little confused about the $t$ here, I think that the $q^{2t}$ should either be $q^{2t_-}$ or $q^{2t_+}$]}
% fixed to the change in time t_+ - t_-
\item \label{it:sameHF} $HF(L_-,b_-) \cong HF(L_+, b_+).$
\end{enumerate}
\end{theorem}

We introduce the following notation.   Let $R_\pm$ denote the small triangle at the center, that is, the bounded region in the complement of $L_{\pm,1} \cup L_{\pm,2} \cup L_{\pm,3}$.   We assume that the areas of the small triangles $R_-,R_+$ are  equal and given by some small real number $\eps$, as in Figure \ref{fig:triple0}.

\begin{figure}[ht]
    \centering
   \scalebox{.9}{ 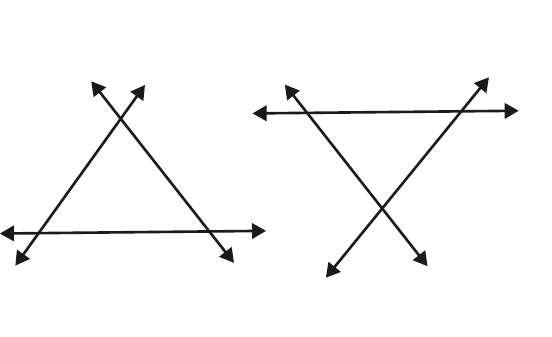}
    \caption{Moving past a triple intersection}
    \label{fig:triple0}
\end{figure}

By assumption, the self-intersection points of the Lagrangians are assumed to have $\Z_{2N}$ gradings for some $N$, and the $\Z_2$ grading by reducing the $\Z_{2n}$ grading is determined by the orientations as follows as in Seidel \cite{se:gr}:  Suppose the corresponding corner $x \in L_{\pm,i} \cap L_{\pm,j}$ has orientations 
\[ o_{\pm,i} \in
\pi_0(T_x L_{\pm,i} - \{ 0 \} ), \quad o_{\pm,j} \in
\pi_0(T_x L_{\pm,j} - \{ 0 \} ) \] 
on both Lagrangians. The following is immediate from the definitions:

\begin{lemma} 
\begin{enumerate} 
\item If both $o_{\pm,i}, o_{\pm,j}$ point outwards from $x$ or both point inwards then the $\Z_2$ grading $|x| \in \Z_2$ is an odd as an ouput corner and even as an input corner, while 
\item if one orientation points outwards and one points inwards then the associated grading $|x| \in \Z_2$ is even as an output corner and odd as an input corner. 
\end{enumerate}
\end{lemma} 

 By an interior angle, we mean an ordered self-intersection point arising as in Figure \ref{fig:triple}
or Figure \ref{fig:triple2}, where the ordering is induced from the counter-clockwise ordering around the boundary of the triangle $R_-$.  

\begin{lemma}  
There are either three even interior angles, or one even interior angle. \end{lemma}   

\begin{proof}  For the symmetric case that the orientations of the three branches 
$L_{-,1}, L_{-,2}, L_{-,3}$ around the inner triangle  $R_-$ are all clockwise, all interior angles correspond to ordered self-intersection points that are even.     Each change in orientation
of one of the branches $L_{-,k}$ switches two of the self-intersection points, say $L_{-,k} \times_X L_{-,i}$ and $ L_{-,k} \times_X L_{i,j}$, corresponding to interior angles from even to odd degree or vice-versa.  So any sequence of switches of orientations of the branches of $L_-$ produces a system of gradings where either one or three of the interior angles is even.  
\end{proof}

For completeness we recall our notation for weakly bounding cochains. Let  $L \to X$ 
be a compact, oriented, immersed Lagrangian equipped with an orientation, spin structure, and grading with transverse self-intersection points.   The generators of $CF(L)$ are given by ordered self-intersection points, together with critical points of 
a Morse function, and two extra generators from the homotopy units
construction, which may be ignored.  
Denote by 
\[ m_d: CF(L)^{\otimes d} \to CF(L)  \] 
 the \ainfty operations  for $L$ obtained by counting holomorphic polygons.   More precisely, if %
 \[ x_1,\ldots, x_d \in L \times_X L  \] 
 are  incoming ordered self-intersection points then 
 \[  m_d(x_d,\ldots, x_1) =  \sum_{u  \text{ has  output  } x_0}
q^{A(u)} \sigma(u) \] 
%
%\red{[comment: I changed the spacing in ``has output'' above. I think it looks better this way]}
%CW: thx
where $\sigma(u) \in \{ \pm 1 \}$ is the orientation sign determined by the orientations on the Lagrangians, and the sum is over rigid holomorphic treed disks with incoming edges with limits 
$x_1,\ldots, x_d$ and outgoing edge with limit $x_0$.   In the case of Morse generators one considers treed holomorphic polygons.   The condition for $b \in CF(L)$
to be a weakly bounding cochain is that $b_\pm$ solve the projective Maurer-Cartan equation 
\[ m_0^{b}(1) := m_0(1) + m_1(b) + m_2(b \otimes b) + \ldots  = W(b) 1_L  \] 
for some element $W(b)$ in the Novikov ring called the {\em disk potential} of $b$; here 
 $1_{L} \in CF(L)$ is the strict unit.   In this case the deformed differential  defined by
\[ m_1^{b}(c) := \sum_{d_0,d_1} m_{1 + d_0 + d_1}( b^{\otimes d_1} \otimes c \otimes  b^{\otimes d_0} ) \]
 squares to zero.  Denote by $HF(L,b)$ its cohomology.   For a pair of Lagrangians $L_0,L_1$ 
 intersecting cleanly and equipped with weakly bounding cochains $b_0,b_1$ there is a similar version $CF(L_0,L_1)$ with differential denoted $m_1^{b_0,b_1}$.   Its cohomology is denoted
 $HF(L_0,L_1)$.  In the special case  $L = L_0 = L_1$ and $b_0 = b_1$ we have $HF(L) = 
 HF(L_0,L_1)$.  Furthermore, the Floer cohomology of a pair is invariant under Hamiltonian 
 diffeomorphisms $\phi$ in the sense that $HF(L_0,L_1) \cong HF(L_0,\phi(L_1))$, assuming the intersections are clean. 

 The main result Theorem \ref{main} states that the Floer cohomology
 is invariant under Maslov flow.  Let
 \[ L_t, t \in [t_-,t_+] \] 
 be such a flow.  Using notation above for $t \in [t_-,t_+]$ let
\[ E_{t_-}^t :  CF(L_{t_-}) \to CF(L_{t}) \]
denote the Euler flow from time $t_-$ to time $t$ from Definition 7.6 and 
\[ b_{-,t}  = E_{t_-}^{t} b_- \in MC(L_{t}) \] 
the flowed weakly bounding cochain, up to the time of the triple
intersection.  The Euler flow is defined so that, in particular, the
coefficients of the curvature change by an overall power of $q$, so
that, in the absence of a singularity or triple crossing, the element
$b_{-,t}$ is a weakly bounding cochain for the flow $L_t$ of $L_-$. We
suppose for the rest of the note that $L_+$ is obtained by flowing
$L_-$ through a triple crossing, as shown in Figure~\ref{fig:triple2}.

\subsection{One even interior angle} 

We first consider the case that the two lower interior angles in the small triangle $R_-$ in the center are degree $0$ ordered self-intersection points, and the top intersection point is degree $1$.  We label the interior angles of even degree  in the Novikov ring as in Figure \ref{fig:triple}, 
giving coefficients of weakly bounding cochain 
\[ b_- = e\ol{x}_- + f \ol{y}_- + g \ol{z}_- + \ldots  \in CF(L_-) .\]

\begin{notation} As in Figure \ref{fig:triple} denote by
 \begin{itemize} 
 \item $u_1$ a general holomorphic disk with output corner labelled $f$,  with arbitrary corners labelled $b_-$;
 \item $u_2^-$ a general holomorphic disk with output corner with coefficient $f$ and first input corner at $e$;  and remaining corners labelled $b_-$; $u_2^+$ a general holomorphic disk with output corner at $e$ and last input corner at $e$;  Thus $u_2^-, u_2^+$ are geometrically the same disk up to 
 permutation of inputs and outputs;
 \item $u_3$ with output corner $e$, and remaining corners labelled $b_-$;  
\item $u_2'$ a general holomorphic disk with output corner $g$, not containing the interior triangle $R_-$, and remaining corners labelled $b_-$; and
 \item $u_2''$ a general holomorphic disk with output corner $g$, containing the interior triangle $R_-$ with area $\eps$, and remaining corners labelled $b_-$;
\end{itemize} 
We use the notation $\# u $ to indicate the weighted count of 
disks with some specified type, contributing to the \ainfty structure
maps.   Thus by $\# u_1$ we mean the weighted count of disks $u_1$ with arbitrary numbers of corners labelled with the weakly bounding cochain, and output $f$:
\[ \# u_1 =  \sum_{u_1 \ \text{output} f}
(-1)^{\heartsuit} q^{A(u_1)} \sigma(u_1) \]
where $\sigma(u) \in \{ \pm 1 \}$ is the orientation sign determined
by the orientations on the Lagrangians, and $\heartsuit$ is the sign
from \eqref{heartsuit}.
\end{notation}

\begin{figure}[ht]
    \centering
   \scalebox{.6}{ 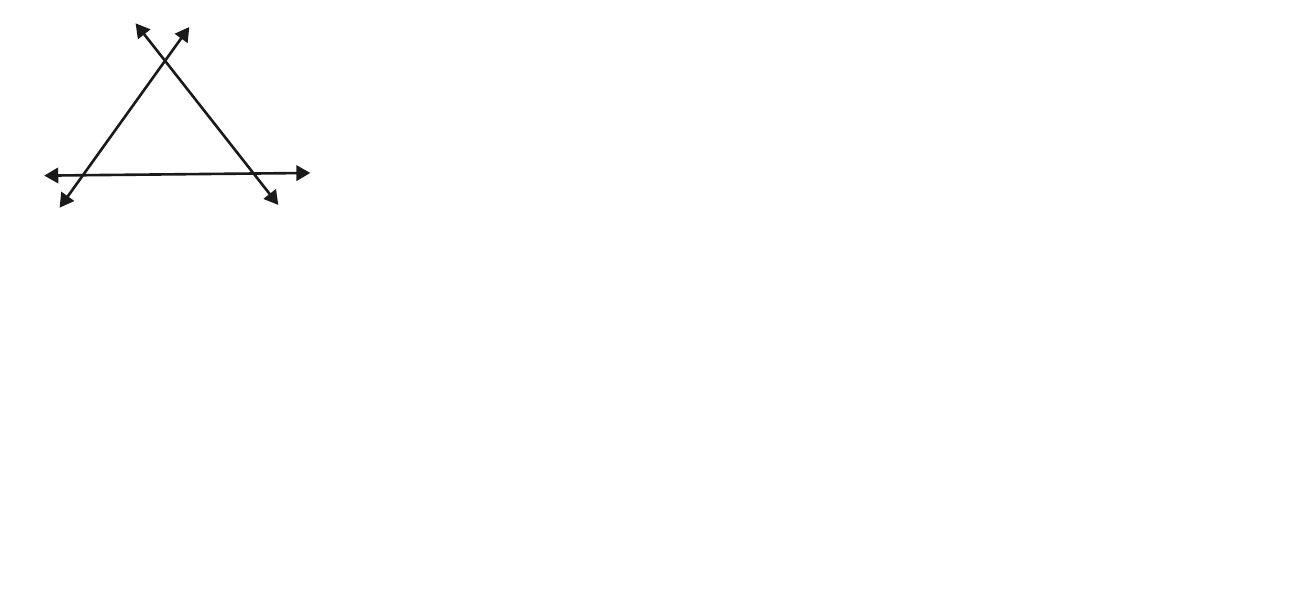}
    \caption{Disks bounding a triple intersection}
    \label{fig:triple}
\end{figure}

\begin{lemma}  With notation as above, the numbers of 
disks containing the bottom region in Figure \ref{fig:triple} are related by 
\begin{enumerate} 
\item $\# u_2^- = - \# u_2^+$ ; 
\item $\# u_2^+ =  \pm \# u_2'' = - \# u_2^-$, with sign depending on the choice of orientations.
\end{enumerate}
\end{lemma}

\begin{proof}  
For the first item, the disks $u_2^-, u_2^+$ are geometrically the same, but are related by naming the output of $u_2^-$ the first input of $u_2^+$.   In the construction of orientations in \cite[(40)]{Wehrheim_Woodward_orientation}, this corresponds to the transposition of determinant lines (one of the outputs with one of the inputs) in the expression 
\[ \det(D_{\gamma_d}^-) \det(TL) \ldots \det(D_{\gamma_1}^-) \det(TL) \det(D^c) \det (TL) \det(D_{\gamma_0})   \] 
where $D^c$ is a Cauchy-Riemann operator for a problem on a surface with boundary but without strip-like ends, from which the orientation on the moduli space of holomorphic disks is induced.   Since these determinant lines are determinants of Freholm operators of odd index 
and the intervening expression 
\[  \det(TL) \ldots \det(D_{\gamma_1}^-) \det(TL) \det(D^c) \det (TL)  \] 
is even, this transposition produces a sign change as in \cite[Remark 2.1.2 (4)]{orient}. 

For the second item, for any such disk $u_2$, one may take the union with the interior triangle $u_2''$.  Vice-versa, given $u_2''$
removing the interior triangle gives a disk $u_2$ with corners $f,e$.   The corresponding compactified boundary value problems are canonically isotopic and, as a result, the determinant lines are canonically isomorphic, and the isomorphism is orientation is orientation preserving for suitable choices of orientations at $e,f,g$.
\end{proof}

Unobstructedness of $L_-$ implies the following equalities:  

\begin{lemma}  Suppose $L_-$ is weakly unobstructed.  Then
\begin{eqnarray*} 
\# u_1 + \# u_2^+ e &=&  0 \\ 
\# u_3 + \# u_2^-   f &=&  0  \\ 
\# u_2 q^{\eps} + \# u_2' &=& 0 .\end{eqnarray*}
\end{lemma}

\begin{proof}    Because holomorphic disks can come into an ordered
self-intersection point we have, for example, 
\[ \# u_1 + \# u_2 e  = \on{coeff}(z_-, m_0^{b_-}(1))  \] 
where $\on{coeff}(y_-, m_0^{b_-}(1))$
means the coefficient of $y_-$ in the curvature $m_0^{b_-}(1)$, which must vanish since $b_-$ is a weakly bounding cochain.    The other equalities  are equivalent to the vanishing of the curvature $\mu_0^{b_-}$ at the other degree zero intersection points in Figure \ref{fig:triple}.
\end{proof}

\begin{definition}  For $t > 0$, let $b_{+,t} \in CF(L_t)$ be the cochain equal to $b_{-,t}$
on the complement of the three self-intersection points shown in Figure \ref{fig:triple},
an given by the values  $e_t,f_t, e_t f_t + q^\eps g_t$ at the degree one self-intersection points near the triple intersection shown in Figure \ref{fig:triple},
where $e_t,f_t,g_t$ are the coefficients of $b_{-,t}$.
\end{definition}

\begin{proof}[Proof of Theorem \ref{thm:inv} \eqref{it:sameW}, case of one even interior angle]  We introduce the following notation.  Denote by $v_1$ a general holomorphic disk bounding $L_+$ with a single corner at $a$, $v_2^{',+}$ a general holomorphic disk with output with coefficient $f$ an last input at $e$, $v_2^{',-}$ a general holomorphic disk with output at $e$ and first input with corner with coefficient $f$,
$v_3$ a general holomorphic disk with $f$ (so $v_1,v_3$ overlap at $v_2'$) and $v_2$ a general holomorphic disk with corner 
labelled $fe + gq^{\eps}$. 

The following relationship holds between the signed counts
of disks before and after the triple intersection:  As in Lemma 7.4, for the family of bounding cochains $b_t$ whose values are given there at intersection points outside of the diagram shown,
\[ \# v_1 = \# u_1 q^{2t}, \quad \# v_2 = \# u_2 q^{\eps} q^{-2t}, \quad \# v_3 = \# u_3 q^{2t} .\]

To show unobstructedness we need to show the three equalities 
\begin{eqnarray*}
 \# v_2' q^{\eps} + \# v_2 &=& 0 \\
 \# v_1 + \# v_2^{',+} f_t &=& 0 \\
 \# v_3 + \# v_2^{',-} e_t &=& 0 \end{eqnarray*}
The signs in the second and third equations are in the opposite order 
of those in $u$ because of the change in order of $f_t,e_t$ in the ordering of punctures around the boundary of the disks.
We have 
\[ \# v_2' q^{\eps} = \# u_2' q^{\eps+2t} = 
-\# u_2 q^{-\eps+2t} q^{\eps} = - \# v_2 \]
as required.  Similarly
\[ \# v_1 = \# u_1 q^{2t} = - \# u_2^+ f_t q^{2t} = - \# u_2' f_t q^{-\eps+2t}  = 
- \# v_2' q^{\eps } f_t q^{-\eps}  =  - \# v_2^{',+} f_t.  \] 
The last inequality is similar.   This shows that
$L_+$ is unobstructed.  

It remains to show that the potentials agree up to a power of $q$. Let $u_-$ be a disk with $k_-$ corners at $g_t$ from inside the triangle, and $k_+$ corners at $g$ from outside the triangle, and no corners at $e_t,f_t$.  From $u_-$ we obtain a disk  $u_+$ with $k_-$ corners 
at the corner marked $f_t e_t  + q^\eps g_t$ from outside the triangle and $k_+$ corners at $f_t e_t + q^\eps g_t$  from inside the triangle. 
There are also $2^{k_-}$  other disks $u_-'$ with corners at either $e_t,f_t$ or at $g_t$, and $2^{k_+}$ other disks $u_+'$ with corners at $e_t$ and $f_t$, or $g_t$.   The curves $u_-,u_+$ are in bijection, but their weighted count differs by a factor 
of $q^{2t}(f_t e_t + q^\eps g_t)^{k_+ - k_-}$.  Hence
\[ \# u_-' = q^{2t}(f_t e_t + q^\eps g_t)^{k_-} \# u_-  =  
 \# u_+ q^{2t}(f_t e_t + q^\eps g_t )^{k_+}  = q^{2t}\# u_+'  \] 
as desired. 
 \end{proof}

\subsection{Three even interior angles}

We next consider the case of three even interior angles,
shown in Figure \ref{fig:triple2}.    Denote by $u_1$ resp.~$u_1'$ a general holomorphic disk with corner at $e$ containing the inner triangle, resp.~not containing the inner triangle;
$u_2$ resp.~$u_2'$ a general holomorphic disk 
with corner at $g$ containing the inner triangle, resp.~not containing the inner triangle; and $u_3$ resp.~$u_3'$ a general holomorphic disk with corner at $f$ containing the inner triangle, resp.~not containing the inner triangle.  Weak unobstructedness of $L_-$ implies:

\begin{lemma}  The equalities 
\[ \# u_2 + q^\eps e_t f_t = -\# u_2', 
\ \# u_1 + q^\eps f_t g_t  = - \# u_1', \ 
\# u_3 +q^\eps g_t e_t = - \# u_3' .\]
hold. \end{lemma}

\begin{proof} The equalities 
%\red{???} 
are the vanishing of the curvature at the intersection points with 
coefficients $g_t$, $f_t$, and $e_t$ respectively. 
\end{proof}

\begin{proof}[Proof of Theorem \ref{thm:inv} \eqref{it:sameW}, case of three even interior angles] 
 From the previous Lemma we obtain 
\[ q^{2t} ( \# v_2 + e_t f_t)  = - q^{2t} v_2' ,  \quad  q^{2t} ( \# v_1 + f_t g_t ) =  - q^{2t} \# v_1', \quad  \#  q^{2t}( v_3 + g_t e_t)  =  - q^{2t} \# v_3'  \]
so
\[   \# v_2  = - \#v_2'   - e_t f_t , \quad 
\# v_1 =  - \# v_1' - f_t g_t,  \quad \# v_3 = - \# v_3' - g_t e_t .\]
These are the equations for the weak unobstructedness of $L_+$; the change in sign come from the fact that the small triangle $R_-$ has induced orientation that reverses (with respect to the orientation induced from $L_\pm$)
under the flow over the triple intersection.   The computation of the potential is similar to the case of one even interior angle, treated above. 
\end{proof}

\begin{figure}[ht]
    \centering
   \scalebox{.9}{ 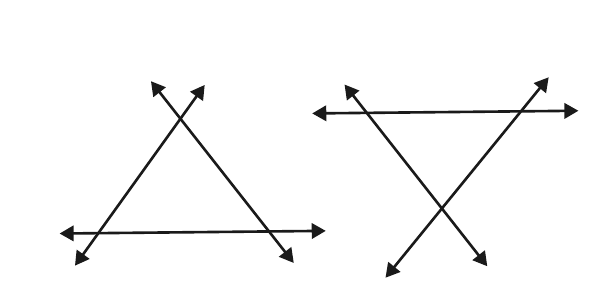}
    \caption{Moving past a triple intersection}
    \label{fig:triple2}
\end{figure}

\subsection{Isomorphism of Floer cohomology}

We finally prove the isomorphism of Floer cohomology.  Following the philosophy suggested by the Yoneda lemma, we do this by comparing the Floer cohomology with a test Lagrangian, which may be taken to be a Hamiltonian isotopy of the original Lagrangian, whose Maslov flow has possibly a triple intersection at a different point in the symplectic manifold.   Let $T_- \to X$ be an immersed Lagrangian with weakly bounding cochain $c_- \in CF(T_-)$ with the same potential $w_-$ as $L_-$.  Let $T_t$ be a Maslov flow of $T_-$ from $t_-$ to $t_+$,
%possibly with a triple crossing at $t  = 0$ at some point $q \in X$ 
with weakly bounding cochain $c_{t}$.  As above, picking any Maslov
flow $T_t$ requires a choice of connection $\alpha_X$ on the
anticanonical divisor $K_X$ and a flat connection $\alpha_L$ on its
restriction to the Lagrangian $T$. In order to define $HF(L_t,T_t)$,
we need the Maslov flow $T_t$ to be \textit{compatible} with $L_t$,
i.e., the chosen connections on $K_X$ agree.  $T_t$ has the same
potential as $L_t$, and so the Floer differential for the pair
$(L_t,T_t)$ squares to zero and the cohomology $HF(L_t,T_t)$ is
well-defined.  Let $(T_+,c_+)$ be the specialization at small,
positive $t_+.$ Let $p \in X$ be the point of triple intersection of
$L$.
 
\begin{proposition} \label{prop:test} Suppose there exist a Darboux neighborhood $U$ of $p \in X$  so that 
\begin{enumerate}
\item 
$U$ is disjoint from $T_t$ for all $t \in [t_-,t_+]$, and the intersection of $L_t$
with $U$ is as depicted in Figure \ref{fig:triple} or \ref{fig:triple2};
%
%\item if a triple intersection $q$ of $T_t$ exists, then $V$ is disjoint from $L_t$ for all $t \in [t_-,t_+]$ and the intersection of $T_t$
%with $U$ is as depicted in Figure \ref{fig:triple} or \ref{fig:triple2}; 
%
\item $T_t$ does not go through a triple intersection for $t \in [t_-,t_+]$.
\item the intersection points $T_t \cap L_t$ are transverse for all $t$.  
\end{enumerate}
Then  $HF((L_-,b_-),(T_-,c_-)) \cong HF((L_+,b_+),(T_+,c_+))$. Note that all three conditions are satisfied for a generic choice of compatible Maslov flow $T_t$ and small time interval $[t_-,t_+]$.
\end{proposition}

\begin{proof}   This is proved in the same way as the equality of potentials:  
Suppose $x, y \in CF((L_-,b_-),(T_-,c_-))$ represent intersection points necessarily in the complement of the neighborhood $U$ of the triple intersection.  The count of contributions to the coefficient of $y$ in $m_1^{b_-}(x)$ is equal, up to a factor of $q$, to the coefficient of 
$y$, by the same argument as in Theorem \ref{thm:inv} \eqref{it:sameW}.
\end{proof}

\begin{proof}[Proof of Theorem \ref{thm:inv} \eqref{it:sameHF}] 
Let $\phi: X \to X $ be a Hamiltonian isotopy so that $\phi(L_{t}), t \in [t_-,t_+]$ is disjoint from a neighborhood of the triple intersection.  Let $T_{t} := \phi(L_{t})$.  Then $T_{t}$ is a Maslov flow and has a singularity at $t = 0$, at a point disjoint from the triple intersection of
that of $L_{t}$.  By choosing the neighborhood $U$ sufficiently small, we may assume that $U$
is disjoint from $T_{t}$, and $\phi(U)$ is disjoint from $L_{t}$.  By Proposition \ref{prop:test}
and invariance of Floer cohomology under Hamiltonian perturbation
\begin{multline} HF(L_-,b_-) \cong HF((L_-,b_-),(L_-,b_-)) \cong HF((L_-,b_-),(T_-,b_-)) \\ \cong HF((L_+,b_+),(T_+,b_+)) \cong HF((L_+,b_+),(L_+,b_+)) \cong HF(L_+,b_+) .\end{multline}
\end{proof}

\def\cprime{$'$} \def\cprime{$'$} \def\cprime{$'$} \def\cprime{$'$}
\def\cprime{$'$} \def\cprime{$'$}
\def\polhk#1{\setbox0=\hbox{#1}{\ooalign{\hidewidtht 
      \lower1.5ex\hbox{`}\hidewidth\crcr\unhbox0}}} \def\cprime{$'$}
\def\cprime{$'$}


\begin{thebibliography}{10}

\bibitem{abouzaid:exotic} 
M.~Abouzaid. 
\newblock Framed bordism and Lagrangian embeddings of exotic spheres. 
\newblock Ann. of Math. (2) 175 (2012), no. 1, 71–185. 

\bibitem{akaho}
M.~Akaho and D.~Joyce. 
Immersed Lagrangian Floer Theory. 
{\em J. Differential Geom.}  86  (2010),  no. 3, 381--500. 


\bibitem{alston}
G.~Alston and E.~Bao.
\newblock Exact, graded, immersed Lagrangians and Floer theory.
\newblock \href{http://www.arxiv.org/abs/1407.3871}{arxiv:1407.3871}.

\bibitem{angenent:cusp} S.~B.~Angenent and
  J.~J.~L.~Vel\'azquez. Asymptotic shape of cusp singularities in
  curve shortening.  {\em Duke Math. J.}  77 (1995), no. 1, 71--110.

\bibitem{angenent} 
S.~B.~Angenent. 
\newblock Curve shortening and the topology of closed geodesics on
surfaces.  \newblock Annals of Mathematics, 162 (2005), 1187–1241.



\bibitem{auroux:complement} D.~Auroux, D.~Gayet, and J.-P.~Mohsen. 
  \newblock Symplectic hypersurfaces in the complement of an isotropic 
  submanifold.  \newblock {\em Math. Ann.}, 321(4):739--754, 2001. 


\bibitem{baker} 
C.~Baker.
\newblock The mean curvature flow of submanifolds of high codimension.
\newblock \href{http://www.arxiv.org/abs/1104.4409}{arxiv:1104.4409}.

\bibitem{bc:ql}
P.~Biran and O.~Cornea. 
\newblock Quantum structures for {L}agrangian submanifolds. 
\newblock 
\newblock \href{http://www.arxiv.org/abs/0708.4221}{arxiv:0708.4221}.

\bibitem{bt}
R.~Bott and L.~W.~Tu.
\newblock Differential forms in algebraic topology.
\newblock Graduate Texts in Mathematics, 82. Springer-Verlag, New York-Berlin,  1982. 

\bibitem{shorten}
S.~Carapetis. 
\newblock  Interactive demonstrative of curve-shortening flow. 
\newblock  Retrieved January 2018 at 
\href{http://a.carapetis.com/csf/}{http://a.carapetis.com/csf/}.


\bibitem{ce:st}
J.~Cerf. 
\newblock La stratification naturelle des espaces de fonctions 
  diff\'erentiables r\'eelles et le th\'eor\`eme de la pseudo-isotopie. 
\newblock {\em Inst. Hautes \'Etudes Sci. Publ. Math.}, (39):5--173,
1970.



\bibitem{charest:clust} F.~{Charest}.  \newblock {Source Spaces and 
  Perturbations for Cluster Complexes}.  \newblock 
\href{http://www.arxiv.org/abs/1212.2923}{arxiv:1212.2923}.

\bibitem{cw:traj}
F.~Charest and C.~Woodward. 
\newblock Floer trajectories and stabilizing divisors.
\newblock Fixed Point Theory. June 2017, Volume 19, Issue 2, pp 1165--1236.

\bibitem{cr:orb}
W.~Chen and Y.~Ruan.
\newblock Orbifold {G}romov-{W}itten theory.
\newblock In {\em Orbifolds in mathematics and physics ({M}adison, {WI},
  2001)}, volume 310 of {\em Contemp. Math.}, pages 25--85. Amer. Math. Soc.,
  Providence, RI, 2002.



\bibitem{cm:trans}
  K.~Cieliebak and K.~Mohnke.
\newblock Symplectic hypersurfaces and transversality in {G}romov-{W}itten 
  theory. 
\newblock {\em J. Symplectic Geom.}, 5(3):281--356, 2007.

\bibitem{flips}
F.~Charest and C.~Woodward. 
\newblock Floer theory and flips. 
\href{http://www.arxiv.org/abs/1508.01573}{arxiv:1508.01573}.



\bibitem{chenma} J.~Chen and J.~M.-S.~Ma.  
Geometry of Lagrangian self-shrinking tori and applications to the Piecewise Lagrangian Mean Curvature Flow
\href{http://www.arxiv.org/abs/1604.07563}{arxiv:1604.07563}.

\bibitem{cg:note}
K.~Cieliebak and E.~Goldstein.
\newblock A note on the mean curvature, Maslov class and symplectic area of Lagrangian immersions.
\newblock {\em  J. Symplectic Geom. }  2  (2004),  no. 2, 261--266.


%\bibitem{chenli}
%Chen, Jingyi; Li, Yuxiang.
%\newblock Bubble tree of a class of conformal mappings and
%applications to Willmore functional.   
%\newblock arXiv:1112.1818

\bibitem{colding} T.~H.~Colding and W.~.P.~Minicozzi II . Generic mean
  curvature flow I: generic singularities.  {\em Ann. of Math.} (2) 175
  (2012), no. 2, 755--833.

\bibitem{cl:clusters}
O.~Cornea and F. Lalonde. 
\newblock Cluster homology: An overview of the construction and results. 
\newblock {\em Electron. Res. Announc. Amer. Math. Soc.}, (12):1-12,
2006.

\bibitem{signs} P.~ Deligne and D.~Freed.  Sign manifesto.  In 
 {\em  Quantum Fields and Strings}, A course for mathematicians, 2 
  vols. Amer. Math. Soc. Providence 1999. (web version) 
  Eds. P. Deligne, P. Etingof, D.S. Freed, L. Jeffrey, D. Kazhdan,
  J. Morgan, D.R. Morrison, E. Witten. 

\bibitem{don:symp}
S.~K. Donaldson. 
\newblock Symplectic submanifolds and almost-complex geometry. 
\newblock {\em  J. Differential Geom.} 44 (1996), no. 4, 666--705. 

\bibitem{ees} T.~Ekholm, J.~Etnyre and M.~Sullivan. The contact
  homology of Legendrian submanifolds in $\bR^{2n+1}$.  \newblock {\em
    J. Differential Geom. } 71 (2005), no. 2, 177--305.
		
\bibitem{ees2} T.~Ekholm, J.~Etnyre, M.~Sullivan.
\newblock   Legendrian contact
  homology in $P \times \R$. 
\newblock {\em  Trans. Amer. Math. Soc.} 359 (2007),
  no. 7, 3301--3335.

\bibitem{ek:or} T.~Ekholm, J.~Etnyre, and M.~Sullivan.  \newblock 
  Orientations in {L}egendrian contact homology and exact {L}agrangian 
  immersions.  \newblock {\em Internat. J. Math.}, 16(5):453--532,
  2005.

\bibitem{fang}  K.-Y.~Fang.  \newblock Geometric constructions of mapping
  cones in the Fukaya category.   \newblock Ph.D. Thesis, Berkeley, 2018. 

\bibitem{floer:monopoles} A. Floer. Monopoles on asymptotically flat 
  manifolds. In: Hofer H., Taubes C.H., Weinstein A., Zehnder E. (eds) 
  The Floer Memorial Volume. Progress in Mathematics, vol 133. 
  Birkhäuser Basel.


\bibitem{floer} A.~Floer.  \newblock Morse theory for Lagrangian
  intersections. \newblock {\em J. Differential Geom.}  28 (1988), no. 3, 513-–547.

\bibitem{floer:unreg} A.~Floer. \newblock The unregularized gradient flow of the symplectic action.
 {\em Comm. Pure Appl. Math.}  41  (1988),  no. 6, 775--813.

\bibitem{totreal}
U.~Frauenfelder and K.~Zehmisch. 
\newblock Gromov compactness for holomorphic discs with totally real boundary conditions. 
\newblock {\em  J. Fixed Point Theory Appl. }  17  (2015),  no. 3, 521--540. 
		



\bibitem{fooo} K.~Fukaya, Y.-G.~Oh, H.~Ohta, and K.~Ono.  \newblock 
  {\em Lagrangian intersection {F}loer theory: anomaly and 
    obstruction.}, volume~46 of {\em AMS/IP Studies in Advanced 
    Mathematics}.  \newblock American Mathematical Society,
  Providence, RI, 2009.  
Orientation chapter at 
  \href{https://www.math.kyoto-u.ac.jp/~fukaya/bookchap9071113.pdf}{https://www.math.kyoto-u.ac.jp/$\sim$fukaya/bookchap9071113.pdf 
    version 2007}.  Surgery chapter at 
  \href{https://www.math.kyoto-u.ac.jp/~fukaya/Chapter10071117.pdf}{https://www.math.kyoto-u.ac.jp/~fukaya/Chapter10071117.pdf}. 



\bibitem{fooo:toric1} K.~Fukaya, Y.-G.~Oh, H.~Ohta, and K.~Ono. 
  \newblock Lagrangian {F}loer theory on compact toric manifolds. {I}. 
  \newblock {\em Duke Math. J.}, 151(1):23--174, 2010.

\bibitem{ganatra}
S.~Ganatra. 
\newblock 
Symplectic Cohomology and Duality for the 
Wrapped Fukaya Category. 
\newblock PhD Thesis, Massachusetts Institute of Technology, 2006. 

\bibitem{grayson} M.~Grayson.  \newblock The heat equation shrinks
  embedded plane curves to round points.  \newblock {\em Jour. of
    Diff. Geom.} 26 (2): 285–314, 1987.

\bibitem{grayson:eight} M.A.~Grayson.  \newblock The shape of a
  figure-eight under the curve shortening flow.  \newblock {\em
    Inventiones math} 96 (1) 177--180, 1989.

\bibitem{gromov}
M. Gromov.
\newblock Proceedings of the international congress of
mathematicians. 
\newblock Nice, 1970. Vol. 2, 221–225; Gauthier-Villars, Paris, 1971.

\bibitem{joyceslides} D.~Joyce.
  \href{https://people.maths.ox.ac.uk/joyce/UCLhandout2.pdf}{UCLhandout2.pdf}.
  Talk slides retrieved 1/17/2018.

\bibitem{joyce:conjectures}
D.~Joyce.
\newblock
Conjectures on Bridgeland stability for Fukaya categories of
Calabi-Yau manifolds, special Lagrangians, and Lagrangian mean
curvature flow.
{\em EMS Surv. Math. Sci. } 2 (2015), 1--62.
\newblock \href{http://www.arxiv.org/abs/1401.4949}{arxiv:1401.4949}.


\bibitem{jlt} D.~Joyce, Y.-I.~Lee, and M.~-P~Tsui.  \newblock
  Self-similar solutions and translating solitons for Lagrangian mean
  curvature flow.  \newblock {\em J. Differential Geom.}  Volume 84,
  Number 1 (2010), 127--161.




  \bibitem{yjl} Yi-Jen~Lee. Reidemeister torsion in Floer-Novikov theory
  and counting pseudo-holomorphic tori. I. {\em J. Symplectic Geom.} 3
  (2005), no. 2, 221--311.  Reidemeister torsion in Floer-Novikov
  theory and counting pseudo-holomorphic tori. II.  {\em J. Symplectic
    Geom.}  3 (2005), no. 3, 385--480.

\bibitem{lees} 
J.~A.~Lees.
\newblock On the classification of Lagrange immersions.
\newblock {\em Duke Math. J.}  43  (1976),  no. 2, 217--224.



\bibitem{lotay:coupled}
J.~D. {Lotay} and T.~{Pacini}. 
\newblock {Coupled flows, convexity and calibrations: Lagrangian and totally 
  real geometry}. 
\newblock \href{http://www.arxiv.org/abs/1404.4227}{arxiv:1404.4227}.

\bibitem{ms:jh}
D.~McDuff and D.~Salamon. 
\newblock {\em {$J$}-holomorphic curves and symplectic topology}, volume~52 of 
  {\em American Mathematical Society Colloquium Publications}. 
\newblock American Mathematical Society, Providence, RI, 2004.

\bibitem{mak:dehn} C.~Y.~Mak and W.~Wu.  \newblock Dehn twists exact 
  sequences through Lagrangian cobordism 
  \href{http://www.arxiv.org/abs/1509.08028}{arxiv:1509.08028}.

\bibitem{ainfty} S.~Ma'u, K.~Wehrheim, and C.T. Woodward.  \newblock
  ${A}_\infty$-functors for {L}agrangian correspondences.  \newblock
  {\em Selecta Math. } 24(3), p. 1913--2002, 2018.

\bibitem{moser} J.~Moser.  \newblock On the Volume Elements on a
  Manifold.  \newblock {\em Transactions of the American Mathematical
    Society}.  Vol. 120, No. 2 (Nov., 1965), pp. 286--294.

\bibitem{pacini} T.~Pacini.  \newblock Maslov, Chern-Weil and Mean
  Curvature.  {\em Jour. of Geom. and Physics} 135 (2019) 129--134
  \href{http://www.arxiv.org/abs/1711.07928}{arxiv:1711.07928}.

\bibitem{pw:surger}
J.~Palmer and C.~Woodward.
\newblock Invariance of immersed Floer cohomology under
Lagrangian surgery.
\href{http://www.arxiv.org/abs/1903.01943}{arxiv:1903.01943}. 

\bibitem{royden} {H.L.~Royden}. 
\newblock {\it Real Analysis}. 
\newblock Macmillan Publishing Company, New York, 1988. 


\bibitem{clean} 
F.~Schm\"aschke. 
\newblock 
Floer homology of Lagrangians in clean intersection. 
\newblock 
\href{http://www.arxiv.org/abs/1606.05327}{arxiv:1606.05327}. 


\bibitem{sikorav} J.-C.~Sikorav.  Some properties of holomorphic 
  curves in almost complex manifolds. Holomorphic curves in symplectic 
  geometry, 165–189, Progr. Math., 117, Birkhäuser, Basel, 1994. 


\bibitem{smoczyk} K.~Smoczyk.  Lagrangian mean curvature flow. 
Habilitation Thesis, Leipzig, 2001. 
\href{http://service.ifam.uni-hannover.de/~smoczyk/publications/preprint07.pdf}{http://service.ifam.uni-hannover.de/~smoczyk/publications/preprint07.pdf}. 


%\bibitem{sw:min}
%R. Schoen and J. Wolfson.
%J. Differential Geom.
%Volume 58, Number 1 (2001), 1-86.
%Minimizing Area Among Lagrangian Surfaces: The Mapping Problem.

\bibitem{sch:morse}
M.~Schwarz. 
\newblock {\em Morse homology}, vol. 111 of {\em Progress in Math}. 
\newblock Birkh\"auser Verlag, Basel, 1993.

\bibitem{se:gr}
P.~Seidel. 
\newblock Graded {L}agrangian submanifolds. 
\newblock {\em Bull. Soc. Math. France}, 128(1):103--149, 2000.

\bibitem{seidel:genustwo}
P.~Seidel. 
\newblock Homological mirror symmetry for the genus two curve. 
\newblock {\em J. Algebraic Geom.} 20:727--769, 2011.

\bibitem{sullivan} 
M.~G.~Sullivan.
\newblock $K$-theoretic invariants for Floer homology.
\newblock  {\em  Geom. Funct. Anal. } 12  (2002),  no. 4, 810--872.

\bibitem{thomasyau} R.P.~Thomas and S.~T.~Yau.  Special Lagrangians,
  stable bundles and mean curvature flow.  {\em Comm. Anal. Geom. } 10 
  (2002), no. 5, 1075--1113.
		
\bibitem{warsch} 
S.~Warschawski.
\newblock On the higher derivatives at the boundary in conformal mapping. 
\newblock {\em Trans. Amer. Math. Soc.}  38 (1935), no. 2, 310–340. 

\bibitem{orient}
K.~Wehrheim and C.T. Woodward. 
\newblock Orientations for pseudoholomorphic quilts. 
\newblock 
\href{http://www.arxiv.org/abs/1503.07803}{arXiv:1503.07803}. 


\bibitem{weibel}
C. Weibel.  Homological algebra. Cambridge University Press, 1995.

\bibitem{weinstein}
A.~Weinstein. 
\newblock Removing intersections of Lagrangian immersions. 
\newblock {\em Illinois J. Math.} 27 (1983), no. 3, 484–500. 

\bibitem{wiki}  File \begin{verbatim}
https://commons.wikimedia.org/wiki/File:Sphere_-_monochrome_simple.svg. \end{verbatim}
retrieved February 2018.


\end{thebibliography}
\end{document}